\documentclass[a4paper]{article}

\usepackage[english]{babel}
\usepackage{amsopn}
\usepackage{amsthm}
\usepackage{amsfonts}
\usepackage{amsmath}
\usepackage{amssymb}
\usepackage{graphicx}
\usepackage{algorithmic}
\usepackage[ruled,vlined]{algorithm2e}
\usepackage{enumitem} 
\usepackage{lscape}
\usepackage{rotating} 
\usepackage{caption}
\usepackage{subcaption}

\usepackage[dvipsnames]{xcolor}
\usepackage{url}
\usepackage{caption}
\usepackage{geometry} 

\newtheorem{theorem}{Theorem}[section]
\newtheorem{definition}[theorem]{Definition}
\newtheorem{lemma}[theorem]{Lemma}
\newtheorem{proposition}[theorem]{Proposition}
\newtheorem{corollary}[theorem]{Corollary}

\theoremstyle{definition}
\newtheorem{example}{Example}

\usepackage[hidelinks]{hyperref}

\definecolor{fuchsia}{rgb}{1.0, 0.0, 1.0} 

\begin{document}
\title{Completely Positive Factorization by a Riemannian Smoothing Method
}
\author{
	ZHIJIAN LAI\thanks{
		Graduate School of Systems and Information Engineering, University of Tsukuba, Tsukuba, Ibaraki 305-8573, Japan. 
		E-mail:\href{mailto:s2130117@s.tsukuba.ac.jp}{s2130117@s.tsukuba.ac.jp}
	}
	\and AKIKO YOSHISE\thanks{
		Corresponding author. 
		Faculty of Engineering, Information and Systems, University of Tsukuba, Tsukuba, Ibaraki 305-8573, Japan. 
		E-mail:\href{mailto:yoshise@sk.tsukuba.ac.jp}{yoshise@sk.tsukuba.ac.jp}
	}
}
\date{}
\maketitle

\begin{abstract}
Copositive optimization is a special case of convex conic programming, and it consists of optimizing a linear function over the cone of all completely positive matrices under  linear constraints. 
Copositive optimization provides powerful relaxations of NP-hard quadratic problems or combinatorial problems, but there are still many open problems regarding copositive or completely positive matrices.
In this paper, we focus on one such problem; finding a completely positive (CP) factorization for a given completely positive matrix.
We treat it as a nonsmooth Riemannian optimization problem, i.e., a minimization problem of a nonsmooth function over a Riemannian manifold.
To solve this problem, we present a general smoothing framework for solving nonsmooth Riemannian optimization problems and show convergence to a stationary point of the original problem. 
An advantage is that we can implement it quickly with minimal effort by directly using the existing standard smooth Riemannian solvers, such as Manopt.
Numerical experiments show the efficiency of our method especially for large-scale CP factorizations.


\smallskip
\textbf{AMS:}
15A23, 15B48, 90C48, 90C59
\end{abstract}

\section{Introduction}\label{sec1}
The space of $n \times n$ real symmetric matrices $\mathcal{S}_{n}$ is endowed with the trace inner product $\langle A, B\rangle:=\operatorname{trace}(A B)$.
A matrix $ A \in \mathcal{S}_{n}$ is called \emph{completely positive} if for some $r \in \mathbb{N}$ 
there exists an entrywise nonnegative matrix $B \in \mathbb{R}^{n \times r}$ such that $ A=B B^{\top}$, and
we call $B$ a \emph{CP factorization} of $A$.
We define $\mathcal{CP}_{n}$ as the set of $n \times n$ completely positive matrices, equivalently characterized as
\begin{equation*}
	\mathcal{CP}_{n}
	:=\{B B^{\top} \in \mathcal{S}_{n} \mid B \text { is a nonnegative matrix }\}
	=\operatorname{conv}\{x x^{\top} \mid x \in \mathbb{R}_{+}^{n}\},
\end{equation*}
where $\operatorname{conv}(S)$ denotes the convex hull of a given set $S$.
We denote the set of $n \times n$ copositive matrices by
$
\mathcal{COP}_{n}:=\{A \in \mathcal{S}_{n} \mid x^{\top} A x \geq 0 \text { for all } x \in \mathbb{R}_{+}^{n}\}.
$
{It is known that $ \mathcal{COP}_n $ and $\mathcal{CP}_{n}$ are duals of each other under the trace inner product; moreover, both $\mathcal{CP}_{n}$ and $\mathcal{COP}_{n}$ are proper convex cones \cite[Section 2.2]{berman2003completely}.} 
For any positive integer $n$, we have the following inclusion relationship among other important cones in conic optimization:
\begin{equation*}
	\mathcal{CP}_{n} \subseteq
	\mathcal{S}_{n}^{+} \cap \mathcal{N}_{n} \subseteq
	\mathcal{S}_{n}^{+} \subseteq 
	\mathcal{S}_{n}^{+}+\mathcal{N}_{n} \subseteq
	\mathcal{COP}_{n},
\end{equation*}
where $ \mathcal{S}_{n}^{+} $ is the cone of $n \times n$ symmetric positive semidefinite matrices and $\mathcal{N}_{n}$ is the cone of $n \times n$ symmetric nonnegative matrices.
See the monograph \cite{berman2003completely} for a comprehensive description of $\mathcal{CP}_{n}$ and $\mathcal{COP}_{n}$.

Conic optimization is a subfield of convex optimization that studies minimization of linear functions over proper cones.
Here, if the proper cone is $\mathcal{CP}_n$ or its dual cone $\mathcal{COP}_n$, we call the conic optimization problem a \emph{copositive programming} problem.
Copositive programming is closely related to many nonconvex, NP-hard quadratic and combinatorial optimizations \cite{dur2021conic}. 
For example, consider the so-called standard quadratic optimization problem, 
\begin{equation}\label{sqo}
	\min \{x^{\top} M x \mid  \textbf{e}^{\top} x=1,
	x \in \mathbb{R}_{+}^{n} \},
\end{equation}
where $M \in \mathcal{S}_{n}$ is possibly not positive semidefinite and $\textbf{e}$ is the all-ones vector.
Bomze et al. \cite{bomze2000copositive} showed that the following completely positive reformulation,
\begin{equation*}
	\min \{\langle M, X\rangle \mid \langle E, X\rangle=1,
	X \in \mathcal{C} \mathcal{P}_{n} \},
\end{equation*}
where $E$ is the all-ones matrix, is equivalent to (\ref{sqo}).
Burer \cite{burer2009copositive} reported a more general result, where any quadratic problem with binary and continuous variables can be rewritten as a linear program over $ \mathcal{CP}_n $. 
As an application to combinatorial problems, consider the problem of computing the independence number $\alpha (G)$ of a graph $G$ with $n$ nodes.
De Klerk and Pasechnik \cite{de2002approximation} showed that
\begin{equation*}
	\alpha (G)=\max \{\langle E, X\rangle \mid\langle A+I, X\rangle=1, X \in \mathcal{C} \mathcal{P}_{n}\},
\end{equation*}
where $A$ is the adjacency matrix of $G$.
For surveys on applications of copositive programming, see \cite{bomze2012copositive,bomze2012think,burer2015gentle,dur2010copositive,dur2021conic}.

The difficulty of the above problems lies entirely in the completely positive conic constraint. Note that because neither $ \mathcal{COP}_{n} $ nor $ \mathcal{CP}_{n} $ is self-dual, the primal-dual interior point method for conic optimization does not work as is. 
Besides this, there are many open problems related to completely positive cones. 
One is checking membership in $\mathcal{CP}_{n}$, which was shown to be NP-hard by \cite{dickinson2014computational}. Computing or estimating the cp-rank, as defined later in (\ref{cprank}), is also an open problem. We refer the reader to \cite{berman2015open,dur2010copositive} for a detailed discussion of those unresolved issues.

In this paper, we focus on finding a CP factorization for a given $A \in \mathcal{CP}_{n}$, i.e., the \emph{CP factorization problem}:
\begin{equation}\tag{CPfact}\label{CPfact}
	\text { Find } B \in \mathbb{R}^{n \times r} \text { s.t. } A=B B^{\top} \text { and } B \geq 0,
\end{equation}
which seems to be closely related to the membership problem $A \in \mathcal{CP}_n$.
Sometimes, a matrix is shown to be completely positive through duality, or rather, $\langle A, X\rangle \geq 0$ for all $X \in \mathcal{COP}_n$, but in this case, a CP factorization will not necessarily be obtained.

\subsection{Related work on CP factorization}\label{subs:rw}

Various methods of solving CP factorization problems have been studied.
Jarre and Schmallowsky \cite{jarre2009computation} stated a criterion for complete positivity, based on the augmented primal dual method to solve a particular second-order cone problem.
Dickinson and D{\"u}r \cite{dickinson2012linear} 
dealt with complete positivity of matrices that possess a specific sparsity pattern and proposed a method for finding CP factorizations of these special matrices that can be performed in linear time.
Nie \cite{nie2014mathcal} formulated the CP factorization problem as an $\mathcal{A}$-truncated $K$-moment problem, for which the author developed an algorithm that solves a series of semidefinite optimization problems.
Sponsel and D{\"u}r \cite{sponsel2014factorization} considered the problem of projecting a matrix onto $\mathcal{CP}_{n}$ and $\mathcal{COP}_{n}$ by using polyhedral approximations of these cones.
With the help of these projections, they devised a method to compute a CP factorization for any matrix in the interior of $\mathcal{CP}_{n}$. 
Bomze \cite{bomze2018building} showed how to construct a CP factorization of an $n \times n$ matrix based on a given CP factorization of an  $(n-1) \times (n-1)$ principal submatrix.
Dutour Sikiri\'{c} et al. \cite{sikiric2020simplex} developed a simplex-like method for a rational CP factorization that works if the input matrix allows a rational CP factorization.

In 2020, Groetzner and D{\"u}r \cite{groetzner2020factorization} applied the alternating projection method to the CP factorization problem by posing it as an equivalent feasibility problem (see (\ref{FeasCP})).
Shortly afterwards, Chen et al. \cite{chen2020difference} reformulated the split feasibility problem as a difference-of-convex optimization problem and 
solved (\ref{FeasCP}) as a specific application.
In fact, we will solve this equivalent feasibility problem (\ref{FeasCP}) by other means in this paper.		
In 2021, Bo{\c{t}} and Nguyen \cite{boct2021factorization} proposed a projected gradient method with relaxation and inertia parameters for the CP factorization problem, aimed at solving
\begin{equation}\label{bot}
	\min _{X}\{\|A-X X^{\top}\|^{2} \mid X \in \mathbb{R}_{+}^{n \times r} \cap \mathcal{B}({0}, \sqrt{\operatorname{trace}(A)})\},
\end{equation}
where $\mathcal{B}(0, \varepsilon):=\{X \in \mathbb{R}^{n \times r} \mid \|X\| \leq \varepsilon\}$ is the closed ball centered at 0.
The authors argued that its optimal value is zero if and only if $A \in \mathcal{CP}_{n}$.

\subsection{Our contributions and organization of the paper}
Inspired by the idea of Groetzner and D{\"u}r \cite{groetzner2020factorization}, wherein (\ref{CPfact}) is shown to be equivalent to a feasibility problem called (\ref{FeasCP}), we treat the problem (\ref{FeasCP}) as a nonsmooth Riemannian optimization problem and solve it through a general Riemannian smoothing method.
Our contributions are summarized as follows:

1. Although it is not explicitly stated in \cite{groetzner2020factorization}, (\ref{FeasCP}) is actually a Riemannian optimization formulation. We propose a new Riemannian optimization technique and apply it to the problem.

2. In particular, we present a general framework of Riemannian smoothing for the nonsmooth Riemannian optimization problem and show convergence to a stationary point of the original problem.

3. We apply the general framework of Riemannian smoothing to CP factorization. Numerical experiments clarify that our method is competitive with other efficient CP factorization methods, especially for large-scale matrices.

In Section \ref{sec:preliminaries}, we review the process to reconstruct (\ref{CPfact}) into another feasibility problem; in particular, we take a different approach to this problem from those in other studies. 
In Section \ref{sec:sm},
we describe the general framework of smoothing methods for Riemannian optimization.
To apply it to the CP factorization problem, we employ a smoothing function named LogSumExp. Section \ref{sec:ne} is a collection of numerical experiments for CP factorization.
As a meaningful supplement, in Section \ref{sec:comparison}, we conduct further experiments (FSV problem and robust low-rank matrix completion) to explore the numerical performance of various sub-algorithms and smoothing functions on different applications.

\section{Preliminaries}\label{sec:preliminaries}

\subsection{cp-rank and cp-plus-rank}\label{subs:cprank}
First, let us recall some basic properties of completely positive matrices. 
Generally, many CP factorizations of a given $A$ may exist, and they may vary in their numbers of columns. This gives rise to the following definitions:
the cp-rank of $A \in \mathcal{S}_{n}$, denoted by $\operatorname{cp}(A)$, is defined as
\begin{equation}\label{cprank}
	\operatorname{cp}(A) 
	:= \min \{r \in \mathbb{N} \mid A=B B^{\top}, B \in \mathbb{R}^{n \times r}, B \geq 0 \},	
\end{equation}
where $ \operatorname{cp}(A) = \infty$ if $A \notin \mathcal{CP}_{n}.$ Similarly, we can define the cp-plus-rank as
\begin{equation*}
	\operatorname{cp}^{+}(A):= \min \{r \in \mathbb{N} \mid 
	A=B B^{\top}, B \in \mathbb{R}^{n \times r}, B > 0 \}.
\end{equation*}
Immediately, for all $A \in \mathcal{S}_{n}$, we have 
\begin{equation}\label{rank}
	\operatorname{rank}(A) \leq \mathrm{cp}(A) \leq \mathrm{cp}^{+}(A).	
\end{equation}
Every CP factorization $ B $ of $A$ is of the same rank as $A$ since
$\operatorname{rank}(X X^{\top})=\operatorname{rank}(X)$
holds for any matrix $X$.
The first inequality of (\ref{rank}) comes from the fact that for any CP factorization $ B $,
\begin{equation*}
	\operatorname{rank}(A) = \operatorname{rank}(B) \leq \text{ the number of columns of $ B $.}
\end{equation*}
The second is trivial by definition.

Note that computing or estimating the cp-rank of any given $A \in \mathcal{CP}_{n}$ is still an open problem.
The following result gives a tight upper bound of the cp-rank for $A \in \mathcal{CP}_n$ in terms of the order $ n $.

\begin{theorem}[Bomze, Dickinson, and Still {\cite[Theorem 4.1]{bomze2015structure}}]
	For all $A \in \mathcal{C} \mathcal{P}_{n}$, we have
	\begin{equation}\label{cpn}
		\operatorname{cp}(A) \leq \mathrm{cp}_{n}:=\left\{\begin{array}{ll}
			n & \text { for } n \in\{2,3,4\} \\
			\frac{1}{2} n(n+1)-4 & \text { for } n \geq 5.
		\end{array}\right.
	\end{equation}
\end{theorem}

The following result is useful for distinguishing completely positive matrices in either the interior or on the boundary of $\mathcal{CP}_n$.
\begin{theorem}[Dickinson {\cite[Theorem 3.8]{dickinson2010improved}}]\label{thm:int_cp}
	We have
	\begin{equation*}
		\begin{aligned}
			\operatorname{int}(\mathcal{C} \mathcal{P}_{n}) 
			&=\{A \in \mathcal{S}_{n}  \mid \operatorname{rank}(A)=n, \operatorname{cp}^{+}(A)<\infty \} \\ 
			&=\{ A \in \mathcal{S}_{n}  \mid \operatorname{rank}(A)=n, A=BB^{\top}, B \in \mathbb{R}^{n \times r}, B \geq 0, \\
			& \qquad b_{j}>0 \text{ for at least one column $b_{j} $ of $ B $} \}.
		\end{aligned}
	\end{equation*}
\end{theorem}

\subsection{CP factorization as a feasibility problem}\label{subs:fea}

Groetzner and D{\"u}r \cite{groetzner2020factorization} reformulated the CP factorization problem as an equivalent feasibility problem containing an orthogonality constraint.

Given $ A \in \mathcal{CP}_{n}$, we can easily get another CP factorization $\widehat{B}$ with $ r^{\prime} $ columns for every integer $r^{\prime} \geq r$, if we also have a CP factorization $B$ with $ r $ columns.
The simplest way to construct such an $n \times r^{\prime}$ matrix $\widehat{B}$ is to append
$k:=r^{\prime}-r$ zero columns to $B,$ i.e., $\widehat{B}:=\left[B, 0_{n \times k}\right] \geq 0.$
Another way is called \emph{column replication}, i.e.,
\begin{equation}\label{col_rep}
	\widehat{B}:=[b_{1}, \ldots, b_{n-1}, \underbrace{\frac{1}{\sqrt{m}} b_{n}, \ldots, \frac{1}{\sqrt{m}} b_{n}}_{m:=r^{\prime}-n+1 \text { columns }}],
\end{equation}
where $ b_{i} $ denotes the $ i $-th column of $ B $. It is easy to see that $\widehat{B} \widehat{B}^{\top}=B B^{\top}=A.$ The next lemma is easily derived from the previous discussion, and it implies that there always exists an $ n \times \mathrm{cp}_{n} $ CP factorization for any $A \in \mathcal{CP}_{n} $. Recall that definition of $\mathrm{cp}_{n} $ is given in (\ref{cpn}).

\begin{lemma}\label{lem:r_cpr}
	Suppose that $ A \in \mathcal{S}_{n}$, $ r \in \mathbb{N} $. Then $r \geq \operatorname{cp}(A)$ if and only if $ A $ has a CP factorization $ B $ with $ r $ columns.
\end{lemma}

Let $\mathcal{O}(r)$ denote the orthogonal group of order $ r $, i.e., the set of $r \times r $ orthogonal matrices. The following lemma is essential to our study. (Note that many authors have proved the existence of such an orthogonal matrix $ X $ (see, e.g., \cite[Lemma 2.1]{burer2005local} and \cite[Lemma 2.6]{groetzner2020factorization}).

\begin{lemma}\label{lem:orth}
	Let $B, C \in \mathbb{R}^{n \times r}$. $ B B^{\top}=C C^{\top} $ if and only if there exists $ X \in \mathcal{O}(r) \text{ with }B X=C$. 
\end{lemma}

The next proposition puts the previous two lemmas together.

\begin{proposition}\label{prop:feas}
	Let $ A \in \mathcal{CP}_n, r \geq \operatorname{cp}(A), A = \bar{B} \bar{B}^{\top} $, where $\bar{B} \in \mathbb{R}^{n \times r}$ may possibly be not nonnegative. Then there exists an orthogonal matrix $ X \in \mathcal{O}(r) $ such that $ \bar{B} X \geq 0 $ and $A=(\bar{B}X)(\bar{B}X)^{\top}. $
\end{proposition}

This proposition tells us that one can find an orthogonal matrix $ X $ which can turn a ``bad'' factorization $ \bar{B} $ into a ``good'' factorization $\bar{B}X$.
{Let $ r \geq \operatorname{cp}(A)$ and $\bar{B} \in \mathbb{R}^{n \times r}$ be an arbitrary (possibly not nonnegative) initial factorization $A=\bar{B} \bar{B}^{\top}$.
The task of finding a CP factorization of $A$ can then be formulated as the following feasibility problem,
\begin{equation}\label{FeasCP}\tag{FeasCP}
	\text { Find } X \text { s.t. } \bar{B} X \geq 0 \text { and } X \in \mathcal{O}(r).
\end{equation}}

We should notice that the condition $ r \geq \operatorname{cp}(A) $ is necessary; otherwise, (\ref{FeasCP}) has no solution even if $A \in \mathcal{CP}_{n}.$
Regardless of the exact 
$\operatorname{cp}(A)$ which is often unknown, one can use $ r=\mathrm{cp}_{n} $ in (\ref{cpn}). 
Note that finding an initial matrix $ \bar{B} $ is not difficult. 
Since a completely positive matrix is necessarily positive semidefinite, one can use Cholesky decomposition or spectral decomposition and then extend it to $r$ columns by using (\ref{col_rep}). The following corollary shows that the feasibility of (\ref{FeasCP}) is precisely a criterion for complete positivity.
\begin{corollary}\label{cor:fea_1}	
	Set $ r \geq \operatorname{cp}(A)$, $\bar{B} \in \mathbb{R}^{n \times r}$ an arbitrary initial factorization of $A$.
	Then
	$A \in \mathcal{CP}_{n}$ if and only if $ (\ref{FeasCP})$ is feasible.
	In this case, for any feasible solution $ X $, $\bar{B}X$ is a CP factorization of $ A $.
\end{corollary}
In this study, solving (\ref{FeasCP}) is the key to finding a CP factorization, but it is still a hard problem because  $\mathcal{O}(r)$ is nonconvex.

\subsection{Approaches to solving (\ref{FeasCP})}\label{subs:approaches}

Groetzner and D{\"u}r \cite{groetzner2020factorization} applied the so-called alternating projections method to (\ref{FeasCP}).
They defined the polyhedral cone,
$
\mathcal{P}:=\{X \in \mathbb{R}^{r \times r} : \bar{B} X \geq 0\},
$
and rewrote (\ref{FeasCP}) as
\begin{equation*}
	\text { Find } X \text { s.t. } X \in \mathcal{P} \cap \mathcal{O}(r).
\end{equation*}
The alternating projections method is as follows: choose a starting point $X_{0} \in \mathcal{O}(r)$; then compute $P_{0}=\operatorname{proj}_{\mathcal{P}}(X_{0})$ and $X_{1}=\operatorname{proj}_{\mathcal{O}(r)}(P_{0})$, and iterate this process.
Computing the projection onto $\mathcal{P}$ amounts to solving a second-order cone problem (SOCP), while computing the projection onto $\mathcal{O}(r)$ amounts to
a singular value decomposition.
Note that we need to solve an SOCP alternately at every iteration, which is still expensive in practice. A modified version without convergence involves calculating an approximation of $\operatorname{proj}_{\mathcal{P}}(X_{k})$ by using the Moore-Penrose inverse of $\bar{B}$; for details, see \cite[Algorithm 2]{groetzner2020factorization}.

Our way is to use the optimization form. Here, we denote by $ \max (\cdot) $ (resp. $\min (\cdot)$) the \emph{max-function} (resp. \emph{min-function})) that selects the largest (resp. smallest) entry of a vector or matrix. Notice that $-\min (\cdot)=\max (-(\cdot)).$ We associate (\ref{FeasCP}) with the following optimization problem:
\begin{equation*}
	\max _{X \in \mathcal{O}(r)} \{\min \:(\bar{B}X)\}.
\end{equation*}
For consistency of notation, we turn the maximization problem into a minimization problem:
\begin{equation}\label{OptCP}\tag{OptCP}
	\min _{X \in \mathcal{O}(r)} \{\max \:(-\bar{B}X)\}.
\end{equation}
The feasible set $ \mathcal{O}(r) $ is known to be compact \cite[Observation 2.1.7]{horn2012matrix}. In accordance with the extreme value theorem \cite[Theorem 4.16]{rudin1964principles}, (\ref{OptCP}) attains the global minimum, say $t$. Summarizing these observations together with Corollary \ref{cor:fea_1} yields the following proposition.

\begin{proposition}\label{prop:cp}
	Set $ r \geq \operatorname{cp}(A)$, and let $\bar{B} \in \mathbb{R}^{n \times r}$ be an arbitrary initial factorization of $A$. 
	Then the following statements are equivalent:
	\begin{enumerate}
		\item $A \in \mathcal{CP}_{n}$.
		
		\item (\ref{FeasCP}) is feasible.
		
		\item In (\ref{OptCP}), there exists a feasible solution $ X $ such that $ \max (-\bar{B}X) \leq 0$; alternatively, $ \min (\bar{B}X) \geq 0$.
		
		\item In (\ref{OptCP}), the global minimum $t \leq 0$. 
	\end{enumerate}
\end{proposition}

\section{Riemannian smoothing method}\label{sec:sm}

The problem of minimizing a real-valued function over a Riemannian manifold $\mathcal{M}$, which is called Riemannian optimization, has been actively studied during the last few decades. In particular, the Stiefel manifold,
\begin{equation*}
	\operatorname{St}(n, p)=\{X \in \mathbb{R}^{n \times p} \mid X^{\top} X=I\},
\end{equation*}
(when $n=p$, it reduces to the orthogonal group) is an important case and is our main interest here. 
We treat the CP factorization problem, i.e., (\ref{OptCP}) as a problem of minimizing a nonsmooth function over a Riemannian manifold, for which variants of subgradient methods \cite{borckmans2014riemannian}, proximal gradient methods \cite{de2016new}, and the alternating direction method of multipliers (ADMM) \cite{kovnatsky2016madmm} have been studied.

Smoothing methods \cite{chen2012smoothing}, which use a parameterized smoothing function to approximate the objective function, are effective on a class of nonsmooth optimizations in Euclidean space. Recently, Zhang, Chen and Ma \cite{zhang2021riemannian} extended a smoothing steepest descent method to the case of Riemannian submanifolds in $\mathbb{R}^{n}$. This is not the first time that smoothing methods have been studied on manifolds. Liu and Boumal \cite{liu2020simple} extended the augmented Lagrangian method and exact penalty method to the Riemannian case. The latter leads to a nonsmooth Riemannian optimization problem to which they applied smoothing techniques. 
Cambier and Absil \cite{cambier2016robust} dealt with the problem of robust low-rank matrix completion by solving a Riemannian optimization problem, wherein they applied a smoothing conjugate gradient method. 

In this section, we propose a general Riemannian smoothing method and apply it to the CP factorization problem.

\subsection{Notation and terminology of Riemannian optimization}
Let us briefly review some concepts in Riemannian optimization, following the notation of \cite{boumal2022intromanifolds}. Throughout this paper, $\mathcal{M}$ will refer to a complete Riemannian submanifold of Euclidean space $\mathbb{R}^{n}$. 
Thus, $\mathcal{M}$ is endowed with a Riemannian metric induced from the Euclidean inner product, i.e., $\langle \xi, \eta \rangle_{x}:=\xi^{\top} \eta $ for any $\xi, \eta \in {\rm T}_{x} \mathcal{M}$, where ${\rm T}_{x} \mathcal{M} \subseteq \mathbb{R}^{n}$ is the tangent space to $\mathcal{M}$ at $x$.
The Riemannian metric induces the usual Euclidean norm $\|\xi\|_{x}:=\|\xi\|=\sqrt{\langle\xi, \xi\rangle_{x}}$ for $\xi \in {\rm T}_{x} \mathcal{M}$.
The tangent bundle ${\rm T} \mathcal{M} :=\bigsqcup_{x \in \mathcal{M}} {\rm T}_{x} \mathcal{M}$ is a disjoint union of the tangent spaces of $\mathcal{M}$. Let $f: \mathcal{M} \rightarrow \mathbb{R}$ be a smooth function on $\mathcal{M}$. The Riemannian gradient of $f$ is a vector field $\operatorname{grad} f$ on $\mathcal{M}$ that is uniquely defined by the identities: for all $ (x, v) \in \rm{T} \mathcal{M} $,
\begin{equation*}
	\mathrm{D}f(x)[v]=\langle v, \operatorname{grad} f(x) \rangle_{x}
\end{equation*}
where $\mathrm{D}f(x):{\rm T}_{x} \mathcal{M} \to {\rm T}_{f(x)} \mathbb{R} \cong \mathbb{R}$ is the differential of $f$ at $x \in \mathcal{M}$. 
Since $ \mathcal{M} $ is an embedded submanifold of $\mathbb{R}^{n}$, we have a simpler statement for $f$ that is also well defined on the whole $\mathbb{R}^{n}$:
\begin{equation*}
	\operatorname{grad} f(x)=\operatorname{Proj}_{x}(\nabla f(x)),
\end{equation*}
where $\nabla f(x)$ is the usual gradient in $\mathbb{R}^{n}$ and $\operatorname{Proj}_{x}$ denotes the orthogonal projector from $\mathbb{R}^{n}$ to ${\rm T}_{x} \mathcal{M}$. For a subset $D \subseteq \mathbb{R}^{n}$, the function $h \in C^{1}(D)$ is smooth, i.e., continuously differentiable on $D$. Given a point $x \in \mathbb{R}^{n}$ and $\delta>0$, $ \mathcal{B}(x, \delta)$ denotes a closed ball of radius $\delta$ centered at $x$. $ \mathbb{R}_{++} $ denotes the set of positive real numbers. We use subscript notation $x_{i}$ to select the $i$th entry of a vector and superscript notation $x^{k}$ to designate an element in a sequence $\{x^{k}\}$.

\subsection{Ingredients}

Now let us consider the nonsmooth Riemannian optimization problem (NROP):
\begin{equation}\label{NROP}\tag{NROP}
	\min _{x \in \mathcal{M}} f(x),
\end{equation}
where $\mathcal{M} \subseteq \mathbb{R}^{n}$ and $f: \mathbb{R}^{n} \rightarrow \mathbb{R}$ is a \emph{proper lower semi-continuous} function (maybe nonsmooth or even non-Lipschitzian) on $\mathbb{R}^{n}$. 
For convenience, the term smooth Riemannian optimization problem (SROP) refers to (\ref{NROP}) when $f(\cdot)$ is continuously differentiable on $\mathbb{R}^{n}$. To avoid confusion in this case, we use $ g $ instead of $ f $,
\begin{equation}\label{SROP}\tag{SROP}
	\min _{x \in \mathcal{M}} g(x).
\end{equation}
Throughout this subsection, we will refer to many of the concepts in \cite{zhang2021riemannian}.

First, let us review the usual concepts and properties related to generalized subdifferentials in $\mathbb{R}^{n}$. For a proper lower semi-continuous function $f: \mathbb{R}^{n} \rightarrow \mathbb{R}$, the \emph{Fréchet subdifferential} and the \emph{limiting subdifferential of $f$ at $x \in \mathbb{R}^{n}$} are defined as
\begin{equation*}
	\begin{aligned}
		\hat{\partial} f(x):=\{\nabla h(x) \mid \exists \delta&>0 \text { such that } h \in C^{1}(\mathcal{B}(x, \delta)) \text { and } \\
		&f-h \text { attains a local minimum at } x \text { on } \mathbb{R}^{n}\},
	\end{aligned}
\end{equation*}
\begin{equation*}
	\partial f(x):=\{\lim _{\ell \rightarrow \infty} v^{\ell}\mid v^{\ell} \in \hat{\partial} f\left(x^{\ell}\right),\left(x^{\ell}, f\left(x^{\ell}\right)\right) \rightarrow(x, f(x))\}.
\end{equation*}
The definition of $\hat{\partial} f(x)$ above is not the standard one: the standard definition follows \cite[8.3 Definition]{rockafellar2009variational}. 
But these definitions are equivalent by \cite[8.5 Proposition]{rockafellar2009variational}. 
For locally Lipschitz functions, the \emph{Clarke subdifferential at $x \in \mathbb{R}^{n}$}, $\partial^{\circ} f(x)$, is the convex hull of the limiting subdifferential. 
Their relationship is as follows:
\begin{equation*}
	\hat{\partial} f(x) \subseteq \partial f(x) \subseteq \partial^{\circ} f(x).
\end{equation*}
Notice that if $f$ is convex, $\partial f(x)$ and $ \partial^{\circ} f(x) $ coincide with the classical subdifferential in convex analysis \cite[8.12 Proposition]{rockafellar2009variational}.

\begin{example}[Bagirov, Karmitsa, and M{\"a}kel{\"a} {\cite[Theorem 3.23]{bagirov2014introduction}}]
	\label{lim_sudf_max}
	From a result on the pointwise max-function in convex analysis, we have 
	\begin{equation*}
		\partial \max(x)=\operatorname{conv}\{e_i \mid i \in \mathcal{I}(x) \},
	\end{equation*}
	where $e_i$'s are the standard bases of $\mathbb{R}^{n} $ and $\mathcal{I}(x)=\{i \mid x_{i}=\max (x)\}.$ 
\end{example}

Next, we extend our discussion to include generalized subdifferentials of a nonsmooth function on submanifolds $\mathcal{M}$.
The \emph{Riemannian Fréchet subdifferential} and the \emph{Riemannian limiting subdifferential of $f$ at $x \in \mathcal{M}$} (see, e.g., \cite[Definition 3.1]{zhang2021riemannian}) are defined as
\begin{equation*}
	\begin{aligned}
		\hat{\partial}_{\mathcal{R}} f(x):=\{\operatorname{grad} h(x) \mid \exists \delta&>0 \text { such that } h \in C^{1}(\mathcal{B}(x, \delta)) \text { and } \\
		&f-h \text { attains a local minimum at } x \text { on } \mathcal{M}\},
	\end{aligned}
\end{equation*}
\begin{equation*}
	\partial_{\mathcal{R}} f(x):=\{\lim _{\ell \rightarrow \infty} v^{\ell}\mid v^{\ell} \in \hat{\partial}_{\mathcal{R}} f\left(x^{\ell}\right),\left(x^{\ell}, f\left(x^{\ell}\right)\right) \rightarrow(x, f(x))\}.
\end{equation*}
If $\mathcal{M}=\mathbb{R}^{n}$, the above definitions coincide with the usual Fréchet and limiting subdifferentials in $\mathbb{R}^{n}$. Moreover, it follows directly that, for all $x \in \mathcal{M}$, one has $\hat{\partial}_{\mathcal{R}} f(x) \subseteq \partial_{\mathcal{R}} f(x)$. According to \cite[Proposition 3.2]{zhang2021riemannian}, if ${x}$ is a local minimizer of $f$ on $\mathcal{M}$, then $0 \in \hat{\partial}_{\mathcal{R}} f({x})$. Thus, {we call a point $x \in \mathcal{M}$ a \emph{Riemannian limiting stationary point of (\ref{NROP})} if 
\begin{equation}\label{Rlsp}
	0 \in \partial_{\mathcal{R}} f(x).
\end{equation}}
In this paper, we will treat it as a necessary condition for a local solution of (\ref{NROP}) to exist.

The smoothing function is the most important tool of the smoothing method.
\begin{definition}[Zhang and Chen {\cite[Definition 3.1]{zhang2020smoothing}}]
	\label{defn:sm}
	A function $\tilde{f}(\cdot, \cdot): \mathbb{R}^{n} \times \mathbb{R}_{++} \rightarrow \mathbb{R}$ is called a \emph{smoothing function} of $f: \mathbb{R}^{n} \rightarrow \mathbb{R}$, if $\tilde{f}(\cdot, \mu)$ is continuously differentiable in $\mathbb{R}^{n}$ for any $\mu >0$,
	\begin{equation*}
		\lim _{z \rightarrow x, \mu \downarrow 0} \tilde{f}(z, \mu)=f(x)
	\end{equation*}
	and there exist a constant $\kappa>0$ and a function $\omega: \mathbb{R}_{++} \rightarrow \mathbb{R}_{++}$ such that
	\begin{equation}\label{defn:kappa}
		\lvert \tilde{f}(x, \mu)-f(x)\rvert \leq \kappa \omega(\mu) \quad \text { with } \quad \lim _{\mu \downarrow 0} \omega(\mu)=0.
	\end{equation}
\end{definition}

\begin{example}[Chen, Wets, and Zhang {\cite[Lemma 4.4]{chen2012stochastic}}]
	\label{exmp:lse1}
	The LogSumExp function, $\operatorname{lse}(x,\mu) : \mathbb{R}^n \times \mathbb{R}_{++} \rightarrow \mathbb{R} $, given by
	\begin{equation*}
		\operatorname{lse}(x,\mu) :=
		\mu \log ( \textstyle \sum_{i=1}^{n} \exp ( x_i/\mu) ),
	\end{equation*}
	is the smoothing function of $ \max (x) $ because we can see that
	
	(i) $\operatorname{lse}(\cdot, \mu)$ is smooth on $\mathbb{R}^{n}$ for any $\mu>0$.
	Its gradient
	$\nabla_{x} \operatorname{lse}(x,\mu) $ is given by $ \sigma(\cdot,\mu) : \mathbb{R}^n \rightarrow \Delta^{n-1}$,
	\begin{equation}\label{sigma1}
		\nabla_{x} \operatorname{lse}(x,\mu) =\sigma(x,\mu):=\frac{1}{\sum_{\ell=1}^{n} \exp (x_\ell/\mu )}
		[\begin{array}{c}
			\exp ( x_{1}/\mu ),
			\cdots,
			\exp ( x_{n}/\mu )
		\end{array}]^{\top},
	\end{equation}
	where $\Delta^{n-1}:=\{x \in \mathbb{R}^{n} \mid \sum_{i=1}^{n} x_i =1, x_{i} \geq 0\}$ is the unit simplex.
	
	(ii) For all $ x \in \mathbb{R}^n$ and $ \mu > 0 $, we have $ \max (x) < \operatorname{lse}(x,\mu) \leq \max( x) + \mu \log(n). $ 
	Then, the constant $\kappa= \log(n)$ and $\omega(\mu)=\mu$. The above inequalities imply that $\lim _{z \rightarrow x, \mu \downarrow 0} \operatorname{lse}(z,\mu)= \max (x).$
\end{example}

Gradient sub-consistency or consistency is crucial to showing that any limit point of the Riemannian smoothing method is also a limiting stationary point of (\ref{NROP}).

\begin{definition}[Zhang, Chen and Ma {\cite[Definition 3.4 \& 3.9]{zhang2021riemannian}}] \label{defn:sub-consistency}
	A smoothing function $\tilde{f}$ of $f$ is said to satisfy \emph{gradient sub-consistency on $\mathbb{R}^{n}$} if, for any $x \in \mathbb{R}^{n}$,
	\begin{equation}\label{eq:sub_R}
		G_{\tilde{f}}(x) \subseteq \partial f(x),
	\end{equation}
	where the \emph{subdifferential of $f$ associated with $\tilde{f}$ at $x \in \mathbb{R}^{n}$} is given by
	\begin{equation*}
		G_{\tilde{f}}(x):=\{u \in \mathbb{R}^{n} \mid \nabla_{x} \tilde{f}\left(z_{k}, \mu_{k}\right) \rightarrow u  \text { for some } z_{k} \rightarrow x, \mu_{k} \downarrow 0\}.
	\end{equation*}
	Similarly, $\tilde{f}$ is said to satisfy \emph{Riemannian gradient sub-consistency on $\mathcal{M}$} if, for any $x \in \mathcal{M}$,
	\begin{equation}\label{eq:sub_M}
		G_{\tilde{f}, \mathcal{R}}(x) \subseteq \partial_{\mathcal{R}} f(x),
	\end{equation}
	where the \emph{Riemannian subdifferential of $f$ associated with $\tilde{f}$ at $x \in \mathcal{M}$} is given by
	\begin{equation*}
		G_{\tilde{f}, \mathcal{R}}(x)=\{v \in \mathbb{R}^{n}\mid \operatorname{grad} \tilde{f}\left(z_{k}, \mu_{k}\right) \rightarrow v \text { for some } z_{k} \in \mathcal{M}, z_{k} \rightarrow x, \mu_{k} \downarrow 0\}.
	\end{equation*}
\end{definition}

If one substitutes the inclusion with equality in (\ref{eq:sub_R}), then $\tilde{f}$ satisfies \emph{gradient consistency} on $\mathbb{R}^{n}$, and similarly in (\ref{eq:sub_M}) for $\mathcal{M}$. Thanks to the following useful proposition from \cite{zhang2021riemannian}, we can induce gradient sub-consistency on $\mathcal{M}$ from that on $\mathbb{R}^{n}$ if $f$ is locally Lipschitz.

\begin{proposition}[Zhang, Chen and Ma {\cite[Proposition 3.10]{zhang2021riemannian}}] 
	Let $f$ be a locally Lipschitz function and $\tilde{f}$ a smoothing function of $f .$ For $\tilde{f}$, if gradient sub-consistency holds on $\mathbb{R}^{n}$, then Riemannian gradient sub-consistency holds on $\mathcal{M}$ as well.
\end{proposition}

The next example illustrates Riemannian gradient sub-consistency on $\mathcal{M}$ for $\operatorname{lse}(x,\mu)$ in Example \ref{exmp:lse1}, since any convex function is locally Lipschitz continuous.

\begin{example}[Chen, Wets, and Zhang {\cite[Lemma 4.4]{chen2012stochastic}}]
	\label{thm:lse2}
	The smoothing function $\operatorname{lse}(x,\mu)$ of $\max(x)$ satisfies gradient consistency on $\mathbb{R}^{n}$. That is, for any $ x \in \mathbb{R}^n$,
	\begin{equation*}
		\partial \max(x)=G_{\operatorname{lse}}(x)=\{\lim _{x^{k} \rightarrow x, \mu_{k} \downarrow 0} \sigma(x^{k},\mu_{k})\}.
	\end{equation*}
	Note that the original assertion of \cite[Lemma 4.4]{chen2012stochastic} is gradient consistency in the Clarke sense, i.e.,
	$ \partial^{\circ} \max(x)=G_{\operatorname{lse}}(x) $.
\end{example}

\subsection{Riemannian smoothing method}

Motivated by the previous papers \cite{cambier2016robust,liu2020simple,zhang2021riemannian} on smoothing methods and Riemannian manifolds, we propose a general Riemannian smoothing method. Algorithm \ref{alg:sm_NROP} is the basic framework of this general method.

\begin{algorithm}
	\caption{Basic Riemannian smoothing method for (\ref{NROP})}
	\label{alg:sm_NROP}
	\begin{algorithmic}
		\STATE{\textbf{Initialization:}}
		Given $\theta \in(0,1), \mu_{0}>0$, and $x_{-1} \in \mathcal{M}$, select a smoothing function $\tilde{f}$ and a Riemannian algorithm (called sub-algorithm here) for (\ref{SROP}).
		
		\STATE{\textbf{for} $k=0,1,2, \ldots$}
		
		\quad Solve
		\begin{equation}\label{smp0}
			x_{k} = \arg\min _{x \in \mathcal{M}} \tilde{f}(x, \mu_{k})
		\end{equation}
		\quad approximately by using the chosen sub-algorithm, starting at $x^{k-1}$

		\quad \textbf{if}
		final convergence test is satisfied
		
		\qquad \textbf{stop} with approximate solution $x_{k}$
		
		\quad \textbf{end if}
		
		\quad Set $\mu_{k+1}=\theta \mu_{k}$
		
		\STATE{\textbf{end for}}
	\end{algorithmic}
\end{algorithm}

Now let us describe the convergence properties of the basic method. First, let us assume that the function $\tilde{f}(x, \mu_{k})$ has a minimizer on $\mathcal{M}$ for each value of $\mu_{k}$.
\begin{theorem}\label{thm:basic1}
	Suppose that each $x^{k}$ is an exact global minimizer of (\ref{smp0}) in Algorithm \ref{alg:sm_NROP}.
	Then every limit point $x^{*}$ of the sequence $\{x^{k}\}$ is a global minimizer of the problem (\ref{NROP}).
\end{theorem}

\begin{proof}
	Let $\bar{x}$ be a global solution of (\ref{NROP}), that is,
	\begin{equation*}
		f(\bar{x}) \leq f(x) \quad \text { for all } x \in \mathcal{M}.
	\end{equation*}
	From the Definition \ref{defn:sm} of the smoothing function, there exist a constant $\kappa>0$ and a function $\omega: \mathbb{R}_{++} \rightarrow \mathbb{R}_{++}$ such that, for all $x \in \mathcal{M}$,
	\begin{equation}\label{16}
		-\kappa \omega(\mu) \leq \tilde{f}(x, \mu)-f(x) \leq \kappa \omega(\mu)
	\end{equation}
	with $\lim _{\mu \downarrow 0} \omega(\mu)=0.$ 
	Substituting $ x^{k} $ and combining with the global solution $ \bar{x} $, we have that
	\begin{equation*}
		\tilde{f}(x^{k}, \mu_{k}) \geq f(x^{k})-\kappa \omega(\mu_{k}) 
		\geq f(\bar{x})-\kappa \omega(\mu_{k}).
	\end{equation*}
	By rearranging this expression, we obtain
	\begin{equation}\label{17}
		-\kappa \omega(\mu_{k}) \leq \tilde{f}(x^{k}, \mu_{k}) -f(\bar{x}).
	\end{equation}
	\indent Since $x^{k}$ minimizes $\tilde{f}(x, \mu_{k})$ on $\mathcal{M} $ for each $\mu_{k}$, we have that $\tilde{f}(x^{k}, \mu_{k}) \leq \tilde{f}(\bar{x}, \mu_{k})$, which leads to
	\begin{equation}\label{18}
		\tilde{f}(x^{k}, \mu_{k}) - f(\bar{x}) \leq \tilde{f}(\bar{x}, \mu_{k}) - f(\bar{x}) \leq \kappa \omega(\mu_{k}).
	\end{equation}
	The second inequality above follows from (\ref{16}).
	Combining (\ref{17}) and (\ref{18}), we obtain
	\begin{equation}\label{eq:231}
		\lvert \tilde{f}(x^{k}, \mu_{k})-f(\bar{x})\rvert \leq \kappa \omega(\mu_{k}).
	\end{equation}
	\indent Now, suppose that $x^{*}$ is a limit point of $\{x^{k}\}$, so that there is an infinite subsequence $\mathcal{K}$ such that
	$
	\lim _{k \in \mathcal{K}} x^{k}=x^{*}.
	$
	Note that $x^{*} \in \mathcal{M}$ because $\mathcal{M}$ is complete.
	By taking the limit as $k \rightarrow \infty, k \in \mathcal{K}$, on both sides of (\ref{eq:231}), again by the definition of the smoothing function, we obtain
	\begin{equation*}
		\lvert f(x^{*})-f(\bar{x}) \rvert =
		\lim _{k \in \mathcal{K}} \lvert \tilde{f}(x^{k}, \mu_{k})-f(\bar{x})\rvert \leq \lim _{k \in \mathcal{K}} \kappa \omega(\mu_{k}) = 0.
	\end{equation*}
	Thus, it follows that $ f(x^{*})=f(\bar{x}).$
	Since $x^{*} \in \mathcal{M}$ is a point whose objective value is equal to that of the global solution $\bar{x}$, we conclude that $x^{*}$, too, is a global solution.
\end{proof}

This strong result requires us to find a global minimizer of each subproblem, which, however, cannot always be done. The next result concerns the convergence properties of the sequence $\tilde{f}(x^{k}, \mu_{k})$ under the condition that $ \tilde{f} $ has the following additional property:
\begin{equation}\label{AP1}
	0< \mu_{2} < \mu_{1} \Longrightarrow
	\tilde{f}(x,\mu_{2})  < \tilde{f}(x,\mu_{1})\text{ for all } x \in \mathbb{R}^n.
\end{equation}

\begin{example}\label{exmp:lse2}
	The above property holds for $\operatorname{lse}(x,\mu)$ in Example \ref{exmp:lse1}; i.e., we have $\operatorname{lse}(x,\mu_{2}) < \operatorname{lse}(x,\mu_{1})$ on $\mathbb{R}^n$, provided that $ 0< \mu_{2} < \mu_{1} $. 
		{Note that under the equality, 
		\begin{equation*}
			\sum_{l=1}^{n} \exp (x_l/{\mu} ) =\exp\{ \operatorname{lse}(x,\mu) / \mu\},
		\end{equation*}
		the $i$-th component of $ \sigma(x,\mu) $ can be rewritten as
		\begin{equation*}
			\sigma_{i}(x,\mu)=
			\exp\{ {( x_i - \operatorname{lse}(x,\mu)) } /\mu\}.
		\end{equation*}
		For any fixed $ x \in \mathbb{R}^n$, consider the derivative of a real function
		$ \mu \rightarrow \operatorname{lse}(x,\cdot) : \mathbb{R}_{++} \to \mathbb{R}$.
		Then we have
		\begin{equation*}
			\begin{aligned}
				\nabla_{\mu} \operatorname{lse}(x,\mu) = \operatorname{lse} / \mu - \frac{\sum_{i=1}^{n} x_{i} \exp (x_{i}/{\mu})}{\mu \exp
					{(\operatorname{lse} / \mu)}}
				=& ( \operatorname{lse} - \textstyle \sum_{i=1}^{n} x_{i} \exp\{( x_i - \operatorname{lse})/\mu \} ) /\mu \\
				=& ( \operatorname{lse} - \textstyle \sum_{i=1}^{n} x_{i} \sigma_i ) /\mu \leq  0,
			\end{aligned}
		\end{equation*}
		where ``$ \operatorname{lse} $, $ \sigma $'' are shorthand for $ \operatorname{lse}(x,\mu)$ and $ \sigma(x,\mu)$.
		For the last inequality above, we observe that $ \sigma \in \Delta^{n-1}$;  hence, the term $ \sum_{i=1}^{n} x_{i} \sigma_i$ is a convex combination of all entries of $ x $, which implies that $ \sum_{i=1}^{n} x_{i} \sigma_i \leq \max (x) < \operatorname{lse}. $
		This completes the proofs of our claims.}
\end{example}

In \cite{cambier2016robust}, the authors considered a special case of Algorithm \ref{alg:sm_NROP}, wherein the smoothing function $\tilde{f}(x,\mu) =\sqrt{\mu^{2}+x^{2}}$ of $\lvert x \rvert$ also satisfies (\ref{AP1}) and a Riemannian conjugate gradient method is used for (\ref{smp0}).

\begin{theorem}\label{thm:basic2}
	Suppose that $f^{*}:=\inf _{x \in \mathcal{M}} f(x)$ exists and the smoothing function $ \tilde{f} $ has property (\ref{AP1}). Let $f^{k}:=\tilde{f}(x^{k}, \mu_{k})$. Then the sequence $\{f^{k}\}$ generated by Algorithm \ref{alg:sm_NROP} is strictly decreasing and bounded below by $f^{*}$; hence,
	\begin{equation*}
		\lim_{k \to \infty} \lvert f^{k} - f^{k-1} \rvert =0.
	\end{equation*}
\end{theorem}

\begin{proof}
	For each $ k \geq 1 $, $x_{k}$ is obtained by approximately solving
	\begin{equation*}
		\min _{x \in \mathcal{M}} \tilde{f}(x, \mu_{k}),
	\end{equation*}
	starting at $x^{k-1}$.
	Then at least, we have
	\begin{equation*}
		\tilde{f}(x^{k-1}, \mu_{k}) \geq \tilde{f}(x^{k}, \mu_{k})=f^{k}.
	\end{equation*}
	Since $\mu_{k}=\theta \mu_{k-1} < \mu_{k-1} $, property (\ref{AP1}) ensures
	\begin{equation*}
		f^{k-1}=\tilde{f}(x^{k-1}, \mu_{k-1}) > \tilde{f}(x^{k-1}, \mu_{k}).
	\end{equation*}
	The claim that sequence $\{f^{k}\}$ is strictly decreasing follows from these two inequalities.\\
	\indent Suppose that, for all $\mu > 0$ and for all $x \in \mathbb{R}^n$,
	\begin{equation}\label{eq:big}
		\tilde{f}(x, \mu) \geq {f}(x).
	\end{equation}
	Then for each $ k $, 
	\begin{equation*}
		f^{k}=\tilde{f}(x^{k}, \mu_{k}) \geq {f}(x^{k}) \geq \inf _{x \in \mathcal{M}} f(x) = f^{*},
	\end{equation*}
	which proves our claims.\\
	\indent
	Now, we show (\ref{eq:big}) is true if the smoothing function has property (\ref{AP1}). Fix any $x \in \mathbb{R}^n$; (\ref{AP1}) implies that $ \tilde{f}(x, \cdot) $ is strictly decreasing as $ \mu \to 0.$ 
	Actually, $ \tilde{f}(x, \cdot) $ is monotonically increasing on $ \mu >0. $ 
	On the other hand, from the definition of the smoothing function, we have that
	\begin{equation*}
		\lim _{\mu \downarrow 0} \tilde{f}(x, \mu)=f(x).
	\end{equation*}
	Hence, we have
	$
	\inf _{\mu > 0} \tilde{f}(x, \mu)=f(x),
	$
	as claimed.
\end{proof}

Note that the above weak result does not ensure that $\{ f^k\} \rightarrow f^*$. Next, for better convergence (compared with Theorem \ref{thm:basic2}) and an effortless implementation (compared with Theorem \ref{thm:basic1}), we propose an enhanced Riemannian smoothing method: Algorithm \ref{alg:sm_NROP222}. This is closer to the version in \cite{zhang2021riemannian}, where the authors use the Riemannian steepest descent method for solving the smoothed problem (\ref{smp1}).

\begin{algorithm}
	\caption{Enhanced Riemannian smoothing method for (\ref{NROP})}
	\label{alg:sm_NROP222}
	\begin{algorithmic}
		\STATE{\textbf{Initialization:}}
		Given $\theta \in(0,1), \mu_{0}>0$ and a nonnegative sequence $\{\delta_{k}\}$ with $\delta_{k} \rightarrow 0$, and $x_{-1} \in \mathcal{M}$, select a smoothing function $\tilde{f}$ and a Riemannian algorithm (called sub-algorithm here) for (\ref{SROP}).
		
		\STATE{\textbf{for} $k=0,1,2, \ldots$}
		
		\quad Solve
		\begin{equation}\label{smp1}
			x_{k} = \arg\min _{x \in \mathcal{M}} \tilde{f}(x, \mu_{k})
		\end{equation}
		\quad approximately by using the chosen sub-algorithm, starting at $x^{k-1}$, such that
		\begin{equation}\label{1601}
			\|\operatorname{grad} \tilde{f}(x^{k}, \mu_{k})\| < \delta_{k}
		\end{equation}
		
		\quad \textbf{if}
		final convergence test is satisfied
		
		\qquad \textbf{stop} with approximate solution $x_{k}$
		
		\quad \textbf{end if}
		
		\quad Set $\mu_{k+1}=\theta \mu_{k}$
		
		\STATE{\textbf{end for}}
	\end{algorithmic}
\end{algorithm}

The following result is adapted from \cite[Proposition 4.2 \& Theorem 4.3]{zhang2021riemannian}. Readers are encouraged to refer to \cite{zhang2021riemannian} for a discussion on the stationary point associated with $\tilde{f}$ on $\mathcal{M}$.

\begin{theorem}\label{thm:c3}
	{In Algorithm \ref{alg:sm_NROP222},
	suppose that the chosen sub-algorithm has the following general convergence property for (\ref{SROP}):
	\begin{equation}\label{eq:weakconv}
		\liminf _{\ell \rightarrow \infty}\| \operatorname{grad} g(x^{\ell})\|=0.
	\end{equation}
	Moreover, suppose that, for all $ \mu_{k} $, the function $\tilde{f}(\cdot,\mu_{k})$ satisfies the convergence assumptions of the sub-algorithm needed for $ g $ above} and $\tilde{f}$ satisfies the Riemannian gradient sub-consistency on $ \mathcal{M} $. Then
	\begin{enumerate}
		\item For each $ k $, there exists an $ x^{k} $ satisfying (\ref{1601}); hence, Algorithm \ref{alg:sm_NROP222} is well defined.
		
		\item {Every limit point $x^{*}$ of the sequence $\{x^{k}\}$ generated by Algorithm \ref{alg:sm_NROP222} is a Riemannian limiting stationary point of (\ref{NROP}) (see (\ref{Rlsp})).} 
	\end{enumerate}
\end{theorem}

\begin{proof}
	Fix any $ \mu_{k} $. By (\ref{eq:weakconv}), we have $ \liminf _{\ell \rightarrow \infty}\| \operatorname{grad} \tilde{f}(x^{\ell},\mu_{k})\|=0 $. Hence, there is a convergent subsequence of $ \| \operatorname{grad} \tilde{f}(x^{\ell},\mu_{k})\| $ whose limit is 0. This means that, for any $\epsilon >0$, there exists an integer $ \ell_\epsilon$ such that $ \| \operatorname{grad} \tilde{f}(x^{\ell_\epsilon},\mu_{k})\| < \epsilon. $ If $ \epsilon = \delta_{k}$, we get $ x^{k}= x^{\ell_\epsilon} $. Thus, statement (1) holds.
	
	Next, suppose that $x^{*}$ is a limit point of $\{x^{k}\}$ generated by Algorithm \ref{alg:sm_NROP222}, so that there is an infinite subsequence $\mathcal{K}$ such that $ \lim _{k \in \mathcal{K}} x^{k}=x^{*}. $ From (1), we have
	\begin{equation*}
		\lim _{k \in \mathcal{K}} \|\operatorname{grad} \tilde{f}(x^{k}, \mu_{k})\|
		\leq \lim _{k \in \mathcal{K}} \delta_{k}=0,
	\end{equation*}
	and we find that $\operatorname{grad} \tilde{f}(x^{k}, \mu_{k}) \rightarrow 0$ for $k \in \mathcal{K}, x^{k} \in \mathcal{M}, x^{k} \rightarrow x^{*}, \mu_{k} \downarrow 0$. Hence,
	\begin{equation*}
		0 \in G_{\tilde{f}, \mathcal{R}}(x^{*}) \subseteq \partial_{\mathcal{R}} f(x^{*}).
	\end{equation*}
\end{proof}

Now let us consider the selection strategy of the nonnegative sequence $\{\delta_{k}\}$ with $\delta_{k} \rightarrow 0$. In \cite{zhang2021riemannian}, when $\mu_{k+1}=\theta \mu_{k}$ shrinks, the authors set 
\begin{equation}\label{par:zhang}
	\delta_{k+1}:=\rho\delta_{k}
\end{equation}
with an initial value of $\delta_{0}$ and constant $\rho\in(0,1)$. In the spirit of the usual smoothing methods described in \cite{chen2012smoothing}, one can set
\begin{equation}\label{par:wo}
	\delta_{k}:=\gamma \mu_{k}
\end{equation} with a constant $\gamma>0$. The latter is an adaptive rule, because $ \mu_{k} $ determines subproblem (\ref{smp1}) and its stopping criterion at the same time. The merits and drawbacks of the two rules require more discussion, but the latter seems to be more reasonable.

\vspace{5mm}

We conclude this section by discussing the connections with \cite{cambier2016robust} and \cite{zhang2021riemannian}. 
Our work is based on an efficient unification of them. 
\cite{cambier2016robust} focused on a specific algorithm and did not discuss the underlying generalities, whereas we studied a general framework for Riemannian smoothing. 
Recall that the ``smoothing function'' is the core tool of the smoothing method. 
In addition to what are required by its definition (see Definition \ref{defn:sm}), it needs to have the following ``additional properties'' (AP) in order for the algorithms to converge:
\begin{enumerate}[label=(AP\arabic*)]
	\setlength{\leftskip}{0.3cm}
	\item Approximate from above, i.e., (\ref{AP1}). (Needed in Algorithm 1)
	\item (Riemannian) gradient sub-consistency, i.e., Definition \ref{defn:sub-consistency}.
	(Needed in Algorithm 2)
\end{enumerate}

We find that not all smoothing functions satisfy (AP1) and for some functions it is hard to prove whether (AP2) holds. For example, all the functions in Table \ref{table:smlist} are smoothing functions of $ |x|$, but only the first three meet (AP1); the last two do not. In \cite{chen2012smoothing}, the authors showed that the first one in Table \ref{table:smlist}, $\tilde{f}_{1}(x, \mu)$, has property (AP2). The others remain to be verified, but doing so will not be a trivial exercise. To a certain extent, Algorithm \ref{alg:sm_NROP} as well as Theorem \ref{thm:basic2} guarantee a fundamental convergence result even if one has difficulty in showing whether one's smoothing function satisfies (AP2). Therefore, it makes sense to consider Algorithms \ref{alg:sm_NROP} and \ref{alg:sm_NROP222} together for the sake of the completeness of the general framework.

\begin{table}[!h]
	\renewcommand{\arraystretch}{1.5}
	\caption{List of smoothing functions of the absolute value function $ |x|$ with $ \kappa $ and $ \omega(\mu) $ in (\ref{defn:kappa})}
	\centering
	\label{table:smlist}
		\begin{tabular}{|l|c|c|}
			\hline
			& $ \kappa $ & $ \omega(\mu) $ \\ \hline
			$
			\tilde{f}_{1}(x,\mu)=\left\{\begin{array}{lll}
				|x| & \text { if } & |x|>\frac{\mu}{2} \\
				\frac{x^{2}}{\mu}+\frac{\mu}{4} & \text { if } & |x| \leq \frac{\mu}{2}
			\end{array}\right.$ 
			& $\frac{1}{4}$   & $\mu$   \\ \hline
			$ \tilde{f}_{2}(x,\mu)= \sqrt{\mu^{2}+x^{2}}$ & $1$    & $\mu$  \\ \hline
			$\tilde{f}_{3}(x, \mu)=2 \mu \log (1+e^{\frac{x}{\mu}})-x$ & $2 \log (2)$   & $\mu$   \\ \hline
			$\tilde{f}_{4}(x, \mu)=x \tanh (\frac{x}{\mu})$, where $\tanh (z)$ is the hyperbolic tangent function. & $ 1 $   & $\mu$   \\ \hline
			$\tilde{f}_{5}(x, \mu)=x \operatorname{erf}(\frac{x}{\mu})$, where $\operatorname{erf}(z):=\frac{2}{\sqrt{\pi}} \int_{0}^{z} e^{-t^{2}} d t$ is the Gauss error function. & $ \frac{2}{e \sqrt{\pi}} $   & $\mu$   \\ \hline
	\end{tabular}
\end{table}

Algorithm \ref{alg:sm_NROP222} expands on the results of \cite{zhang2021riemannian}. It allows us to use any standard method of (\ref{SROP}), not just steepest descent, to solve the smoothed problem (\ref{smp1}). 
Various standard Riemannian algorithms for (\ref{SROP}), such as 
the Riemannian conjugate gradient method \cite{sato2015new} (which often performs better than Riemannian steepest descent), 
the Riemannian Newton method \cite[Chapter 6]{absil2009optimization}, 
and the Riemannian trust region method \cite[Chapter 7]{absil2009optimization}, 
have extended the concepts and techniques used in Euclidean space to Riemannian manifolds. 
As shown by Theorem \ref{thm:c3}, no matter what kind of sub-algorithm is implemented for (\ref{smp1}), it does not affect the final convergence as long as the chosen sub-algorithm has property (\ref{eq:weakconv}).
On the other hand, we advocate that the sub-algorithm should be viewed as a ``Black Box'' and the user should not have to care about the code details of the sub-algorithm at all.
We can directly use an existing solver, e.g., Manopt \cite{manopt}, which includes the standard Riemannian algorithms mentioned above. Hence, we can choose the \emph{most suitable} sub-algorithm for the application and quickly implement it with minimal effort.

\section{Numerical experiments on CP factorization}\label{sec:ne}

The numerical experiments in Section \ref{sec:ne} and \ref{sec:comparison} were performed on a computer equipped with an Intel Core i7-10700 at 2.90GHz with 16GB of RAM using
Matlab R2022a. 
Our Algorithm \ref{alg:sm_NROP222} is implemented in the Manopt framework \cite{manopt} (version 7.0).
The number of iterations to solve the smoothed problem (\ref{smp1}) with the sub-algorithm is recorded in the total number of iterations.
We refer readers to the supplementary material of this paper for the available codes.\footnote{Alternatively, \url{https://github.com/GALVINLAI/General-Riemannian-Smoothing-Method}.}

In this section, we describe numerical experiments that we conducted on CP factorization in which we solved (\ref{OptCP}) using Algorithm \ref{alg:sm_NROP222}, where different Riemannian algorithms were employed as sub-algorithms and $\operatorname{lse}(-\bar{B}X,\mu)$ was used as the smoothing function. To be specific, we used three built-in Riemannian solvers of Manopt 7.0 --- steepest descent (SD), conjugate gradient (CG), and trust regions (RTR), denoted by {SM\_SD}, {SM\_CG} and {SM\_RTR}, respectively. We compared our algorithms with the following non-Riemannian numerical algorithms for CP factorization that were mentioned in subsection \ref{subs:rw}. We followed the settings used by the authors in their papers.

\begin{itemize}
	\item SpFeasDC\_ls \cite{chen2020difference}: A difference-of-convex functions approach for solving the split feasibility problem, it can be applied to (\ref{FeasCP}). The implementation details regarding the parameters we used are the same as in the numerical experiments reported in \cite[Section 6.1]{chen2020difference}.
	
	\item RIPG\_mod \cite{boct2021factorization}: This is a projected gradient method with relaxation and inertia parameters for solving (\ref{bot}). As shown in \cite[Section 4.2]{boct2021factorization}, RIPG\_mod is the best among the many strategies of choosing parameters.
	
	\item APM\_mod \cite{groetzner2020factorization}: A modified alternating projection method for CP factorization; it is described in Section \ref{subs:approaches}.
\end{itemize}
We have shown that $\operatorname{lse}(x,\mu)$ is a smoothing function of $\max (x)$ with gradient consistency. 
$\operatorname{lse}(\cdot, \mu)$ of the matrix argument can be simply derived from entrywise operations.
Then from the properties of compositions of smoothing functions \cite[Proposition 1 (3)]{bian2013neural}, we have that $\operatorname{lse}(-\bar{B}X,\mu)$ is a smoothing function of $\max(-\bar{B}X)$ with gradient consistency.
In practice, it is important to avoid numerical overflow and underflow when evaluating $\operatorname{lse}(x,\mu)$.
Overflow occurs when any $ x_i $ is large and underflow occurs when all $ x_i $ are small. To avoid these problems, we can shift each component $ x_i $ by $\max(x) $ and use the following formula:
\begin{equation*}
	\operatorname{lse}(x,\mu) =
	\mu \log (\textstyle \sum_{i=1}^{n} \exp ((x_{i}-\max (x))/{\mu} ) ) + \max (x),
\end{equation*}
whose validity is easy to show.

The details of the experiments are as follows.
If $ A \in \mathcal{CP}_{n}$ was of full rank, for accuracy reasons, we obtained an initial $ \bar{B} $ by using Cholesky decomposition. Otherwise, $ \bar{B} $ was obtained by using spectral decomposition.
Then we extended $ \bar{B} $ to $ r $ columns by column replication (\ref{col_rep}). We set $ r = \operatorname{cp}(A) $ if $\operatorname{cp}(A)$ was known or $r$ was sufficiently large.
We used \verb|RandOrthMat.m| \cite{RandOrthMat} to generate a random starting point $ X^0$ on the basis of the Gram-Schmidt process.

For our three algorithms, we set $ \mu_{0}=100, \theta=0.8$ and used an adaptive rule (\ref{par:wo}) of $\delta_{k}:=\gamma \mu_{k}$ with $\gamma=0.5 $. Except for RIPG\_mod, all the algorithms terminated successfully at iteration $k$, where $\min (\bar{B}X^{k}) \geq -10^{-15} $ was attained before the maximum number of iterations (5,000) was reached. In addition, SpFeasDC\_ls failed when $\bar{L}_{k}>10^{10}$. Regarding RIPG\_mod, it terminated successfully when $\|A-X_{k} X_{k}^{\top}\|^{2}/\|A\|^{2}<10^{-15}$ was attained before at most 10,000 iterations for $n<100$, and before at most 50,000 iterations in all other cases.
In the tables of this section, we report the rounded success rate (Rate) over the total number of trials, although the definitions of ``Rate'' in the different experiments (described in Sections 4.1-4.4) vary slightly from one experiment to the other. We will describe them later.

\subsection{Randomly generated instances}\label{subnerand}

We examined the case of randomly generated matrices to see how the methods were affected by the order $n$ or $r$. The instances were generated in the same way as in \cite[Section 7.7]{groetzner2020factorization}. We computed $ C $ by setting $C_{i j}:=\lvert B_{i j}\rvert $ for all $i, j,$ where $ B $ is a random $ n \times 2n $ matrix based on the Matlab command \texttt{randn}, and we took $A=C C^{\top}$ to be factorized. In Table \ref{test_randomCP}, we set $r=1.5 n$ and $r=3 n$ for the values $n \in \{20,30,40,100,200,400,600,800\}$.
For each pair of $ n $ and $ r $, we generated 50 instances if $ n \leq 100 $ and 10 instances otherwise. 
For each instance, we initialized all the algorithms at the same random starting point $ X^{0} $ and initial decomposition $ \bar{B} $, except for RIPG\_mod. Note that each instance $ A $ was assigned only one starting point. 

Table \ref{test_randomCP} lists the average time in seconds (Time$_{\mathrm{s}}$) and the average number of iterations (Iter$_{\mathrm{s}}$) among the successful instances. For our three Riemannian algorithms, Iter$_{\mathrm{s}}$ contains the number of iterations of the sub-algorithm. {Table \ref{test_randomCP} also lists the rounded success rate (Rate) over the total number (50 or 10) of instances for each pair of $ n $ and $ r $. Boldface highlights the two best results in each row.}

As shown in Table \ref{test_randomCP}, except for APM\_mod, each method had a success rate of 1 for all pairs of $ n $ and $ r $. Our three algorithms outperformed the other methods on the large-scale matrices with $ n \geq 100 $. 
In particular, SM\_CG with the conjugate-gradient method gave the best results.
\subsection{A specifically structured instance}

Let $\textbf{e}_{n}$ denote the all-ones vector in $\mathbb{R}^{n}$ and consider the matrix \cite[Example 7.1]{groetzner2020factorization},
\begin{equation*}
	A_{n} =\left(\begin{array}{cc}
		{0} & {\textbf{e}_{n-1}^{\top}} \\
		{\textbf{e}_{n-1}} & {I_{n-1}}
	\end{array}\right)^{\top}\left(\begin{array}{cc}
		{0} & {\textbf{e}_{n-1}^{\top}} \\
		{\textbf{e}_{n-1}} & {I_{n-1}}
	\end{array}\right) \in \mathcal{CP}_{n}.
\end{equation*}
Theorem \ref{thm:int_cp} shows that $A_{n} \in \operatorname{int}(\mathcal{CP}_{n})$ for every $n \geq 2.$  
By construction, it is obvious that $\operatorname{cp}(A_{n})=n$. We tried to factorize $A_{n}$ for the values $n \in \{10,20,50,75,100,150\}$ in Table \ref{test_specialCP}. 
For each $A_{n}$, {using $ r=\operatorname{cp}(A_{n})=n $} and the same initial decomposition $\bar{B}$, we tested all the algorithms on the same 50 randomly generated starting points, except for RIPG\_mod. Note that each instance was assigned 50 starting points. 

Table \ref{test_specialCP} lists the average time in seconds (Time$_{\mathrm{s}}$) and the average number of iterations (Iter$_{\mathrm{s}}$) among the successful starting points. 
{It also lists the rounded success rate (Rate) over the total number (50) of starting points for each $ n $.} 
Boldface highlights the two best results for each $ n $. 

We can see from Table \ref{test_specialCP} that the success rates of our three algorithms were always 1, whereas the success rates of the other methods decreased as $ n $ increased. 
Likewise, SM\_CG with the conjugate-gradient method gave the best results.

\begin{landscape}
	\begin{table}
		\small
		\caption{CP factorization of random completely positive matrices.} 
		\centering
		\label{test_randomCP}
		
		\begin{tabular}{l|lll|lll|lll|lll|lll|lll}
			\hline \hline
			\multicolumn{1}{l|}{Methods} & 		\multicolumn{3}{l|}{SM\_SD} &
			\multicolumn{3}{l|}{SM\_CG} & \multicolumn{3}{l|}{SM\_RTR} & \multicolumn{3}{l|}{SpFeasDC\_ls} & \multicolumn{3}{l|}{RIPG\_mod} & \multicolumn{3}{l}{APM\_mod} 
			\\ 
			\hline 
			
			{$n (r=1.5n)$} &
			Rate & Time$_{\mathrm{s}}$ & Iter$_{\mathrm{s}}$ & Rate & Time$_{\mathrm{s}}$ & Iter$_{\mathrm{s}}$ &
			Rate & Time$_{\mathrm{s}}$ & Iter$_{\mathrm{s}}$ & Rate & Time$_{\mathrm{s}}$ & Iter$_{\mathrm{s}}$ & Rate & Time$_{\mathrm{s}}$ & Iter$_{\mathrm{s}}$ & Rate & Time$_{\mathrm{s}}$ & Iter$_{\mathrm{s}}$ \\ 
			\hline
			
			20 & 
			1 & 0.0409 & 49 & 
			1 & 0.0394 & 41 & 
			1 & 0.0514 & 35 & 
			1 & \textbf{0.0027} & 24 & 
			1 & \textbf{0.0081} & 1229 & 
			0.32 & 0.3502 & 2318 \\
			
			30 & 
			1 & 0.0549 & 58 &
			1 & 0.0477 & 44 & 
			1 & 0.0690 & 36 & 
			1 & \textbf{0.0075} & 24 & 
			1 & \textbf{0.0231} & 1481 & 
			0.04 & 1.0075 & 2467 \\

			40 & 
			1 & 0.0735 & 65 &
			1 & 0.0606 & 46 & 
			1 & 0.0859 & 37 & 
			1 & \textbf{0.0216} & 46 & 
			1 & \textbf{0.0574} & 1990 & 
			0 & - & - \\

			100 & 
			1 & 0.2312 & 104 &
			1 & \textbf{0.1520} & 56 & 
			1 & 0.4061 & 45 & 
			1 & \textbf{0.2831} & 109 &
			1 & 0.8169 & 4912 & 
			0 & - & - \\
			
			200 & 
			1 & \textbf{1.0723} & 167 &
			1 & \textbf{0.5485} & 69 & 
			1 & 1.9855 & 53 & 
			1 & 2.2504 & 212 & 
			1 & 5.2908 & 9616 & 
			0 & - & - \\
			
			400 & 
			1 & \textbf{14.6}  & 314 &
			1 & \textbf{4.1453} & 86 & 
			1 & 22.1 & 69 & 
			1 & 36.9 & 636 & 
			1 & 90.6 & 17987 & 
			0 & - & - \\
			
			600 & 
			1 & 50.6 & 474 &
			1 & \textbf{14.7} & 105 & 
			1 & \textbf{46.4} & 80 & 
			1 & 140.1 & 882 & 
			1 & 344.7 & 26146 & 
			0 & - & - \\
			
			800 &  
			1 & 133.3 & 643 &
			1 & \textbf{30.0} & 109 & 
			1 & \textbf{93.5} & 89 & 
			1 & 413.3 & 1225 & 
			1 & 891.1 & 34022 & 
			0 & - & - 
			\\ 
			\hline
			
			\multicolumn{1}{l|}{Methods} & 		\multicolumn{3}{l|}{SM\_SD} &
			\multicolumn{3}{l|}{SM\_CG} & \multicolumn{3}{l|}{SM\_RTR} & \multicolumn{3}{l|}{SpFeasDC\_ls} & \multicolumn{3}{l|}{RIPG\_mod} & \multicolumn{3}{l}{APM\_mod} 
			\\
			\hline
			
			{$ n (r=3n)$} & 
			Rate & Time$_{\mathrm{s}}$ & Iter$_{\mathrm{s}}$ &
			Rate & Time$_{\mathrm{s}}$ & Iter$_{\mathrm{s}}$ & Rate & Time$_{\mathrm{s}}$ & Iter$_{\mathrm{s}}$ & Rate & Time$_{\mathrm{s}}$ &Iter$_{\mathrm{s}}$ & Rate & Time$_{\mathrm{s}}$ & Iter$_{\mathrm{s}}$ & Rate & Time$_{\mathrm{s}}$  & Iter$_{\mathrm{s}}$ \\ 
			\hline
			
			20 & 
			1 & 0.0597 & 52 &
			1 & 0.0551 & 42 & 
			1 & 0.0842 & 37 & 
			1 & \textbf{0.0057} & 15 & 
			1 & \textbf{0.0105} & 1062 & 
			0.30 & 0.7267 & 2198 \\
			
			30 & 
			1 & 0.0793 & 59 &
			1 & 0.0673 & 45 & 
			1 & 0.1161 & 39 & 
			1 & \textbf{0.0128} & 17 & 
			1 & \textbf{0.0336} & 1127 & 
			0 & - & - \\
			
			40 & 
			1 & 0.1035 & 67 &
			1 & 0.0882 & 48 & 
			1 & 0.1961 & 41 & 
			1 & \textbf{0.0256} & 19 & 
			1 & \textbf{0.0822} & 1460 & 
			0 & - & - \\
			
			100 &
			1 & \textbf{0.5632} & 103 &
			1 & \textbf{0.3395} & 57 & 
			1 & 1.4128 & 50 & 
			1 & 0.8115 & 86 & 
			1 & 1.1909 & 4753 & 
			0 & - & - \\
			
			200 & 
			1 & \textbf{4.5548} & 163 &
			1 & \textbf{2.4116} & 68 & 
			1 & 14.3 & 65 & 
			1 & 8.1517 & 184 & 
			1 & 9.2248 & 9402 & 
			0 & - & - \\

			400 & 
			1 & \textbf{46.5} & 296 &
			1 & \textbf{19.2} & 89 & 
			1 & 77.1 & 80 & 
			1 & 124.3 & 453 & 
			1 & 156.6 & 17563 & 
			0 & - & - \\
			
			600 & 
			1 & \textbf{209.0} & 446 &
			1 & \textbf{65.7} & 99 & 
			1 & 294.5 & 76 & 
			1 & 981.8 & 795 & 
			1 & 616.7 & 25336 & 
			0 & - & - \\
			
			800 & 
			1 & \textbf{609.7} & 628 &
			1 & \textbf{160.4} & 114 & 
			1 & 650.7 & 83 & 
			1 & 4027.4 & 1070 & 
			1 & 1289.4  & 26820 & 
			0 & - & - 
			\\
			\hline
		\end{tabular}
	\end{table}
	
	\begin{table}
		\small
		\caption{CP factorization of a family of specifically structured instances.} 
		\centering
		\label{test_specialCP}
		\begin{tabular}{l|lll|lll|lll|lll|lll|lll}
			\hline \hline
			Methods &
			\multicolumn{3}{l|}{SM\_SD} &
			\multicolumn{3}{l|}{SM\_CG} & \multicolumn{3}{l|}{SM\_RTR} & \multicolumn{3}{l|}{{SpFeasDC\_ls}} & \multicolumn{3}{l|}{{RIPG\_mod}} & \multicolumn{3}{l}{{APM\_mod}}  
			\\ 
			\hline
			
			{$ n (r=n)$} & Rate & Time$_{\mathrm{s}}$ & Iter$_{\mathrm{s}}$ & Rate & Time$_{\mathrm{s}}$ &
			Iter$_{\mathrm{s}}$ & Rate & Time$_{\mathrm{s}}$ & Iter$_{\mathrm{s}}$ & Rate & Time$_{\mathrm{s}}$ & Iter$_{\mathrm{s}}$ & Rate & Time$_{\mathrm{s}}$ & Iter$_{\mathrm{s}}$ & Rate & Time$_{\mathrm{s}}$ & Iter$_{\mathrm{s}}$ 
			\\ 
			\hline
			
			10 & 
			1 & 0.0399 & 71 & 
			1 & 0.0313 & 49 & 
			1 & 0.0424 & 45 & 
			1 & \textbf{0.0043} & 149 & 
			1 & \textbf{0.0074} & 2085 & 
			0.80 & 0.0174 & 616 \\
			
			20 & 
			1 & 0.0486 & 85 & 
			1 & 0.0408 & 63 & 
			1 & 0.0637 & 55 & 
			0.98 & \textbf{0.0139} & 201 & 
			0.74 & \textbf{0.0212} & 3478 & 
			0.90 & 0.0591 & 864 \\
			
			50 & 
			1 & 0.2599 & 295 & 
			1 & \textbf{0.1073} & 101 & 
			1 & \textbf{0.2104} & 76 & 
			0.98 & 0.3389 & 770 & 
			0 & - & - & 
			0.76 & 0.6948 & 1416 \\
			
			75 & 
			1 & \textbf{0.3843} & 329 & 
			1 & \textbf{0.1923} & 135 & 
			1 & 0.4293 & 93 & 
			0.98 & 1.0706 & 1186 & 
			0 & - & - & 
			0.64 & 1.4809 & 1510 \\
			
			100 & 
			1 & \textbf{0.7459} & 458 & 
			1 & \textbf{0.3289} & 168 & 
			1 & 0.9074 & 108 & 
			0.80 & 1.6653 & 1083 & 
			0 & - & - & 
			0.60 & 2.8150 & 1690 \\
			
			150 & 
			1 & \textbf{1.8076} & 647 & 
			1 & \textbf{0.7837} & 241 & 
			1 & 2.6030 & 145 & 
			0.70 & 3.7652 & 1170 & 
			0 & - & - & 
			0.35 & 9.9930 & 2959 
			\\ 
			\hline
		\end{tabular}
	\end{table}
\end{landscape}

\subsection{An easy instance on the boundary of $\mathcal{CP}_n$}

Consider the following matrix from \cite[Example 2.7]{so2015simple}:
\begin{equation*}
	A=\left(\begin{array}{ccccc}
		41 & 43 & 80 & 56 & 50 \\
		43 & 62 & 89 & 78 & 51 \\
		80 & 89 & 162 & 120 & 93 \\
		56 & 78 & 120 & 104 & 62 \\
		50 & 51 & 93 & 62 & 65
	\end{array}\right).
\end{equation*}
The sufficient condition from \cite[Theorem 2.5]{so2015simple} ensures that this matrix is completely positive and $ \operatorname{cp}(A)= \operatorname{rank}(A)=3.$ Theorem \ref{thm:int_cp} tells us that $ A \in \operatorname{bd}(\mathcal{C P}_5) $, since $ \operatorname{rank}(A) \neq 5 $. 

We found that all the algorithms could easily factorize this matrix. 
However, our three algorithms returned a CP factorization $ B $ whose smallest entry was as large as possible. 
In fact, they also maximized the smallest entry in the $ n \times r $ symmetric factorization of $ A $, since (\ref{OptCP}) is equivalent to
\begin{equation*}
	\max _{A=XX^{\top}, X \in \mathbb{R}^{ n \times r }} \{\min \:(X)\}.
\end{equation*}
When we did not terminate as soon as $\min (\bar{B}X^{k}) \geq -10^{-15} $, for example, after 1000 iterations, our algorithms gave the following CP factorization whose the smallest entry is around $2.8573 \gg -10^{-15}$:
\begin{equation*}
	A=B B^{\top} \text { , where } B \approx\left(\begin{array}{lll}
		3.5771 & 4.4766 & \textbf{2.8573} \\
		2.8574 & 3.0682 & 6.6650 \\
		8.3822 & 7.0001 & 6.5374 \\
		5.7515 & 2.8574 & 7.9219 \\
		2.8574 & 6.7741 & 3.3085
	\end{array}\right).
\end{equation*}

\subsection{A hard instance on the boundary of $\mathcal{CP}_n$}
\label{subnebd}

Next, we examined how well these methods worked on a hard matrix on the boundary of $\mathcal{CP}_n$. 
Consider the following matrix on the boundary taken from \cite{dur2008interior}:
\begin{equation*}A=\left(\begin{array}{lllll}
		8 & 5 & 1 & 1 & 5 \\
		5 & 8 & 5 & 1 & 1 \\
		1 & 5 & 8 & 5 & 1 \\
		1 & 1 & 5 & 8 & 5 \\
		5 & 1 & 1 & 5 & 8
	\end{array}\right) \in \operatorname{bd} (\mathcal{CP}_{5}).
\end{equation*}

Since $A \in \mathrm{bd}(\mathcal{C P}_{5})$ and $ A $ is of full rank, it follows from Theorem \ref{thm:int_cp} that $\operatorname{cp}^{+}(A) = \infty$; i.e., there is no strictly positive CP factorization for $ A. $ Hence, the global minimum of (\ref{OptCP}), $ t=0 $, is clear. None of the algorithms could decompose this matrix under our tolerance, $ 10^{-15} $, in the stopping criteria. As was done in \cite[Example 7.3]{groetzner2020factorization}, we investigated slight perturbations of this matrix. Given
\begin{equation*} M M^{\top}=:C \in \operatorname{int}(\mathcal{C P}_{5}) \text { with } M=\left(\begin{array}{cccccc}
		1 & 1 & 0 & 0 & 0 & 0 \\
		1 & 0 & 1 & 0 & 0 & 0 \\
		1 & 0 & 0 & 1 & 0 & 0 \\
		1 & 0 & 0 & 0 & 1 & 0 \\
		1 & 0 & 0 & 0 & 0 & 1
	\end{array}\right),
\end{equation*}
we factorized $A_{\lambda}:=\lambda A+(1-$ $\lambda) C$ for different values of $\lambda \in[0,1)$ using $r= 12 > \mathrm{cp}_{5}=11.$ Note that $A_{\lambda} \in \operatorname{int}(C \mathcal{P}_{5})$ provided $0 \leq \lambda<1$ and $ A_{\lambda} $ approached the boundary as $ \lambda \to 1 $. We chose the largest $ \lambda = 0.9999$. 
{For each $A_{\lambda},$ we tested all of the algorithms on 50 randomly generated starting points and computed the success rate over the total number of starting points.}

Table \ref{test_bdint_plot} shows how the success rate of each algorithm changes as $ A_{\lambda} $ approaches the boundary. The table sorts the results from left to right according to overall performance. Except for SM\_RTR, whose success rate was always 1, the success rates of all the other algorithms significantly decreased as $ \lambda $ increased to 0.9999. Surprisingly, the method of SM\_CG, which performed well in the previous examples, seemed unable to handle instances close to the boundary.

\begin{table}[!t]
	\small
	\caption{Success rate of CP factorization of $A_{\lambda}$ for values of $\lambda$ from 0.6 to 0.9999.} 
	\centering
	\label{test_bdint_plot}
	\begin{tabular}{l|l|l|l|l|l|l}
		\hline
		$\lambda$ & SM\_RTR & SM\_SD & RIPG\_mod & SM\_CG & SpFeasDC\_ls & APM\_mod \\ \hline 
		0.6    & 1       & 1      & 1& 1      & 1   & 0.42     \\
		0.65   & 1       & 1      & 1& 1      & 1   & 0.44     \\
		0.7    & 1       & 1      & 1& 1      & 1   & 0.48     \\
		0.75   & 1       & 1      & 1& 1      & 1   & 0.52     \\
		0.8    & 1       & 1      & 1& 1      & 0.96& 0.46     \\ \hline
		0.82   & 1       & 1      & 1& 1      & 0.98& 0.4      \\
		0.84   & 1       & 1      & 1& 1      & 0.86& 0.24     \\
		0.86   & 1       & 1      & 1& 1      & 0.82& 0.1      \\
		0.88   & 1       & 1      & 1& 1      & 0.58& 0.18     \\
		0.9    & 1       & 1      & 1& 1      & 0.48& 0.18     \\ \hline
		0.91   & 1       & 1      & 1& 1      & 0.4 & 0.14     \\
		0.92   & 1       & 1      & 1& 1      & 0.2 & 0.18     \\
		0.93   & 1       & 1      & 0.98      & 1      & 0.22& 0.22     \\
		0.94   & 1       & 1      & 0.98      & 1      & 0.1 & 0.2      \\
		0.95   & 1       & 1      & 1& 1      & 0.12& 0.32     \\
		0.96   & 1       & 1      & 0.96      & 0.98   & 0.06& 0.34     \\
		0.97   & 1       & 1      & 0.86      & 0.82   & 0.06& 0.14     \\
		0.98   & 1       & 1      & 0.76      & 0.28   & 0.02& 0        \\
		0.99   & 1       & 0.68   & 0.42      & 0      & 0   & 0        \\ \hline
		0.999  & 1       & 0      & 0.14      & 0      & 0   & 0        \\
		0.9999 & 1       & 0      & 0& 0      & 0   & 0  \\ \hline     
	\end{tabular}
\end{table}

\section{Further numerical experiments: comparison with \cite{cambier2016robust,zhang2021riemannian}}
\label{sec:comparison}

As described at the end of Section \ref{sec:sm}, the algorithms in \cite{zhang2021riemannian} and \cite{cambier2016robust} are both special cases of our algorithm. 
In this section, we compare them to show whether it performs better when we use other sub-algorithms or other smoothing functions. 
We applied Algorithm \ref{alg:sm_NROP222} to two problems: finding a sparse vector (FSV) in a subspace and robust low-rank matrix completion, which are problems implemented in \cite{zhang2021riemannian} and \cite{cambier2016robust}, respectively. Since they both involve approximations to the $ \ell_{1} $ norm, we applied the smoothing functions listed in Table \ref{table:smlist}.

We used the \emph{six} solvers built into Manopt 7.0, namely, 
steepest descent; 
Barzilai-Borwein (i.e., gradient-descent with BB step size); 
Conjugate gradient; 
trust regions; 
BFGS (a limited-memory version); 
ARC (i.e., adaptive regularization by cubics).

\subsection{FSV problem}

The FSV problem is to find the sparsest vector in an $n$-dimensional linear subspace $W\subseteq \mathbb{R}^{m}$; it has applications in robust subspace recovery, dictionary learning, and many other problems in machine learning and signal processing \cite{qu2014finding,qu2020finding}. 
Let $Q \in \mathbb{R}^{m \times n}$ denote a matrix whose columns form an orthonormal basis of $W$: this problem can be formulated as
\begin{equation*}
	\min_{ x \in S^{n-1}} \|Q x\|_{0},
\end{equation*}
where $S^{n-1}:=\left\{x \in \mathbb{R}^{n}\mid\|x\|=1\right\}$ is the sphere manifold, and $\|z\|_{0}$ counts the number of nonzero components of $z$.
Because this discontinuous objective function is unwieldy, in the literature, one instead focuses on solving the $\ell_{1}$ norm relaxation given below:
\begin{equation}\label{fsv}
		\min_{ x \in S^{n-1}} \|Q x\|_{1},
\end{equation}
where $\|z\|_{1}:=\sum_{i}\left|z_{i}\right|$ is the $\ell_{1}$ norm of the vector $z$.

Our synthetic problems of the $\ell_{1}$ minimization model (\ref{fsv}) were generated in the same way as in \cite{zhang2021riemannian}: i.e., we chose $m \in\{4 n, 6 n, 8 n, 10 n\}$ for $n=5$ and $m \in\{6 n, 8 n, 10 n, 12n\}$ for $n=10$. We defined a sparse vector $e_{n}:=(1, \ldots, 1,0, \ldots, 0)^{\top} \in \mathbb{R}^{m}$, whose first $n$ components are 1 and the remaining components are 0. Let the subspace $W$ be the span of $e_{n}$ and some $n-1$ random vectors in $\mathbb{R}^{m}$. By \texttt{mgson.m} \cite{mgson}, we generated an orthonormal basis of $W$ to form a matrix $Q \in \mathbb{R}^{m \times n}$. With this construction, the minimum value of $\|Q x\|_{0}$ should be equal to $n$. We chose the initial points by using the \texttt{M.rand()} tool of Manopt 7.0 that returns a random point on the manifold \texttt{M} and set \texttt{x0 = abs(M.rand())}. The nonnegative initial point seemed to be better in the experiment. Regarding the the settings of our Algorithm \ref{alg:sm_NROP222}, we chose the same smoothing function $\tilde{f}_{1}(x, \mu)$ in Table \ref{table:smlist} and the same gradient tolerance strategy (\ref{par:zhang}) as in \cite{zhang2021riemannian}: $\mu_{0}=1, \theta=0.5, \delta_{0}=0.1, \rho=0.5.$ We compared the numerical performances when using different sub-algorithms. Note that with the choice of the steepest-descent method, our Algorithm \ref{alg:sm_NROP222} is exactly the same as the one in \cite{zhang2021riemannian}.

For each $(n, m)$, we generated 50 pairs of random instances and random initial points. We claim that an algorithm successfully terminates if $\|Q x_{k}\|_{0}=n$, where $x_{k}$ is the $ k $-th iteration. Here, when we count the number of nonzeros of $Q x_{k}$, we truncated the entries as
\begin{equation}
	(Q x_{k})_{i}=0 \quad \text { if } \left|(Q x_{k})_{i}\right|<\tau,
\end{equation}
where $\tau>0$ is a tolerance related to the precision of the solution, taking values from $ 10^{-5} $ to $ 10^{-12} $. Tables \ref{table_fsv_5} and \ref{table_fsv_10} report the number of successful cases out of 50 cases. Boldface highlights the best result for each row.

\begin{table}[!t]
	\small
	\caption{Number of successes from 50 pairs of random instances and random initial points for the $\ell_{1}$ minimization model (\ref{fsv}) and $n=5$.}
	\centering
	\label{table_fsv_5}
	\begin{tabular}{c|c|c|c|c|c|c|c}
		\hline
		$(n, m)$       &$ \tau $    & \multicolumn{1}{c|}{\begin{tabular}[c]{@{}c@{}}Steepest-\\ descent\end{tabular}} & \multicolumn{1}{c|}{\begin{tabular}[c]{@{}c@{}}Barzilai-\\ Borwein\end{tabular}} & \multicolumn{1}{c|}{\begin{tabular}[c]{@{}c@{}}Conjugate-\\ gradient\end{tabular}} & \multicolumn{1}{c|}{\begin{tabular}[c]{@{}c@{}}Trust-\\ regions\end{tabular}} & \multicolumn{1}{c|}{BFGS} & \multicolumn{1}{c}{ARC} \\ \hline
		$(5,20)$& $10^{-5}$ & 21  & 19  & 0     & 22  & \textbf{23}     & \textbf{23}  \\
		& $10^{-6}$ & 21  & 19  & 0     & 22  & \textbf{23}     & \textbf{23}  \\
		& $10^{-7}$ & 21  & 19  & 0     & 22  & \textbf{23}     & \textbf{23}  \\
		& $10^{-8}$ & 16  & 19  & 0     & 22  & \textbf{23}     & \textbf{23}  \\ \hline
		$(5,30)$ & $10^{-5}$ & 36  & \textbf{42}  & 0     & 34  & 36     & 35  \\
		& $10^{-6}$ & 36  & \textbf{42}  & 0     & 34  & 36     & 35  \\
		& $10^{-7}$ & 36  & \textbf{42}  & 0     & 34  & 36     & 35  \\
		& $10^{-8}$ & 34  & \textbf{42}  & 0     & 34  & 36     & 35  \\ \hline
		$(5,40)$& $10^{-5}$ & 44  & \textbf{47}  & 1     & 44  & \textbf{47}     & 45  \\
		& $10^{-6}$ & 44  & \textbf{47}  & 0     & 44  & \textbf{47}     & 45  \\
		& $10^{-7}$ & 44  & \textbf{47}  & 0     & 44  & \textbf{47}     & 45  \\
		& $10^{-8}$ & 43  & \textbf{47}  & 0     & 44  & \textbf{47}     & 45  \\ \hline
		$(5,50)$& $10^{-5}$ & \textbf{47}  & \textbf{47}  & 2     & 45  & 45     & 45  \\
		& $10^{-6}$ & \textbf{47}  & \textbf{47}  & 2     & 45  & 45     & 45  \\
		& $10^{-7}$ & \textbf{47}  & \textbf{47}  & 0     & 45  & 45     & 45  \\
		& $10^{-8}$ & 46  & \textbf{47}  & 0     & 45  & 45     & 45  \\
		\hline \hline
		$(n, m)$       &$ \tau $    & \multicolumn{1}{c|}{\begin{tabular}[c]{@{}c@{}}Steepest-\\ descent\end{tabular}} & \multicolumn{1}{c|}{\begin{tabular}[c]{@{}c@{}}Barzilai-\\ Borwein\end{tabular}} & \multicolumn{1}{c|}{\begin{tabular}[c]{@{}c@{}}Conjugate-\\ gradient\end{tabular}} & \multicolumn{1}{c|}{\begin{tabular}[c]{@{}c@{}}Trust-\\ regions\end{tabular}} & \multicolumn{1}{c|}{BFGS} & \multicolumn{1}{c}{ARC} \\ \hline
		$(5,20)$& $10^{-9}$ &	0   & 19  & 0     & 22  & \textbf{23}     & \textbf{23}  \\
		& $10^{-10}$ &	0   & 19  & 0     & 22  & \textbf{23}     & \textbf{23}  \\
		& $10^{-11}$ &	0   & 19  & 0     & 22  & \textbf{23}     & 19  \\
		& $10^{-12}$ &	0   & 18  & 0     & \textbf{22}  & \textbf{22}     & 17  \\ \hline
		$(5,30)$& $10^{-9}$ &	8   & \textbf{42} & 0     & 34  & 36     & 35  \\
		& $10^{-10}$  &	1   & \textbf{42} & 0     & 34  & 36     & 35  \\
		& $10^{-11}$  &	0   & \textbf{42} & 0     & 34  & 36     & 33  \\
		& $10^{-12}$  &	0   & \textbf{42} & 0     & 34  & 34     & 29  \\ \hline
		$(5,40)$& $10^{-9}$ &	3   & \textbf{47}  & 0     & 44  & \textbf{47}     & 45  \\
		& $10^{-10}$ &	2   & \textbf{47}  & 0     & 44  & \textbf{47}     & 45  \\
		& $10^{-11}$ &	1   & \textbf{47}  & 0     & 44  & \textbf{47}     & 44  \\
		& $10^{-12}$ &	0   & \textbf{46}  & 0     & 44  & 44     & 36  \\ \hline
		$(5,50)$& $10^{-9}$ &	5   & \textbf{47}  & 0   & 45  & 45  & 45  \\
		& $10^{-10}$ &	2   & \textbf{47}  & 0   & 45  & 45  & 45  \\
		& $10^{-11}$ &	0   & \textbf{47}  & 0   & 45  & 45  & 45  \\
		& $10^{-12}$ &	0   & \textbf{47}  & 0   & 45  & 45  & 37  \\ \hline
	\end{tabular}
\end{table}

As shown in Table \ref{table_fsv_5} and \ref{table_fsv_10}, surprisingly, the conjugate-gradient method, which performed best on the CP factorization problem in Section \ref{sec:ne}, performed worst on the FSV problem. In fact, it was almost useless. Moreover, although the steepest-descent method employed in \cite{zhang2021riemannian} was not bad at obtaining low-precision solutions with $ \tau \in \{ 10^{-5}, 10^{-6}, 10^{-7}, 10^{-8}\} $, it had difficulty obtaining high-precision solutions with $ \tau \in \{ 10^{-9}, 10^{-10}, 10^{-11}, 10^{-12}\} $. The remaining four sub-algorithms easily obtained high-precision solutions, with the Barzilai-Borwein method performing the best in most occasions. Combined with the results in Section \ref{sec:ne}, this shows that in practice, the choice of sub-algorithm in the Riemannian smoothing method (Algorithm \ref{alg:sm_NROP222}) is highly problem-dependent. For the other smoothing functions in Table \ref{table:smlist}, we obtained similar results as in Table \ref{table_fsv_5} and \ref{table_fsv_10}.

\begin{table}[!t]
	\small
	\caption{Number of successes from 50 pairs of random instances and random initial points for the $\ell_{1}$ minimization model (\ref{fsv}) and $n=10$.}
	\centering
	\label{table_fsv_10}
	\begin{tabular}{c|c|c|c|c|c|c|c}
		\hline
		$(n, m)$ &$ \tau $    & \multicolumn{1}{c|}{\begin{tabular}[c]{@{}c@{}}Steepest-\\ descent\end{tabular}} & \multicolumn{1}{c|}{\begin{tabular}[c]{@{}c@{}}Barzilai-\\ Borwein\end{tabular}} & \multicolumn{1}{c|}{\begin{tabular}[c]{@{}c@{}}Conjugate-\\ gradient\end{tabular}} & \multicolumn{1}{c|}{\begin{tabular}[c]{@{}c@{}}Trust-\\ regions\end{tabular}} & \multicolumn{1}{c|}{BFGS} & \multicolumn{1}{c}{ARC} \\ \hline
		$(10,60)$& $10^{-5}$ & 24 & \textbf{28} & 0 & \textbf{28} & \textbf{28} & 25 \\
				& $10^{-6}$ & 24 & \textbf{28} & 0 & \textbf{28} & \textbf{28} & 25 \\
				& $10^{-7}$ & 24 & \textbf{28} & 0 & \textbf{28} & \textbf{28} & 25 \\
				& $10^{-8}$ & 23 & \textbf{28} & 0 & \textbf{28} & \textbf{28} & 25 \\ \hline
		$(10,80)$& $10^{-5}$ & 39 & 37 &  1 & \textbf{40} & 39 & \textbf{40} \\
				& $10^{-6}$ & 39 & 37 &  0 & \textbf{40} & 39 & \textbf{40} \\
				& $10^{-7}$ & 39 & 37 &  0 & \textbf{40} & 39 & \textbf{40} \\
				& $10^{-8}$ & 39 & 37 &  0 & \textbf{40} & 39 & \textbf{40} \\ \hline
		$(10,100)$& $10^{-5}$ & 45 & \textbf{48} &  3 & 45 & 43 & 41 \\
				& $10^{-6}$ & 45 & \textbf{48} &  0 & 45 & 43 & 41 \\
				& $10^{-7}$ & 45 & \textbf{48} &  0 & 45 & 43 & 41 \\
				& $10^{-8}$ & 45 & \textbf{48} &  0 & 45 & 43 & 41 \\ \hline
		$(10,120)$& $10^{-5}$ & 44 & \textbf{46} &  1   & 44  & 44    & 43  \\
				& $10^{-6}$ & 44 & \textbf{46} &   0   & 44  & 44    & 43  \\
				& $10^{-7}$ & 44 & \textbf{46} &   0   & 44  & 44    & 43  \\
				& $10^{-8}$ & 44 & \textbf{46} &   0   & 44  & 44    & 43  \\
		\hline \hline
		$(n, m)$       &$ \tau $    & \multicolumn{1}{c|}{\begin{tabular}[c]{@{}c@{}}Steepest-\\ descent\end{tabular}} & \multicolumn{1}{c|}{\begin{tabular}[c]{@{}c@{}}Barzilai-\\ Borwein\end{tabular}} & \multicolumn{1}{c|}{\begin{tabular}[c]{@{}c@{}}Conjugate-\\ gradient\end{tabular}} & \multicolumn{1}{c|}{\begin{tabular}[c]{@{}c@{}}Trust-\\ regions\end{tabular}} & \multicolumn{1}{c|}{BFGS} & \multicolumn{1}{c}{ARC} \\ \hline
		$(10,60)$& $10^{-9}$ &	3   & \textbf{28}  & 0     & \textbf{28}  & \textbf{28}     & 25  \\
				& $10^{-10}$ &	0   & \textbf{28}  & 0     & \textbf{28}  & \textbf{28}     & 25  \\
				& $10^{-11}$ &	0   & \textbf{28}  & 0     & \textbf{28}  & \textbf{28}     & 22  \\
				& $10^{-12}$ &	0   & \textbf{28}  & 0     & \textbf{28}  & 27     & 12  \\ \hline
		$(10,80)$& $10^{-9}$ &	5   & 37 & 0     & \textbf{40}  & 39     & \textbf{40}  \\
				& $10^{-10}$ &	0  & 37 & 0     & \textbf{40}  & 39     & \textbf{40}  \\
				& $10^{-11}$ &	0  & 37 & 0     & \textbf{40}  & 39     & 39  \\
				& $10^{-12}$ &	0  & 37 & 0     & \textbf{40}  & 37     & 30  \\ \hline
		$(10,100)$& $10^{-9}$ &	13  & \textbf{48}  & 0     & 45  & 43     & 41  \\
				& $10^{-10}$ &	2   & \textbf{48}  & 0     & 45  & 43     & 41  \\
				& $10^{-11}$ &	0   & \textbf{48}  & 0     & 45  & 43     & 40  \\
				& $10^{-12}$ &	0   & \textbf{48}  & 0     & 45  & 43     & 37  \\ \hline
		$(10,120)$& $10^{-9}$ & 14 & \textbf{46} &   0    & 44  & 44     & 43  \\
				& $10^{-10}$ & 0 & \textbf{46} &   0    & 44  & 44     & 43  \\
				& $10^{-11}$ & 0 & \textbf{46} &   0    & 44  & 44     & 43  \\
				& $10^{-12}$ & 0 & \textbf{46} &   0    & 44  & 43     & 40  \\ \hline
	\end{tabular}
\end{table}

\begin{figure}[!t]\centering
	\begin{subfigure}{0.24\textwidth} 
		\includegraphics[width=\textwidth]{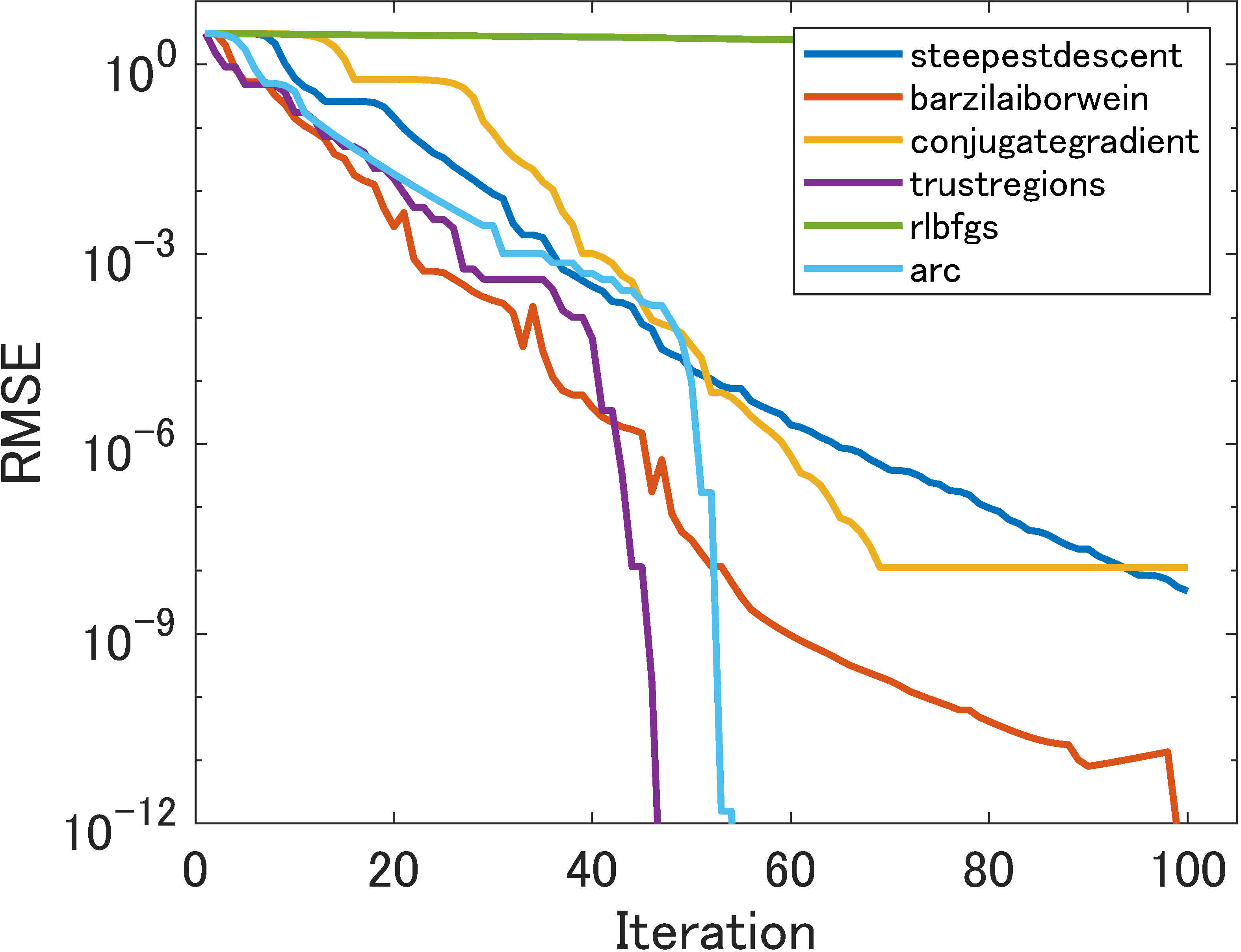}
		\caption{$ \tilde{f}_{1} $}
	\end{subfigure}
	\begin{subfigure}{0.24\textwidth}
		\includegraphics[width=\textwidth]{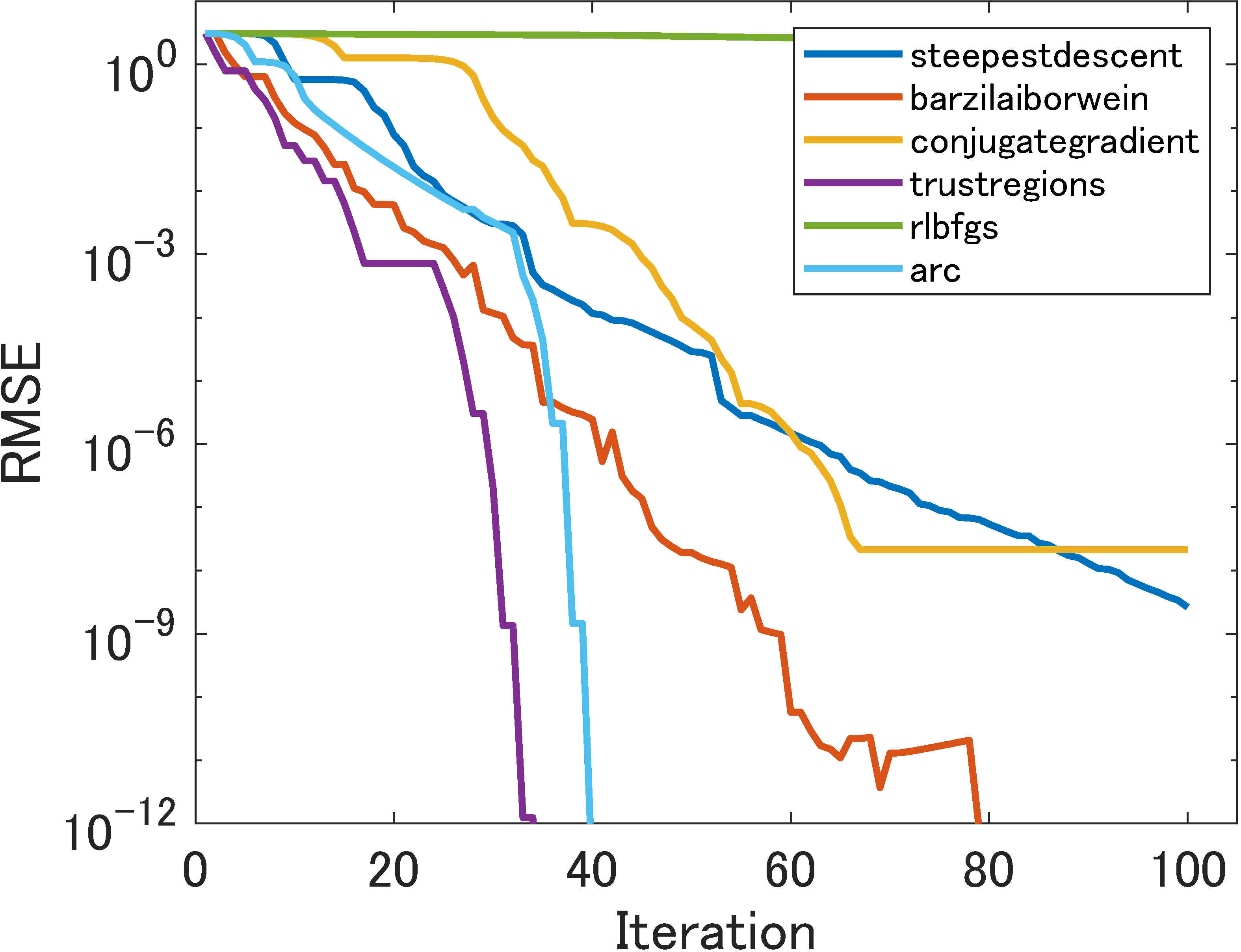}
		\caption{$ \tilde{f}_{2} $}
	\end{subfigure}
	\begin{subfigure}{0.24\textwidth}
		\includegraphics[width=\textwidth]{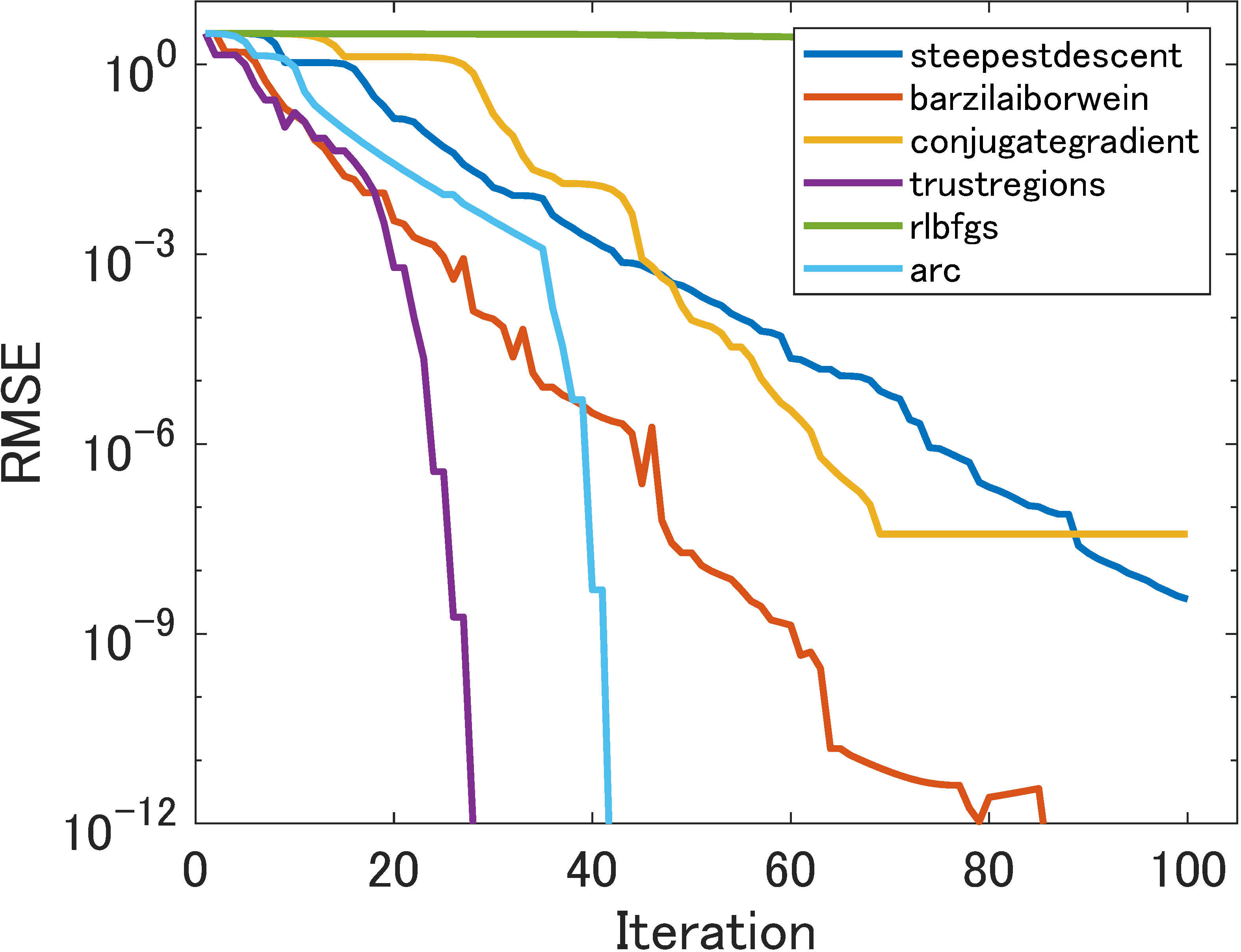}
		\caption{$ \tilde{f}_{3} $}
	\end{subfigure}
	\begin{subfigure}{0.24\textwidth}
		\includegraphics[width=\textwidth]{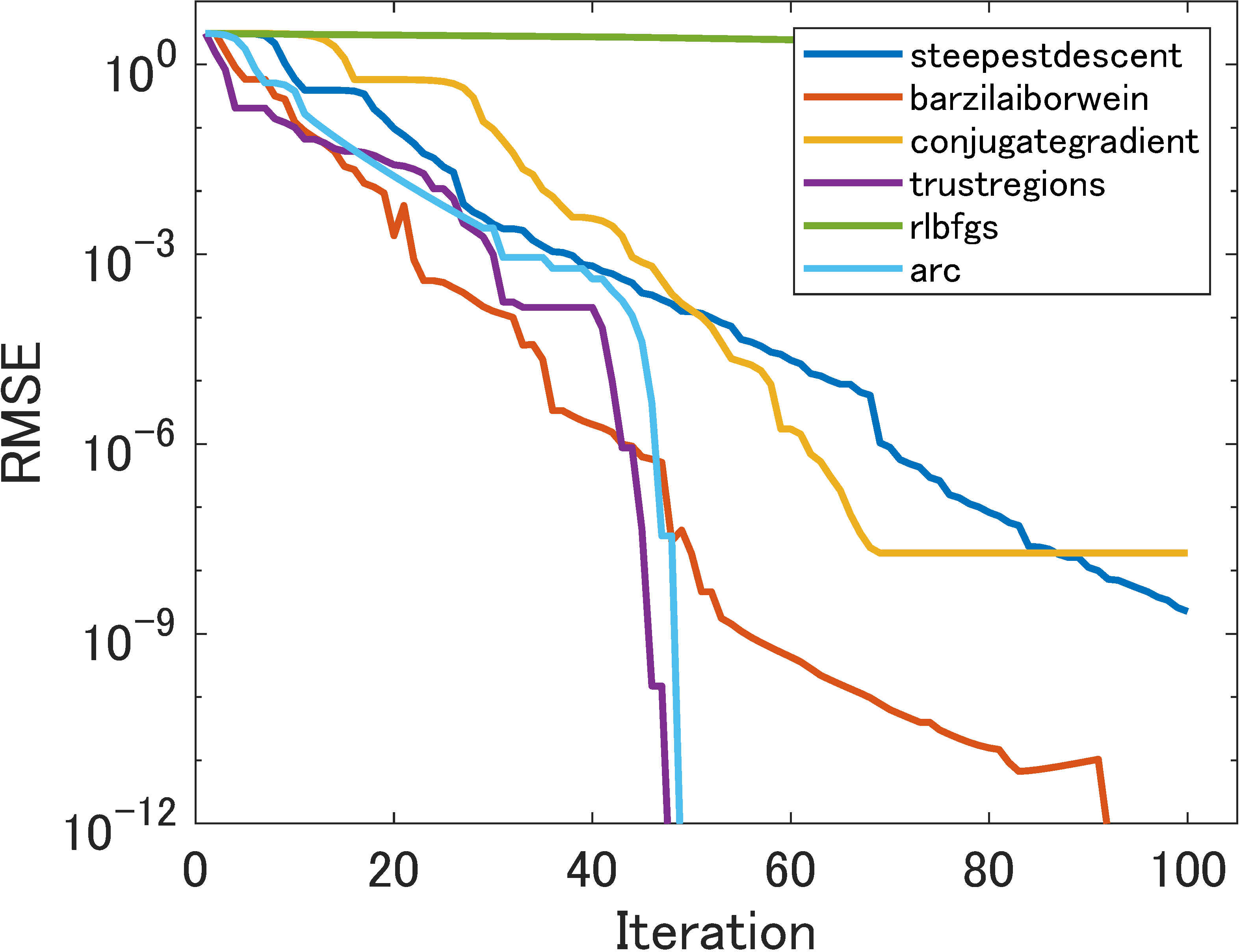}
		\caption{$ \tilde{f}_{4} $}
	\end{subfigure}
	\hfill
	\begin{subfigure}{0.24\textwidth}
		\includegraphics[width=\textwidth]{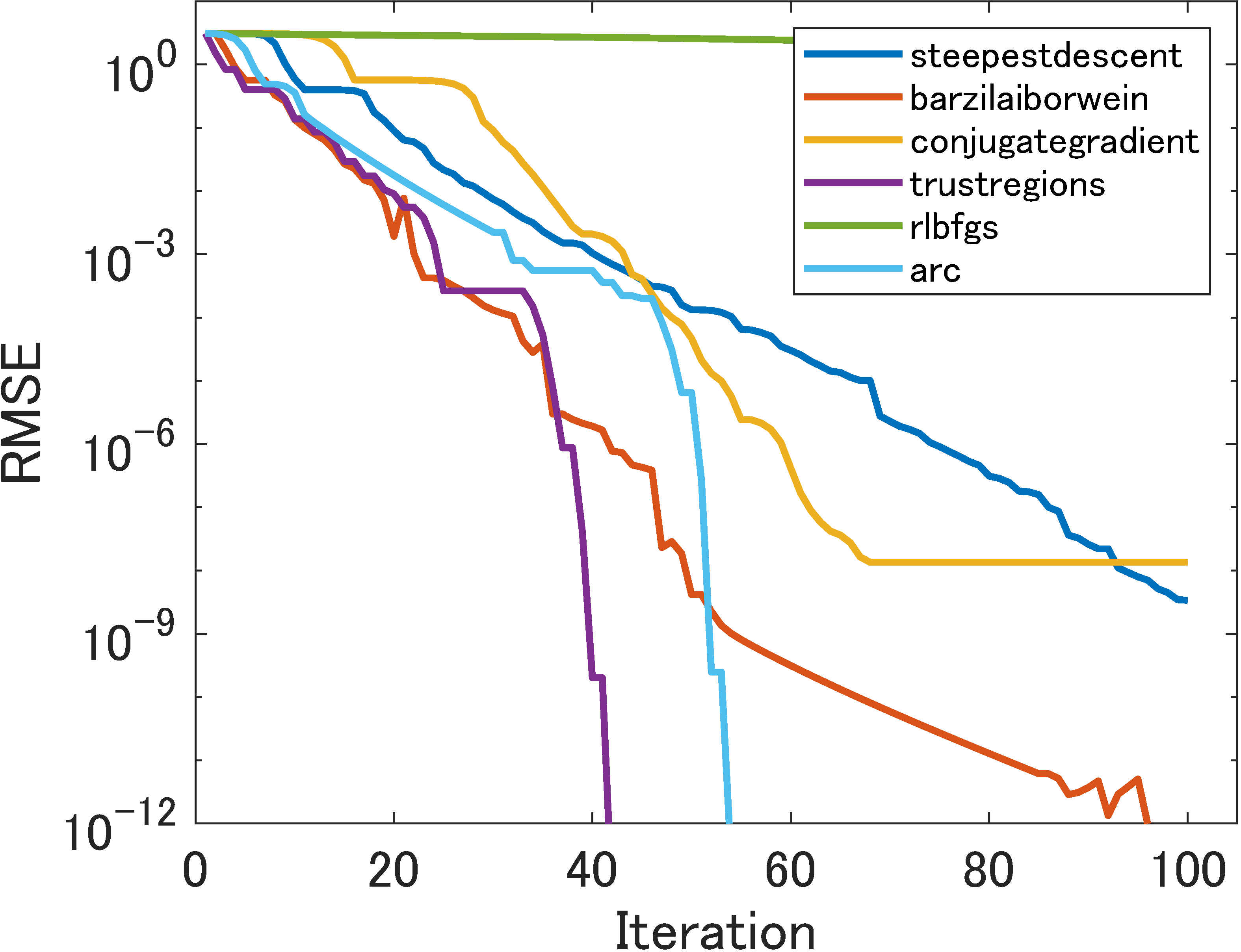}
		\caption{$ \tilde{f}_{5} $}
	\end{subfigure}
	\begin{subfigure}{0.24\textwidth} 
		\includegraphics[width=\textwidth]{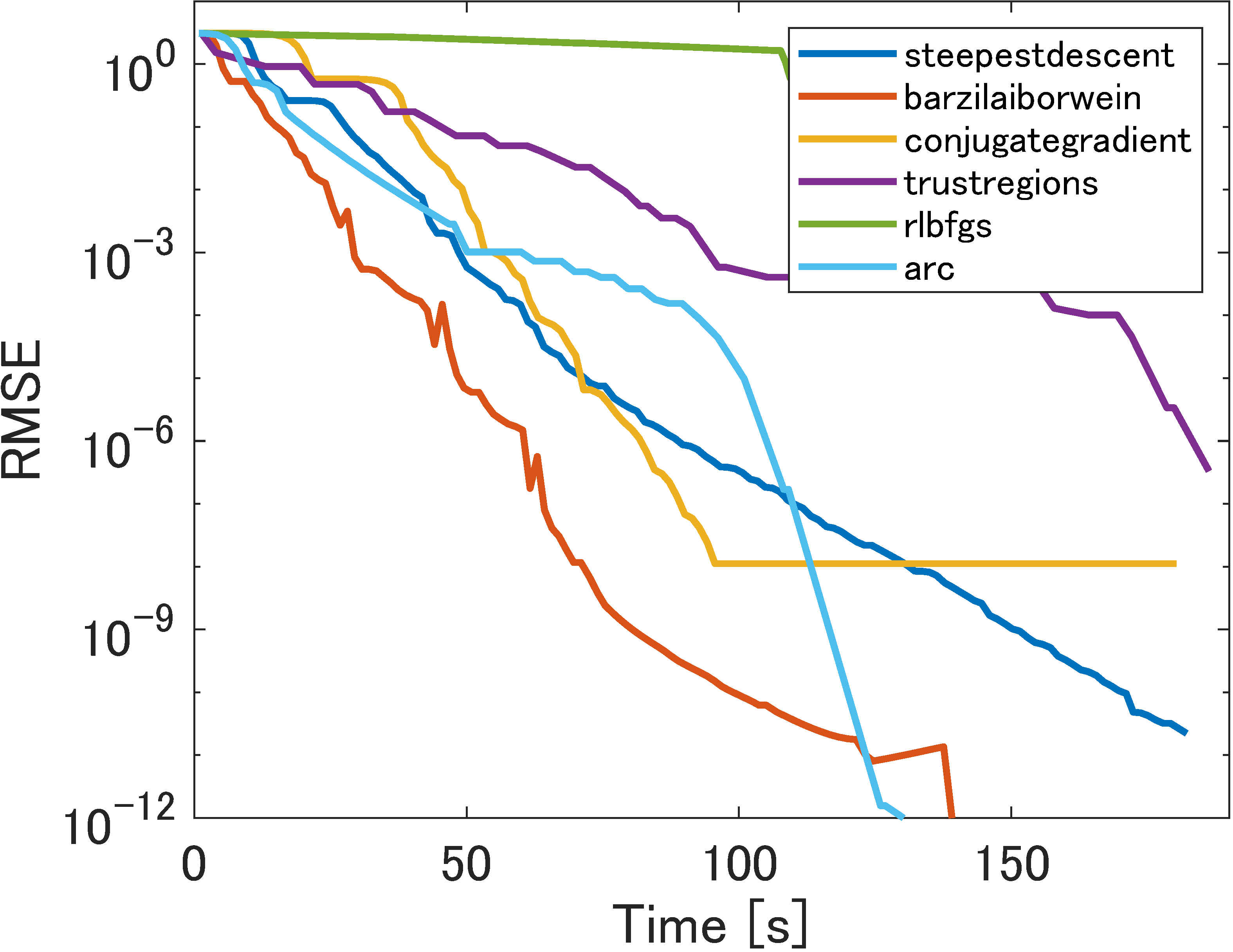}
		\caption{$ \tilde{f}_{1} $}
	\end{subfigure}
	\begin{subfigure}{0.24\textwidth}
		\includegraphics[width=\textwidth]{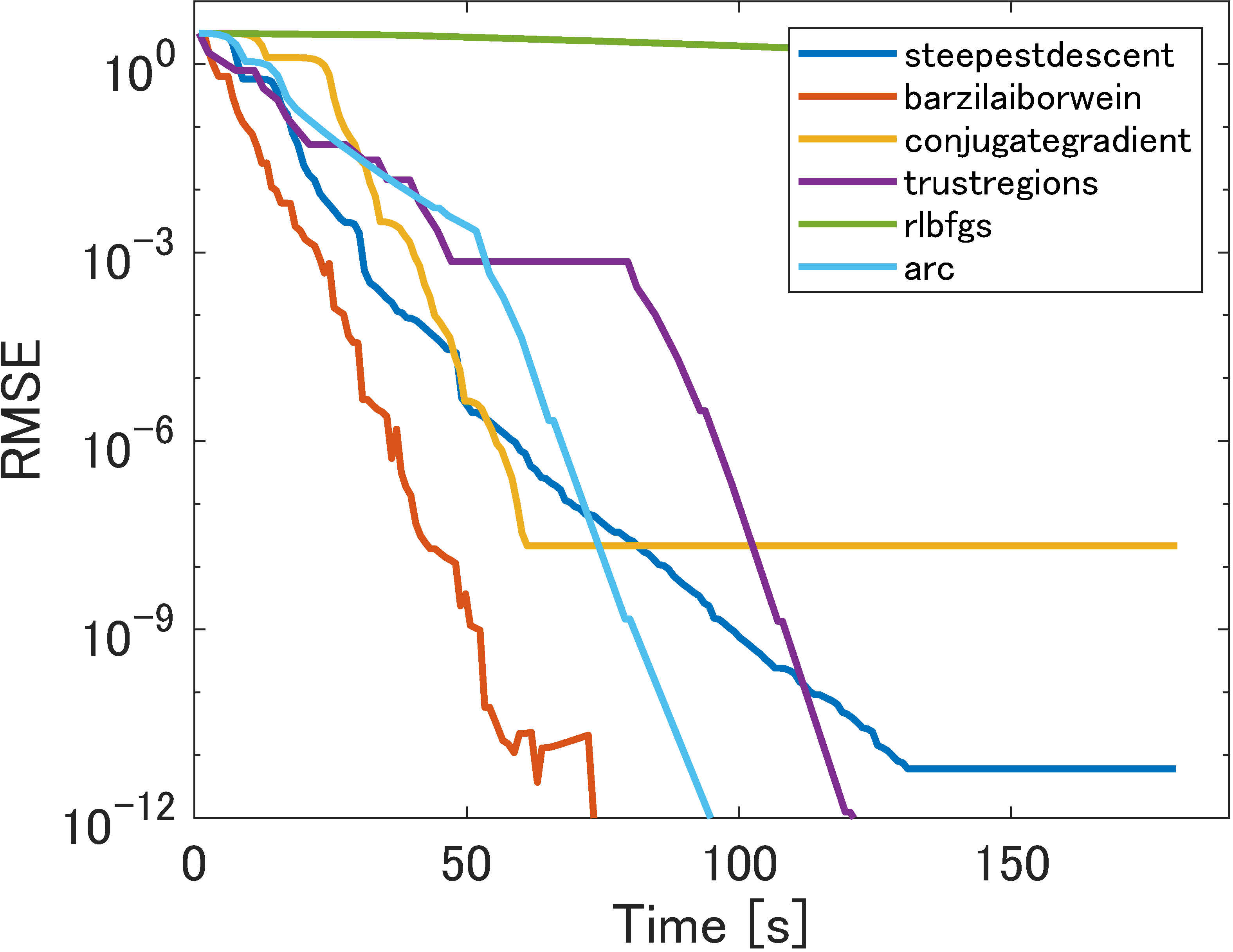}
		\caption{$ \tilde{f}_{2} $}
	\end{subfigure}
	\begin{subfigure}{0.24\textwidth}
		\includegraphics[width=\textwidth]{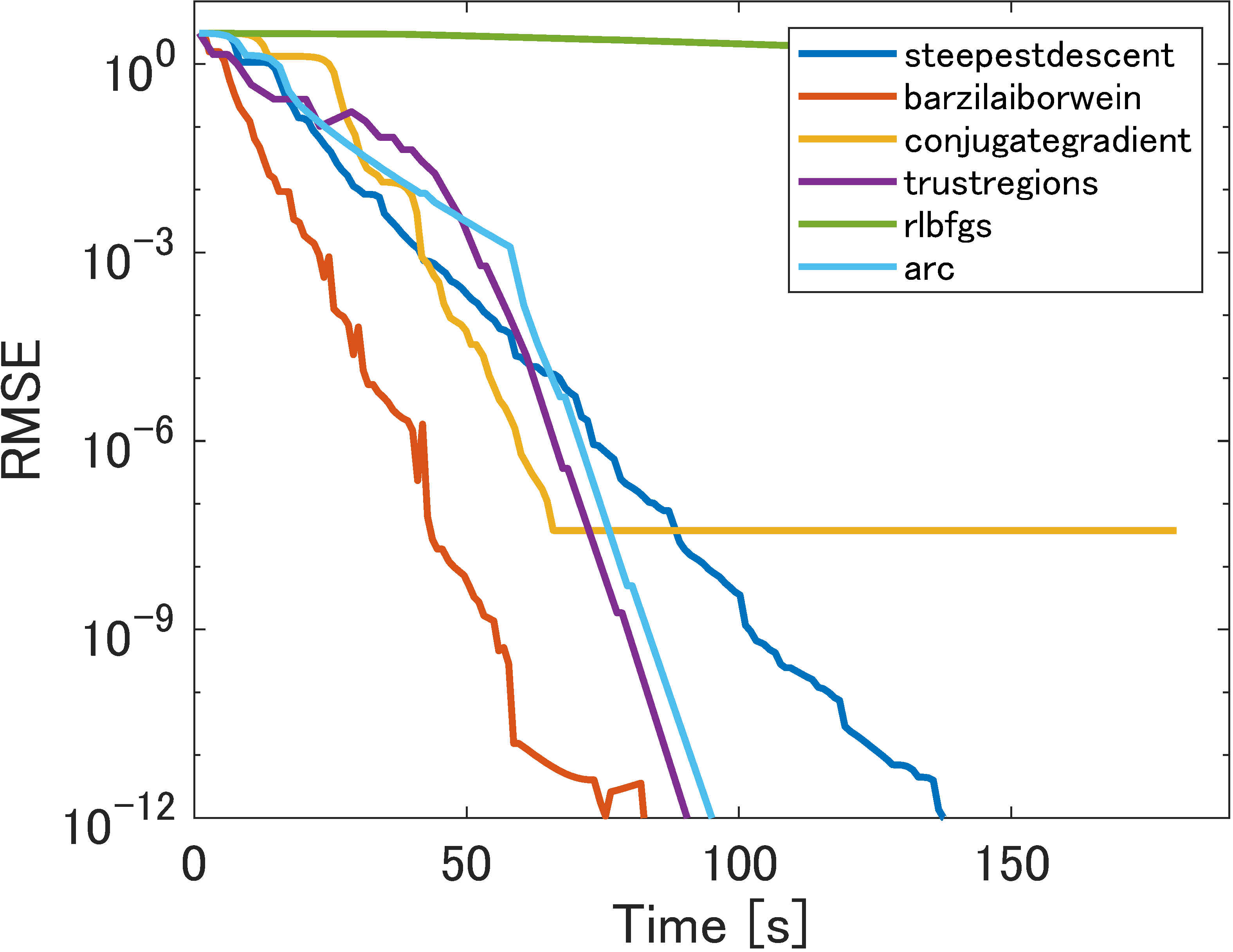}
		\caption{$ \tilde{f}_{3} $}
	\end{subfigure}
	\hfill
	\begin{subfigure}{0.24\textwidth}
		\includegraphics[width=\textwidth]{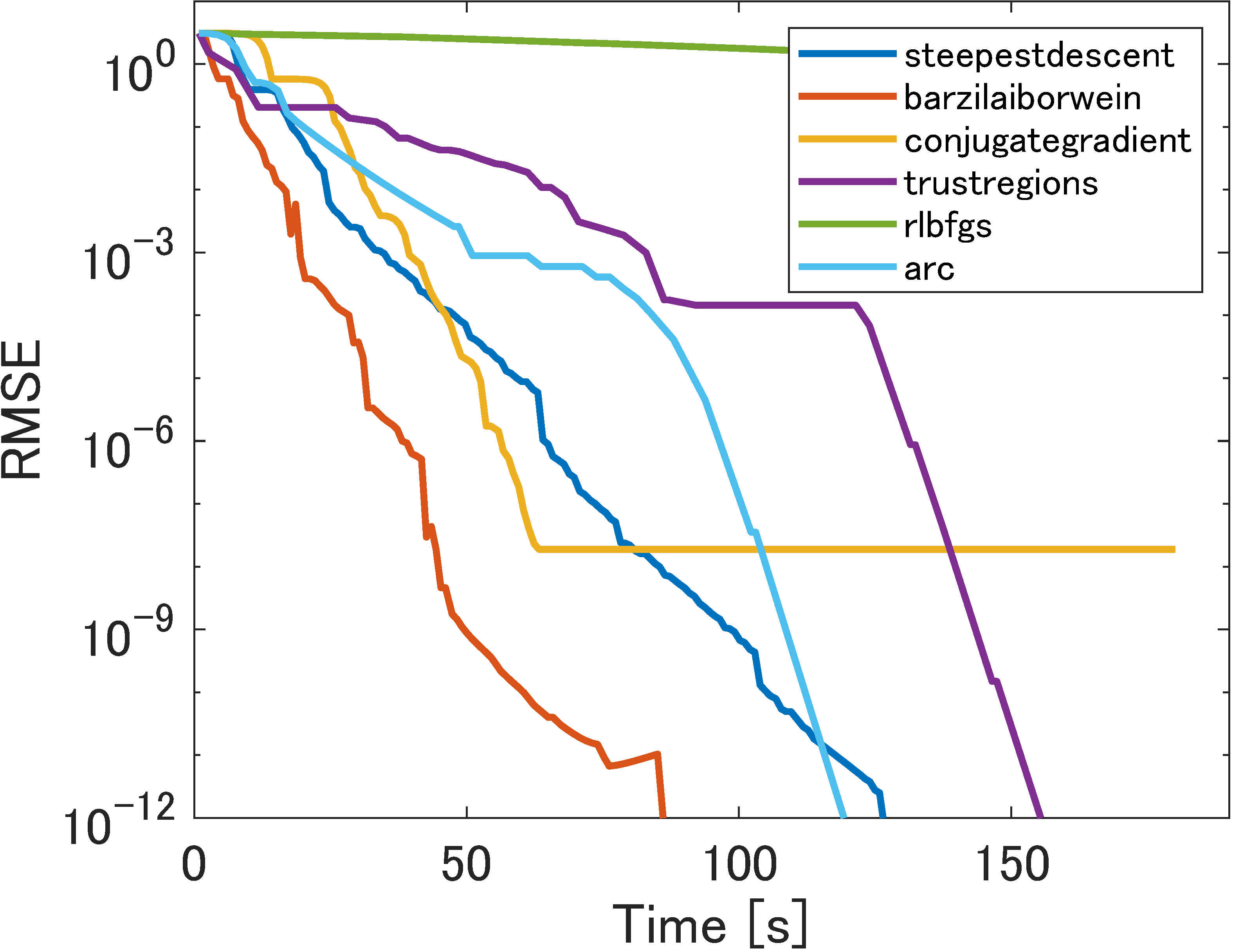}
		\caption{$ \tilde{f}_{4} $}
	\end{subfigure}
	\begin{subfigure}{0.24\textwidth}
		\includegraphics[width=\textwidth]{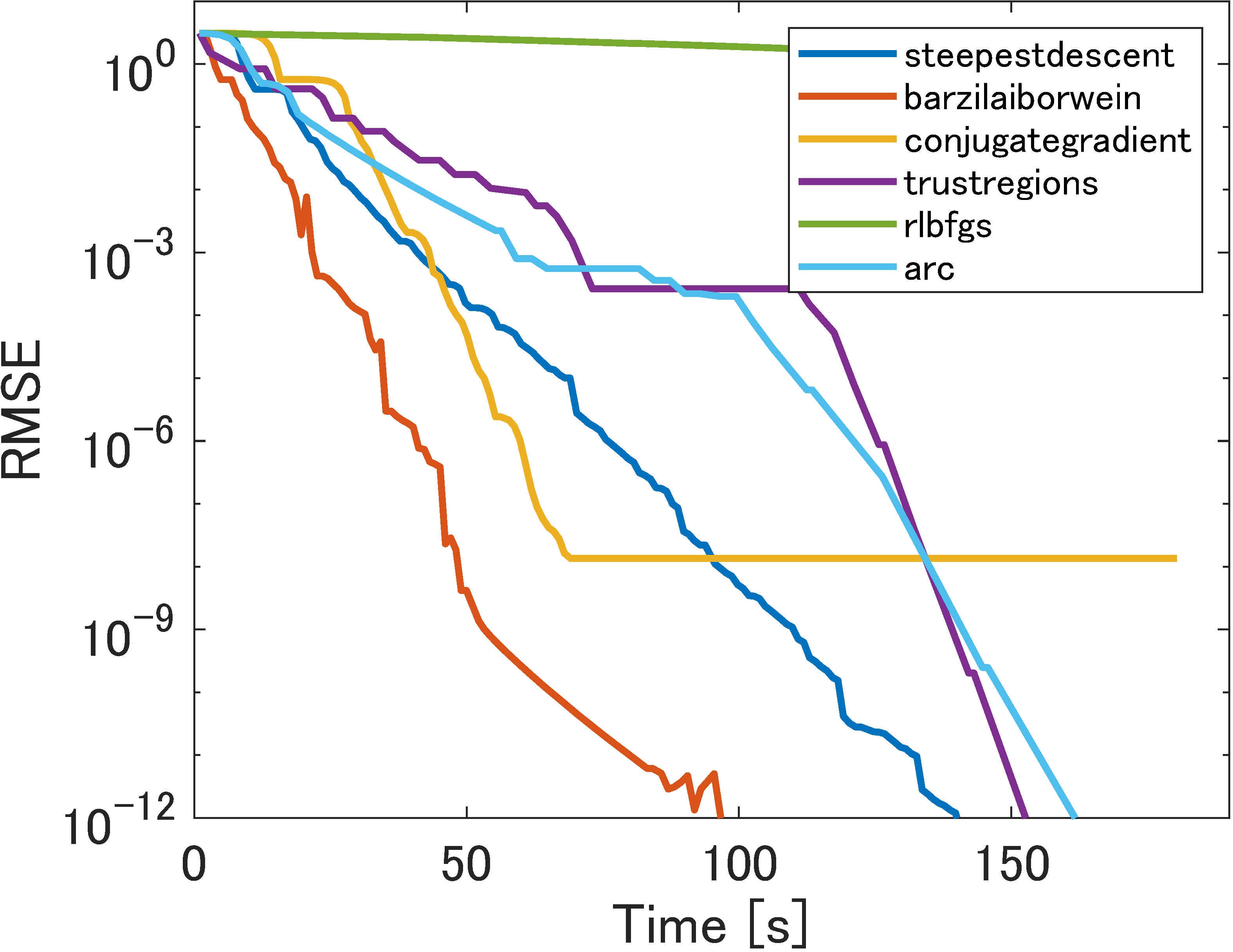}
		\caption{$ \tilde{f}_{5} $}
	\end{subfigure}
	\caption{Perfect low-rank matrix completion of a rank-10 $5000 \times 5000$ matrix without any outliers using different smoothing functions in Table \ref{table:smlist}. (a)--(e) comprise the running iteration comparison; (f)--(j) comprise the time comparison.}
	\label{fig:perfect_5000}
\end{figure}

\subsection{Robust low-rank matrix completion}

Low-rank matrix completion consists of recovering a rank-$r$ matrix $M$ of size $m \times n$ from only a fraction of its entries with $r \ll \min (m, n)$. 
The situation in \emph{robust} low-rank matrix completion is one where only a few observed entries, called \emph{outliers}, are perturbed, i.e.,
\begin{equation*}
	M=M_{0}+S,
\end{equation*}
where $M_{0}$ is the unperturbed original data matrix of rank $r$ and $S$ is a sparse matrix. 
This is a case of adding non-Gaussian noise for which the traditional $\ell_{2}$ minimization model,
\begin{equation*}
	\min _{X \in \mathcal{M}_{r}}\left\|\mathcal{P}_{\Omega}(X-M)\right\|_{2}
\end{equation*}
is not well suited to recovery of $ M_{0} $. 
Here, $\mathcal{M}_{r}:=\left\{X \in \mathbb{R}^{m \times n} \mid \operatorname{rank}(X)=r\right\}$ is a fixed rank manifold, $\Omega$ denotes the set of indices of observed entries, and $\mathcal{P}_{\Omega}: \mathbb{R}^{m \times n} \to \mathbb{R}^{m \times n}$ is the projection onto $\Omega$, defined as
\begin{equation*}
	Z_{ij} \stackrel{P_{\Omega}}{\longmapsto} 
	\begin{cases}
		Z_{ij} & \text { if }(i, j) \in \Omega \\
		     0 & \text { if }(i, j) \notin \Omega.
	\end{cases}
\end{equation*}
In \cite{cambier2016robust}, the authors try to solve
\begin{equation*}
	\min _{X \in \mathcal{M}_{r}}\left\|\mathcal{P}_{\Omega}(X-M)\right\|_{1},
\end{equation*}
because the sparsity-inducing property of the $ \ell_{1} $ norm leads one to expect exact recovery when the noise consists of just a few outliers.

In all of the experiments, the problems were generated in the same way as in \cite{cambier2016robust}. 
In particular, after picking the values of $m, n, r$, we generated the ground truth $U \in \mathbb{R}^{m \times r}$, $V \in \mathbb{R}^{n \times r}$ with independent and identically distributed (i.i.d.) Gaussian entries of zero mean and unit variance and $M:=U V^{\top}$. 
We then sampled $k:=$ $\rho r(m+n-r)$ observed entries uniformly at random, where $\rho$ is the oversampling factor. 
In our experiments, we set $\rho=5$ throughout. We chose the initial points $X_{0}$ by using the rank-$r$ truncated SVD decomposition of $\mathcal{P}_{\Omega}(M)$.

Regarding the setting of our Algorithm \ref{alg:sm_NROP222}, we tested all combinations of the five smoothing functions in Table \ref{table:smlist} and six sub-algorithms mentioned before (30 cases in total). We set $\mu_{0}=100 $ and chose an aggressive value of $\theta=0.05$ for reducing $ \mu $, as in \cite{cambier2016robust}. 
The stopping criterion of the loop of the sub-algorithm was set to a maximum of 40 iterations or the gradient tolerance (\ref{par:wo}), whichever was reached first. 
We monitored the iterations $X_{k}$ through the root mean square error (RMSE), which is defined as the error on all the entries between $X_{k}$ and the original low-rank matrix $M_{0}$, i.e.,
\begin{equation*}
	\operatorname{RMSE}\left(X_{k}, M_{0}\right):=\sqrt{\frac{\sum_{i=1}^{m}\sum_{j=1}^{n}\left(X_{k, i j}-M_{0, i j}\right)^{2}}{m n}}.
\end{equation*}

\begin{figure}[!b]\centering
	\begin{subfigure}{0.24\textwidth} 	
		\includegraphics[width=\textwidth]{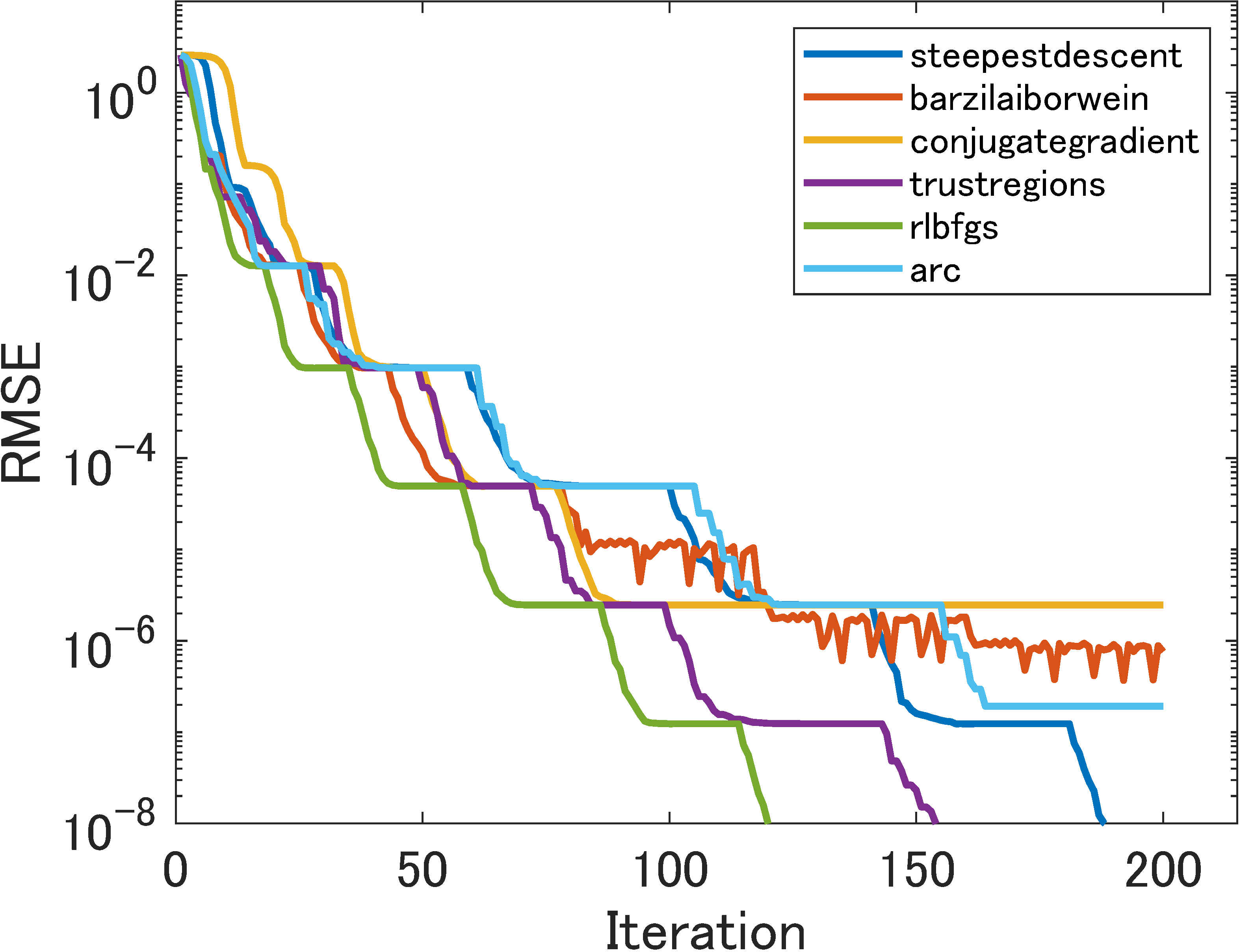}
		\caption{$ \tilde{f}_{1} $}
	\end{subfigure}
	\begin{subfigure}{0.24\textwidth}
		\includegraphics[width=\textwidth]{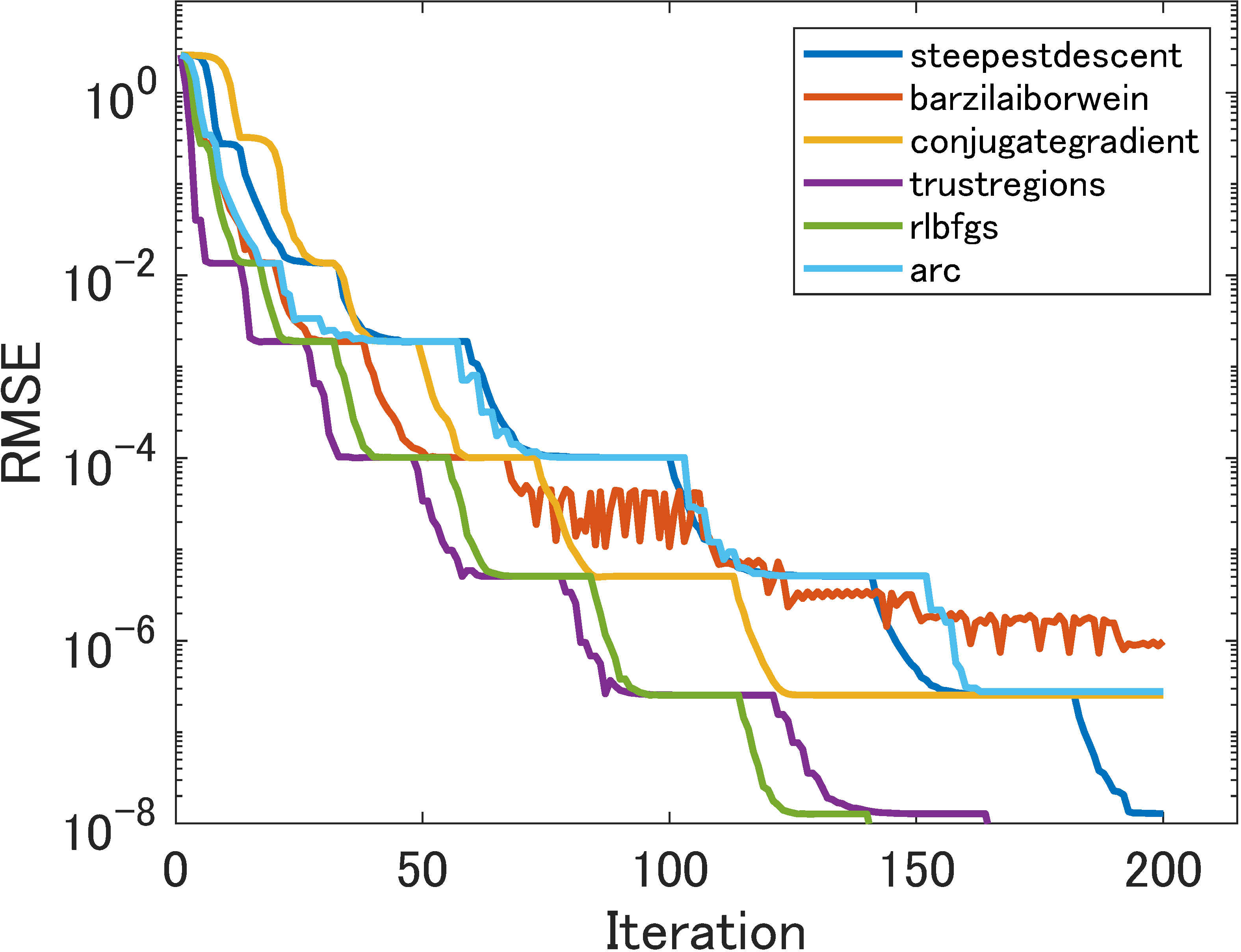}
		\caption{$ \tilde{f}_{2} $}
	\end{subfigure}
	\begin{subfigure}{0.24\textwidth}
		\includegraphics[width=\textwidth]{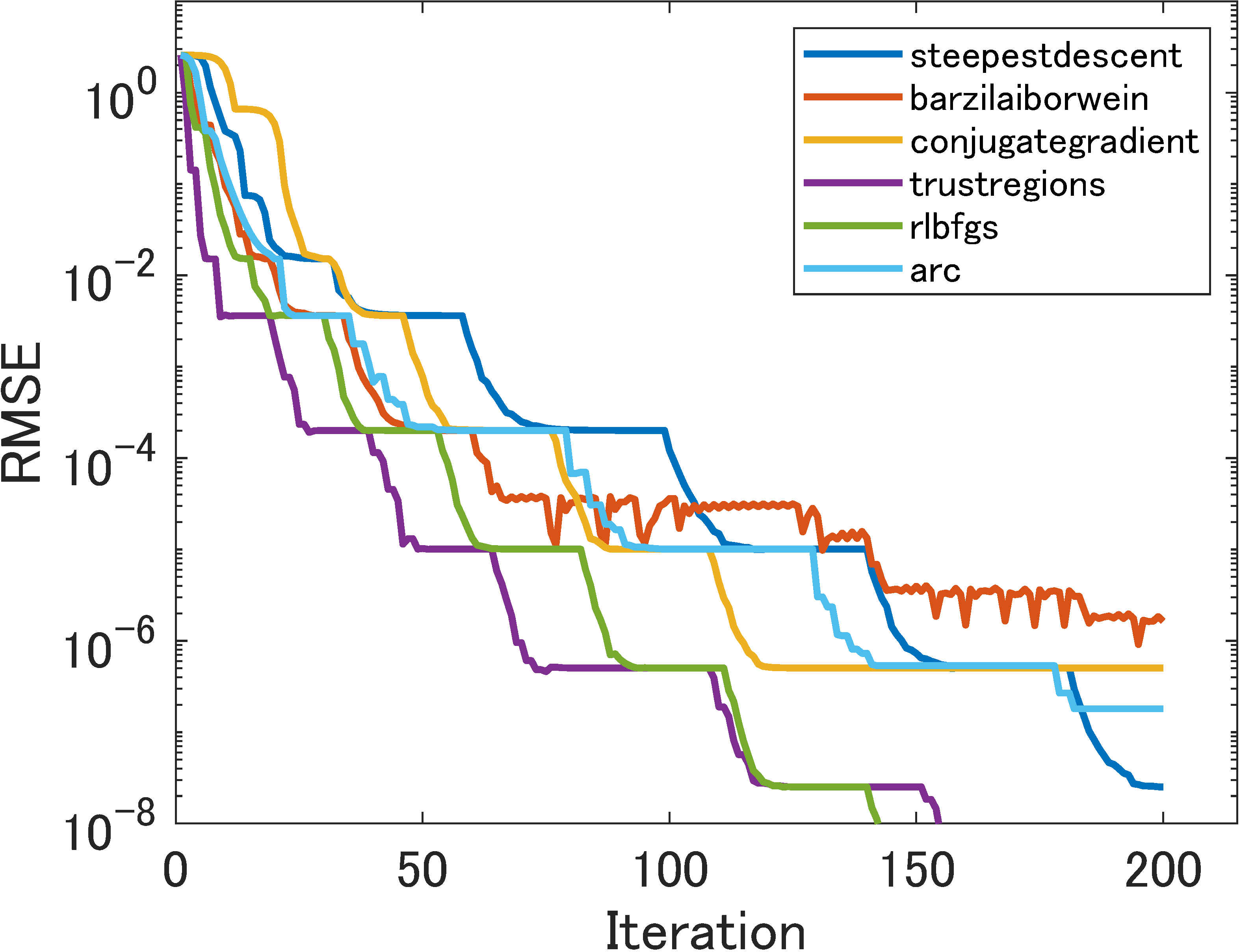}
		\caption{$ \tilde{f}_{3} $}
	\end{subfigure}
	\begin{subfigure}{0.24\textwidth}
		\includegraphics[width=\textwidth]{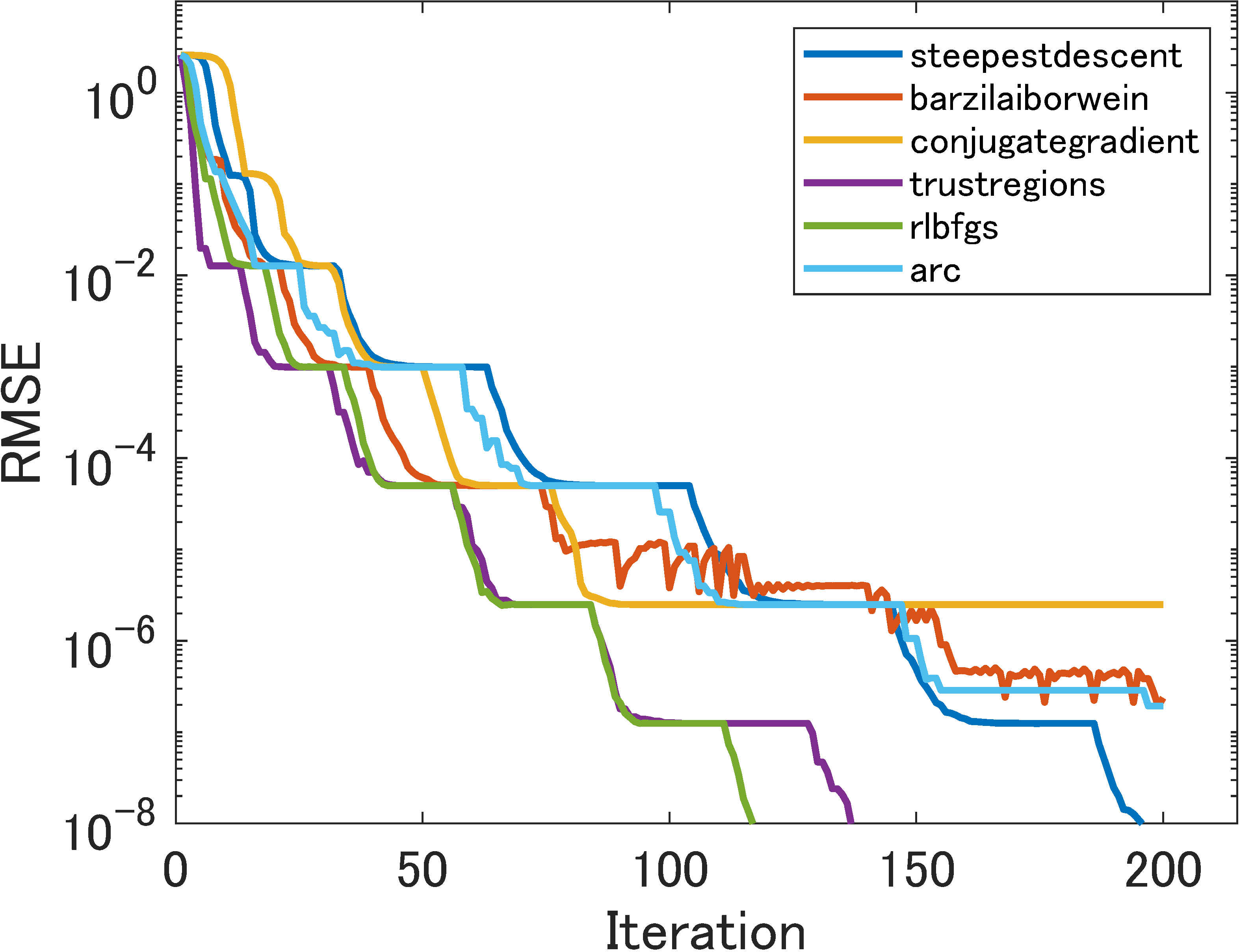}
		\caption{$ \tilde{f}_{4} $}
	\end{subfigure}
	\hfill 
	\begin{subfigure}{0.24\textwidth}
		\includegraphics[width=\textwidth]{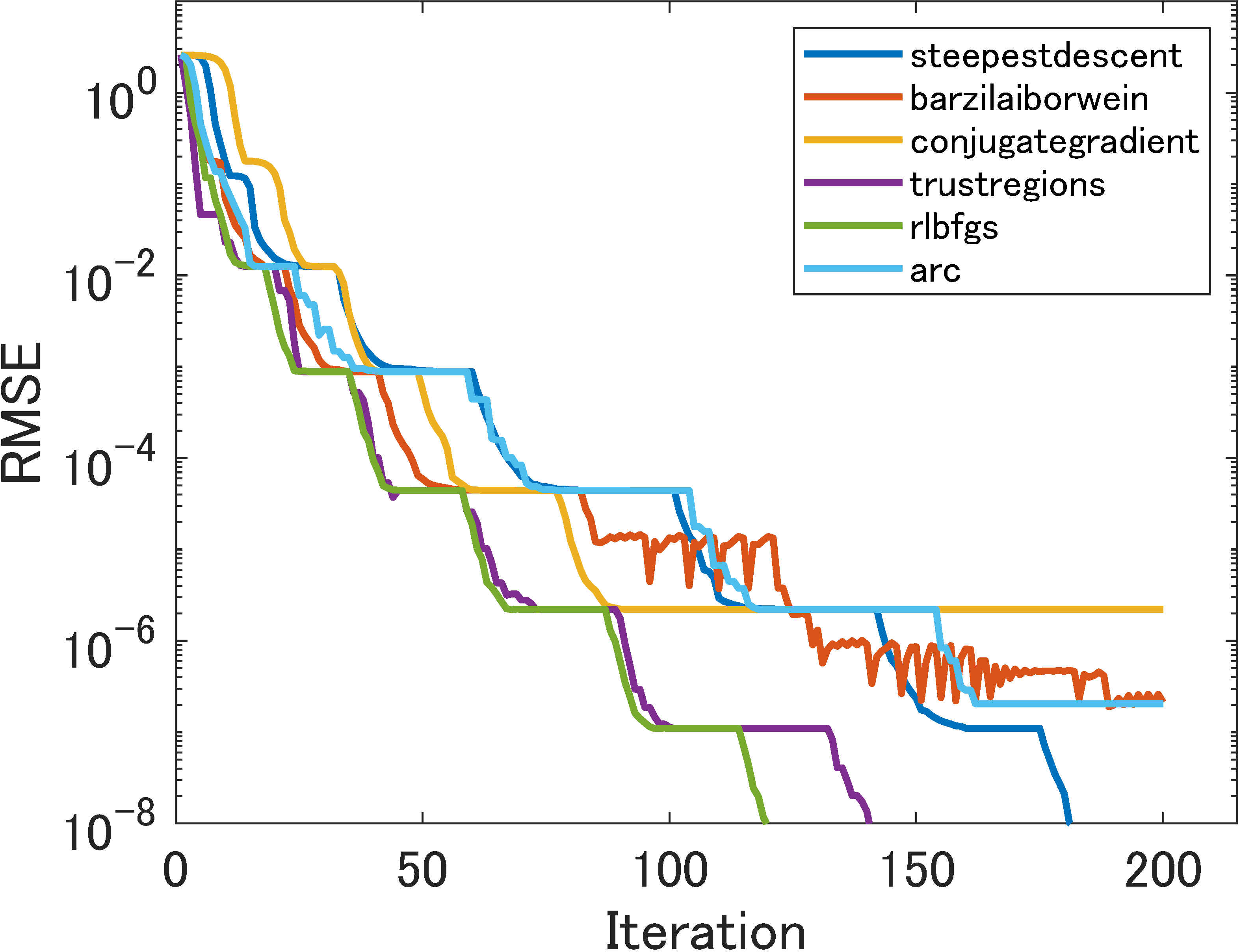}
		\caption{$ \tilde{f}_{5} $}
	\end{subfigure}
	\begin{subfigure}{0.24\textwidth}  
		\includegraphics[width=\textwidth]{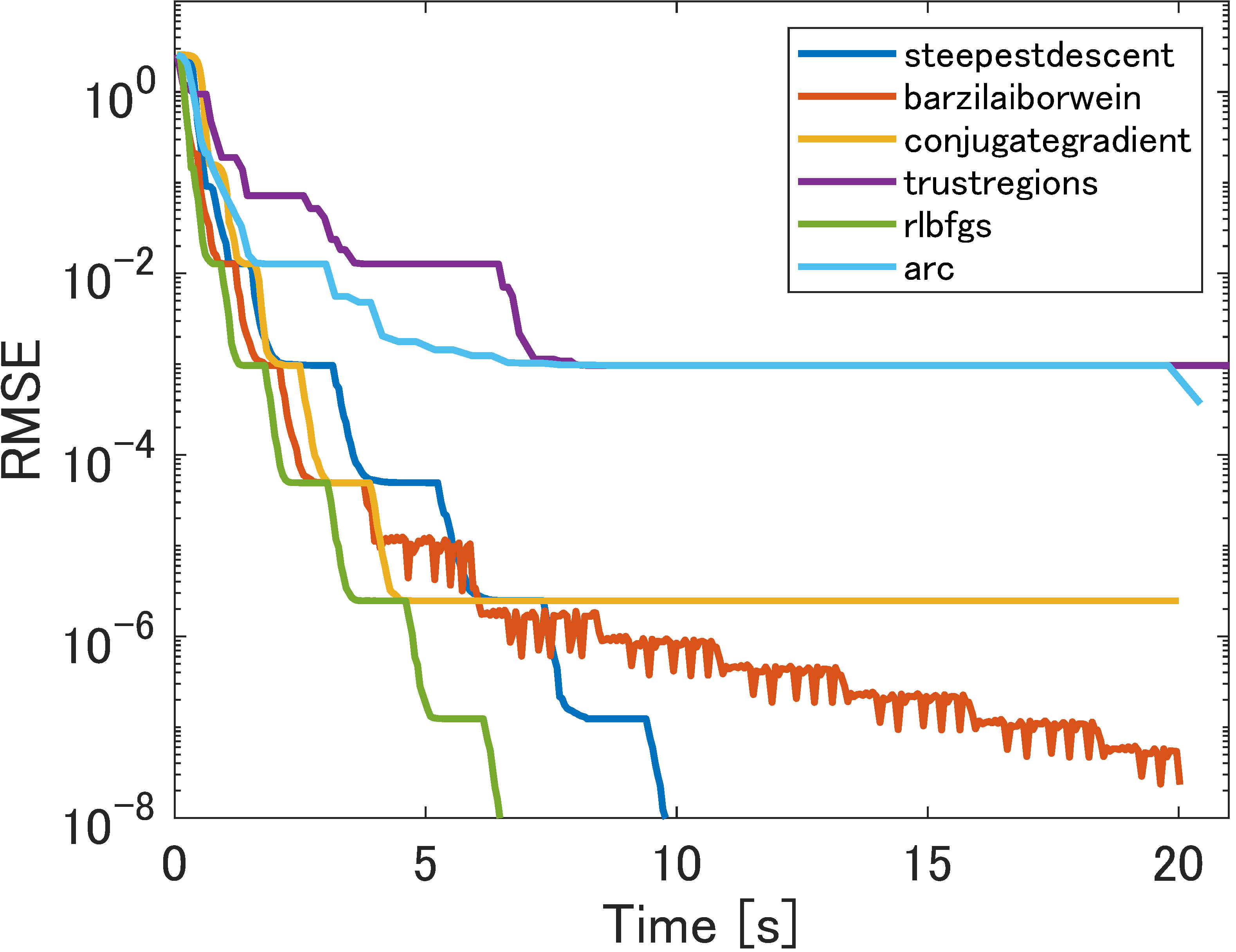}
		\caption{$ \tilde{f}_{1} $}
	\end{subfigure}
	\begin{subfigure}{0.24\textwidth}
		\includegraphics[width=\textwidth]{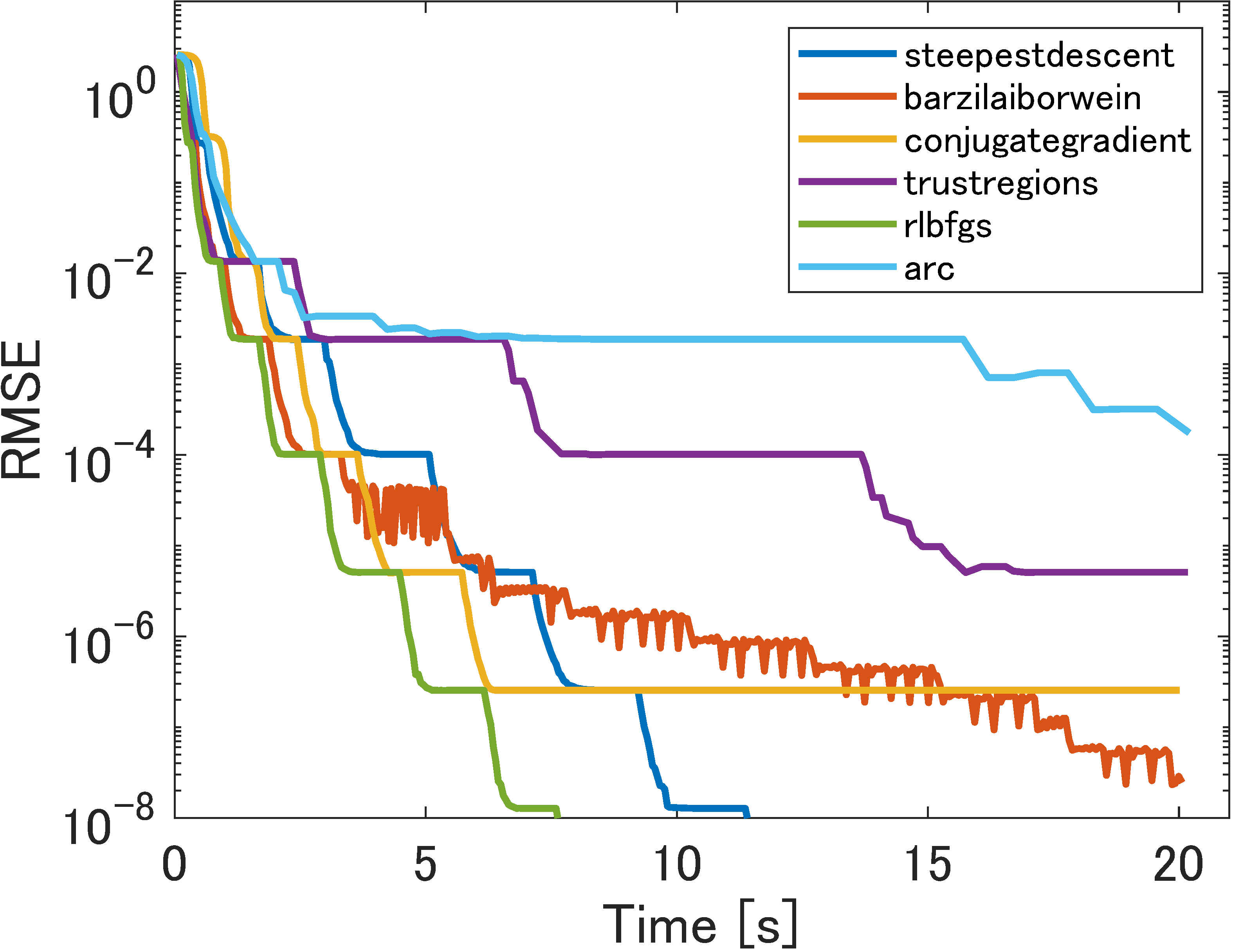}
		\caption{$ \tilde{f}_{2} $}
	\end{subfigure}
	\begin{subfigure}{0.24\textwidth}
		\includegraphics[width=\textwidth]{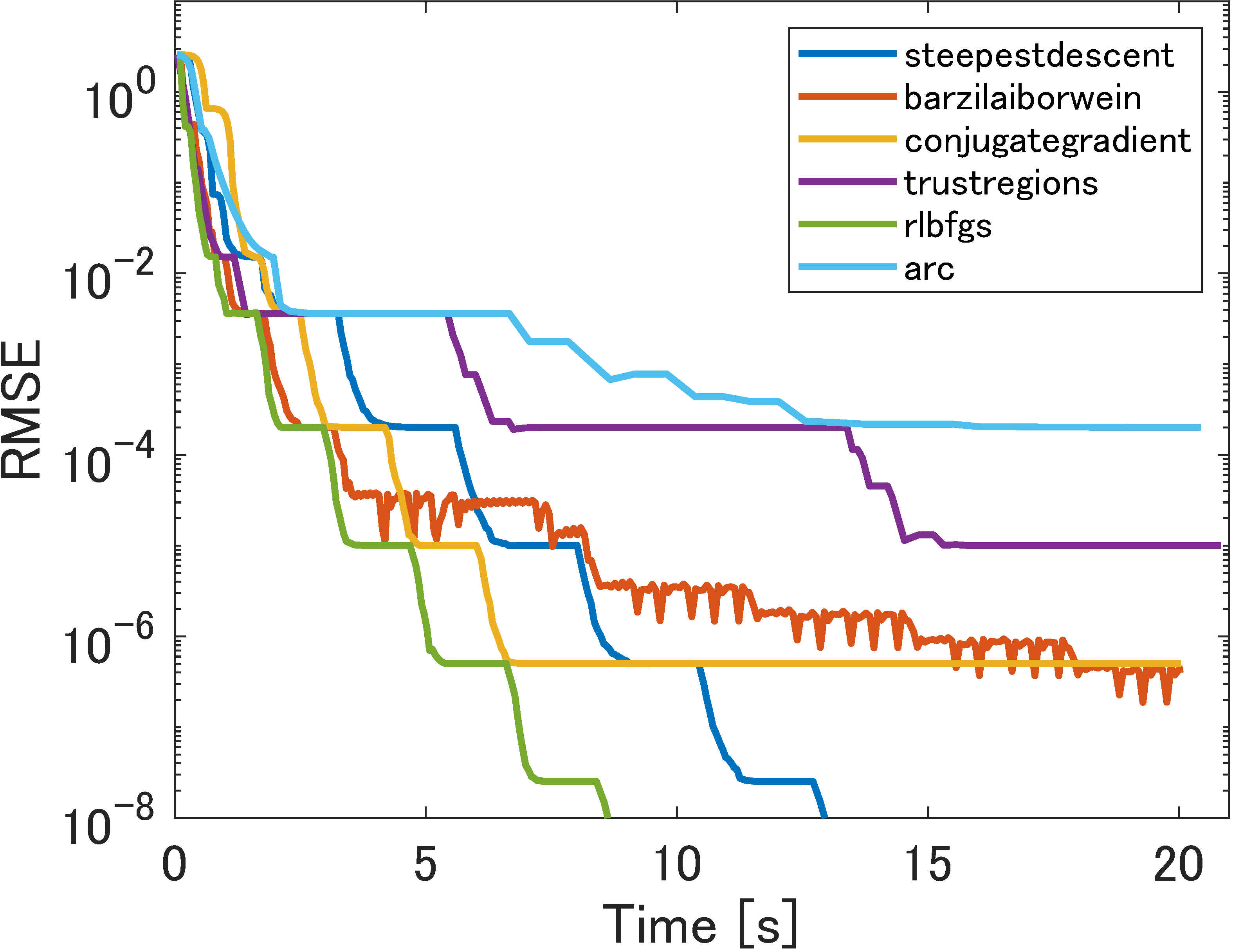}
		\caption{$ \tilde{f}_{3} $}
	\end{subfigure}
	\hfill
	\begin{subfigure}{0.24\textwidth}
		\includegraphics[width=\textwidth]{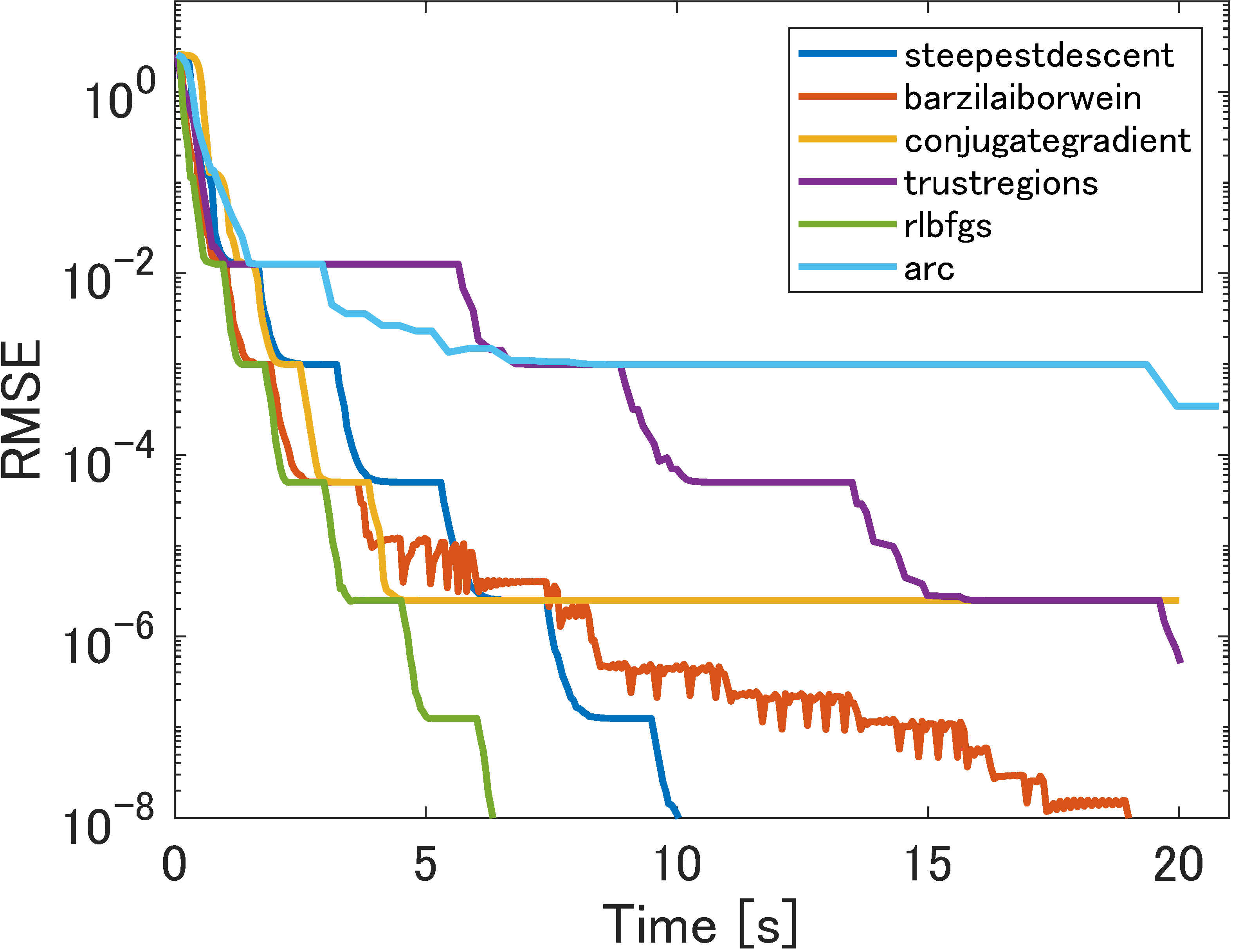}
		\caption{$ \tilde{f}_{4} $}
	\end{subfigure}
	\begin{subfigure}{0.24\textwidth}
		\includegraphics[width=\textwidth]{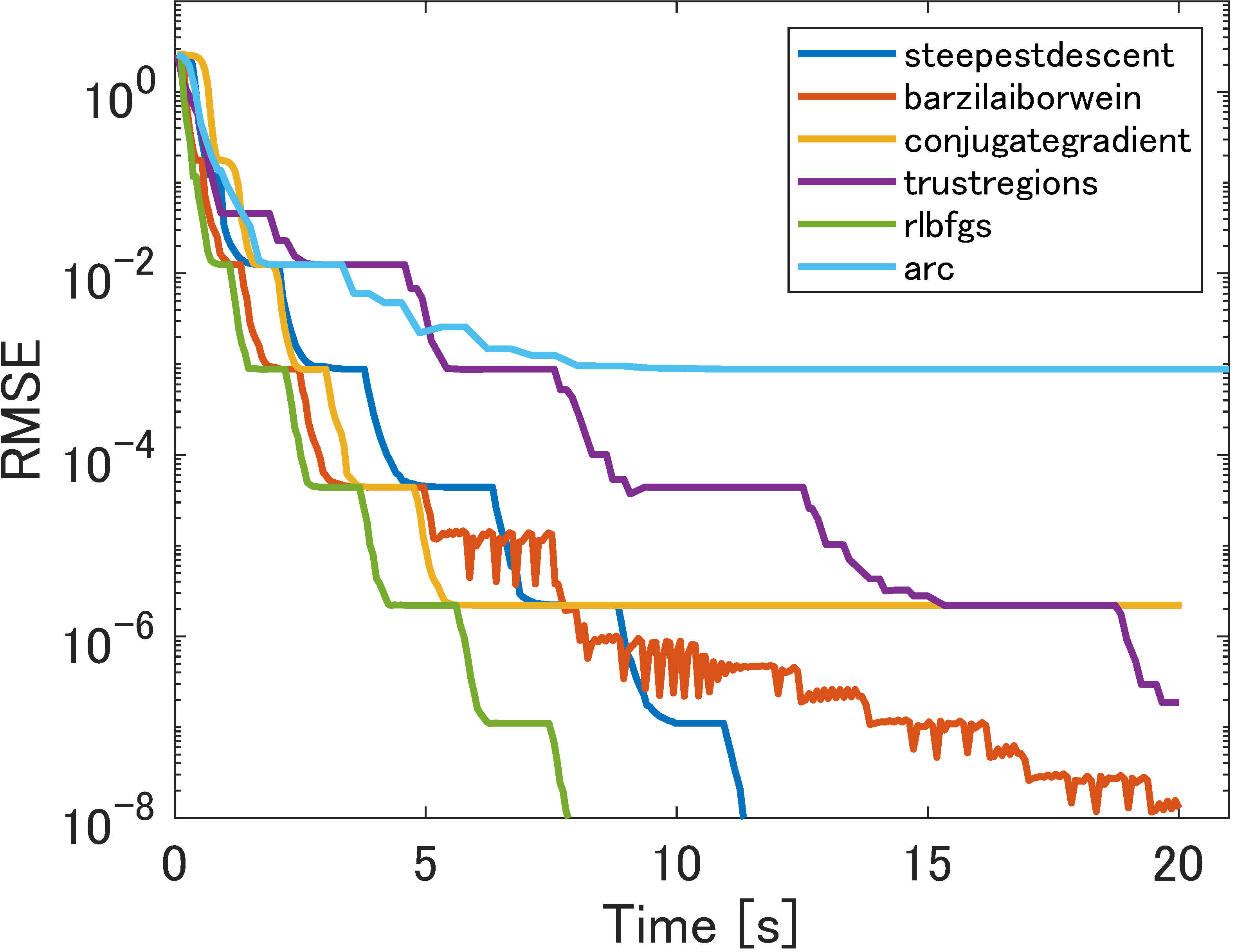}
		\caption{$ \tilde{f}_{5} $}
	\end{subfigure}
	\begin{subfigure}{0.24\textwidth}	
		\includegraphics[width=\textwidth]{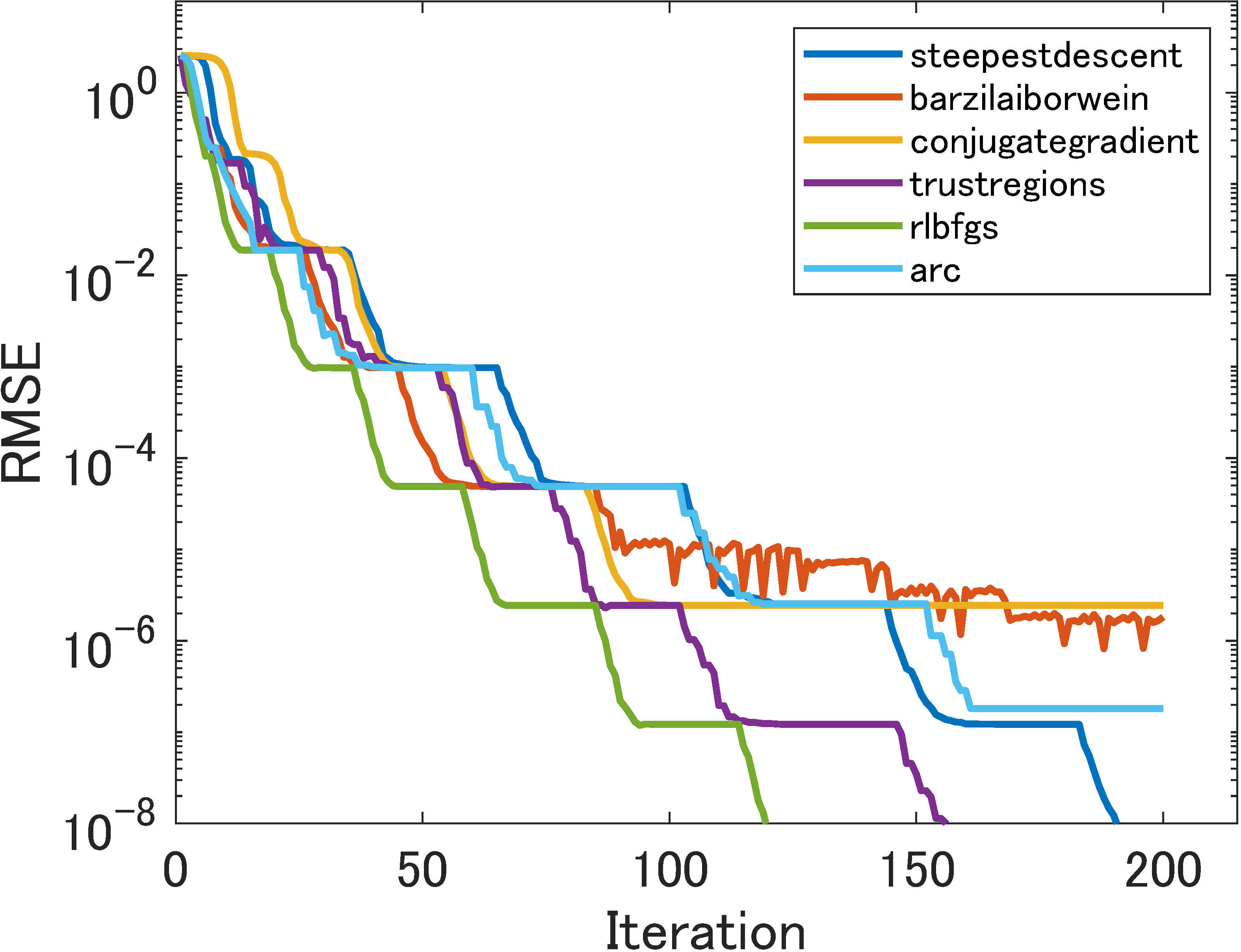}
		\caption{$ \tilde{f}_{1} $}
	\end{subfigure}
	\begin{subfigure}{0.24\textwidth}
		\includegraphics[width=\textwidth]{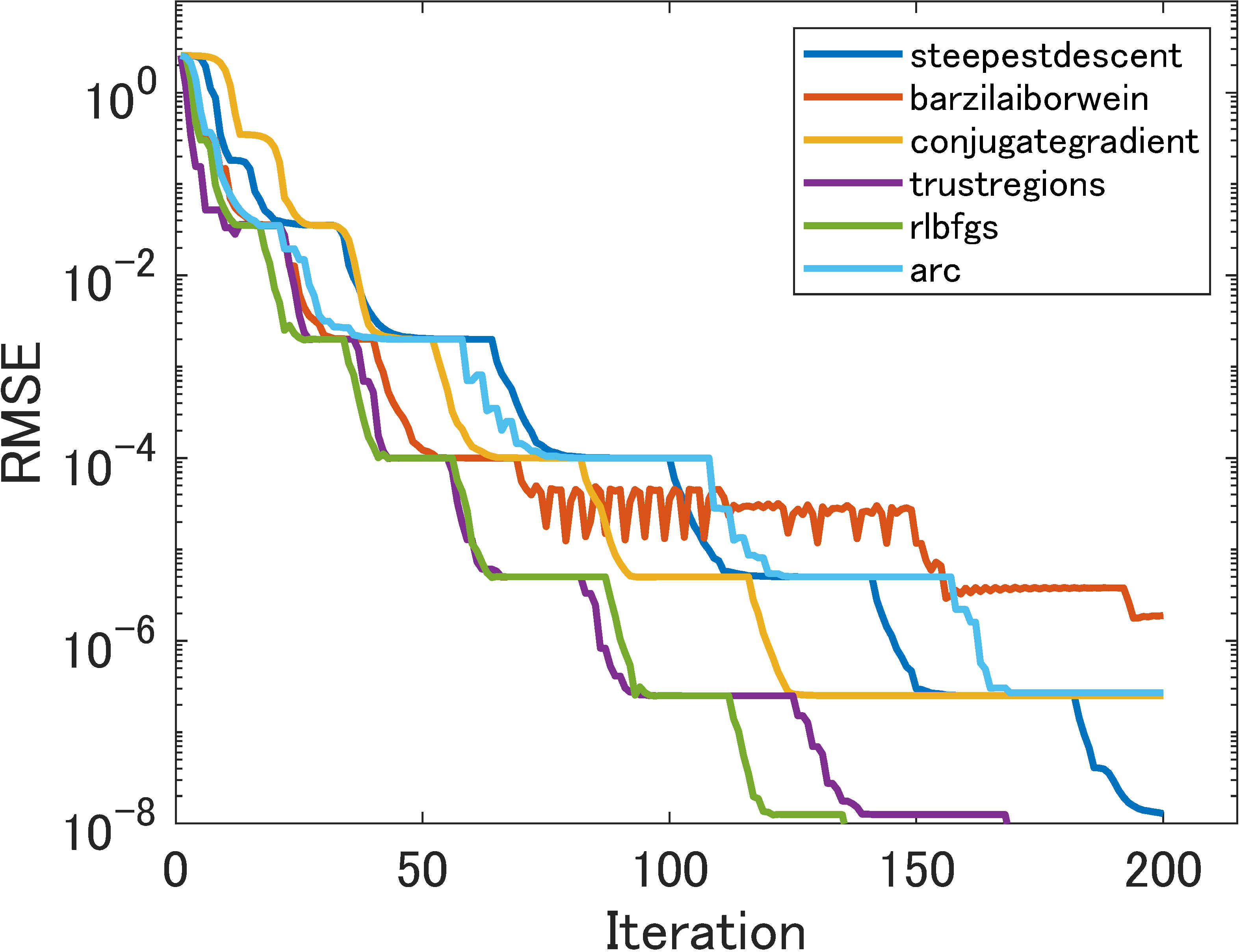}
		\caption{$ \tilde{f}_{2} $}
	\end{subfigure}
	\hfill
	\begin{subfigure}{0.24\textwidth}
		\includegraphics[width=\textwidth]{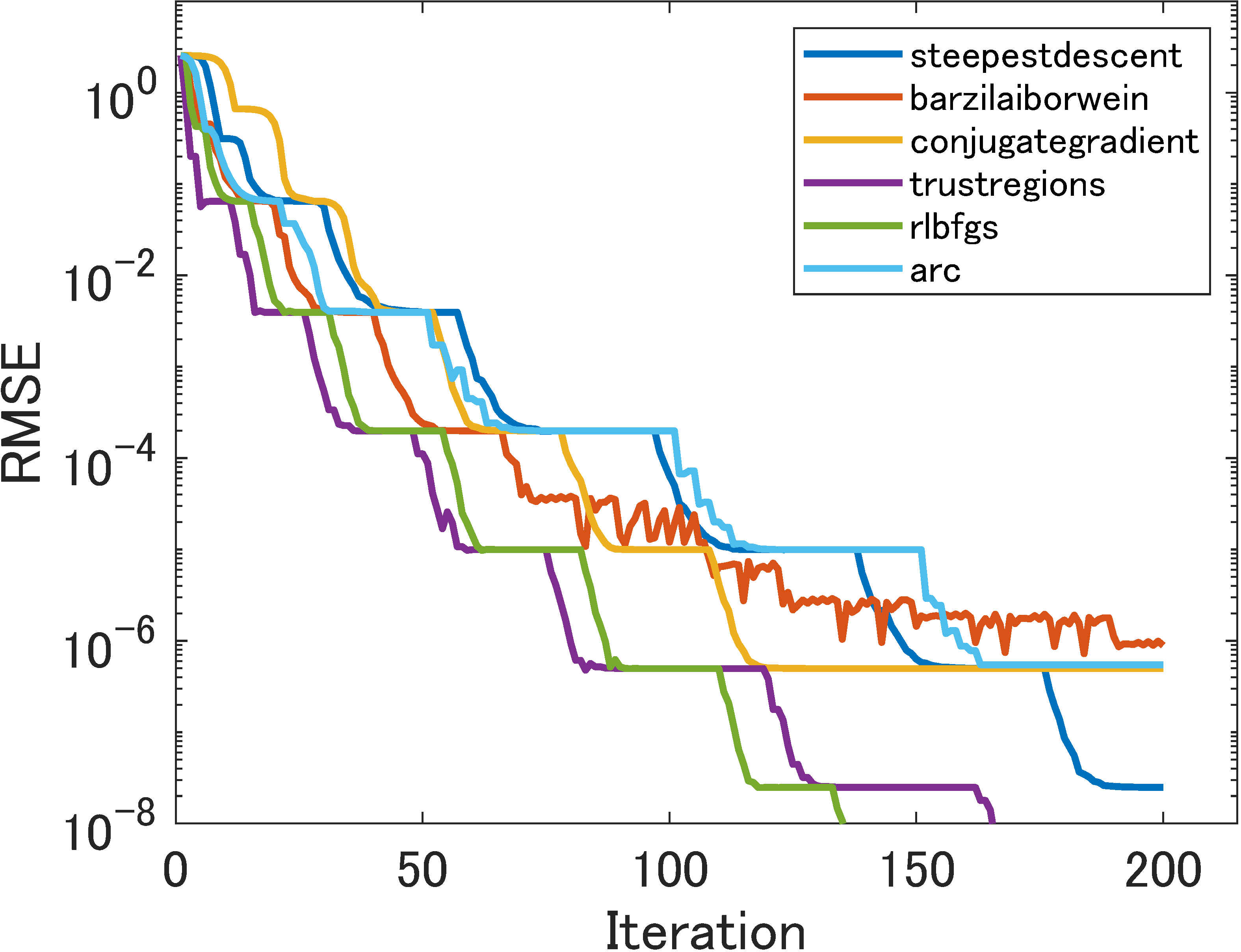}
		\caption{$ \tilde{f}_{3} $}
	\end{subfigure}
	\begin{subfigure}{0.24\textwidth}
		\includegraphics[width=\textwidth]{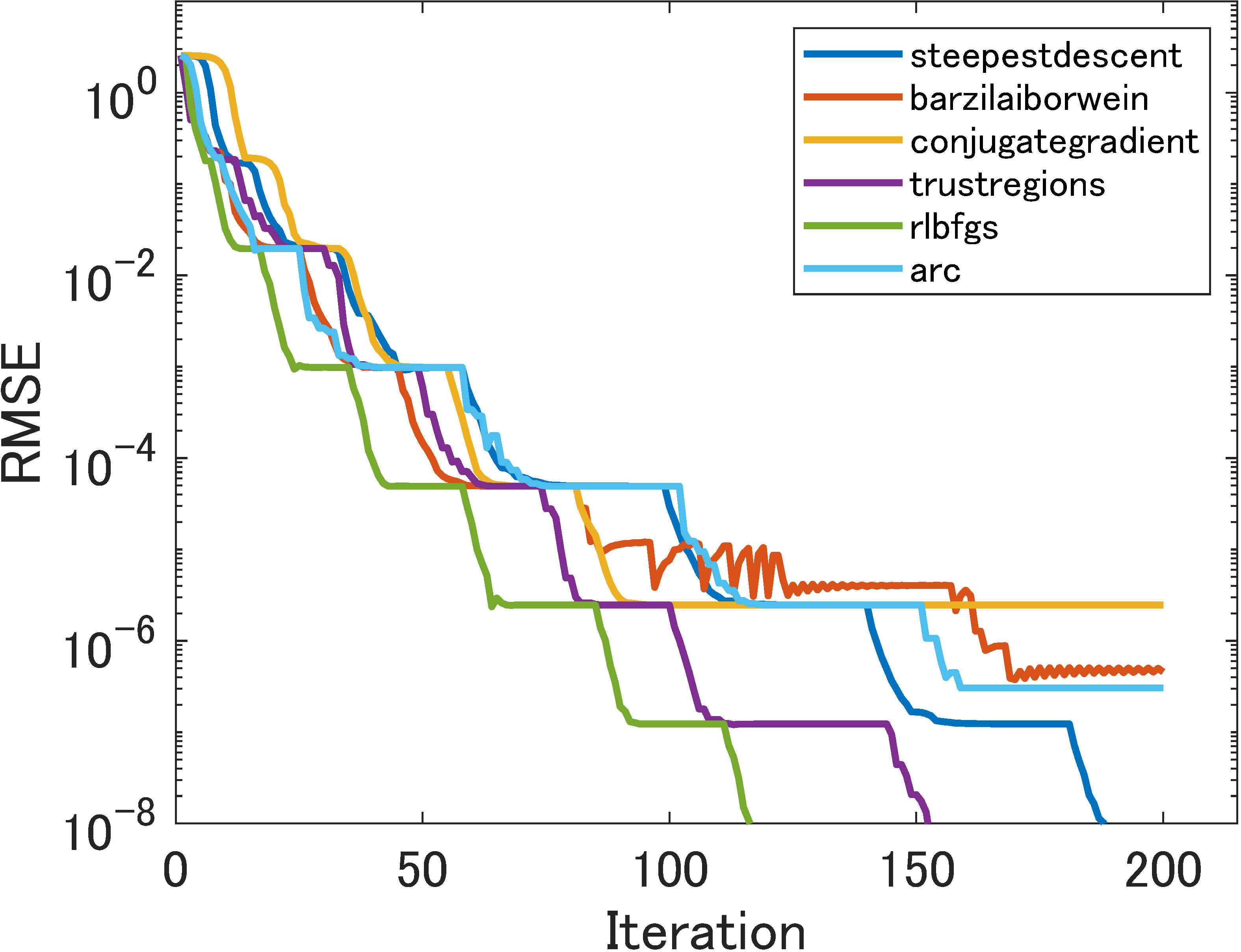}
		\caption{$ \tilde{f}_{4} $}
	\end{subfigure}
	\begin{subfigure}{0.24\textwidth}
		\includegraphics[width=\textwidth]{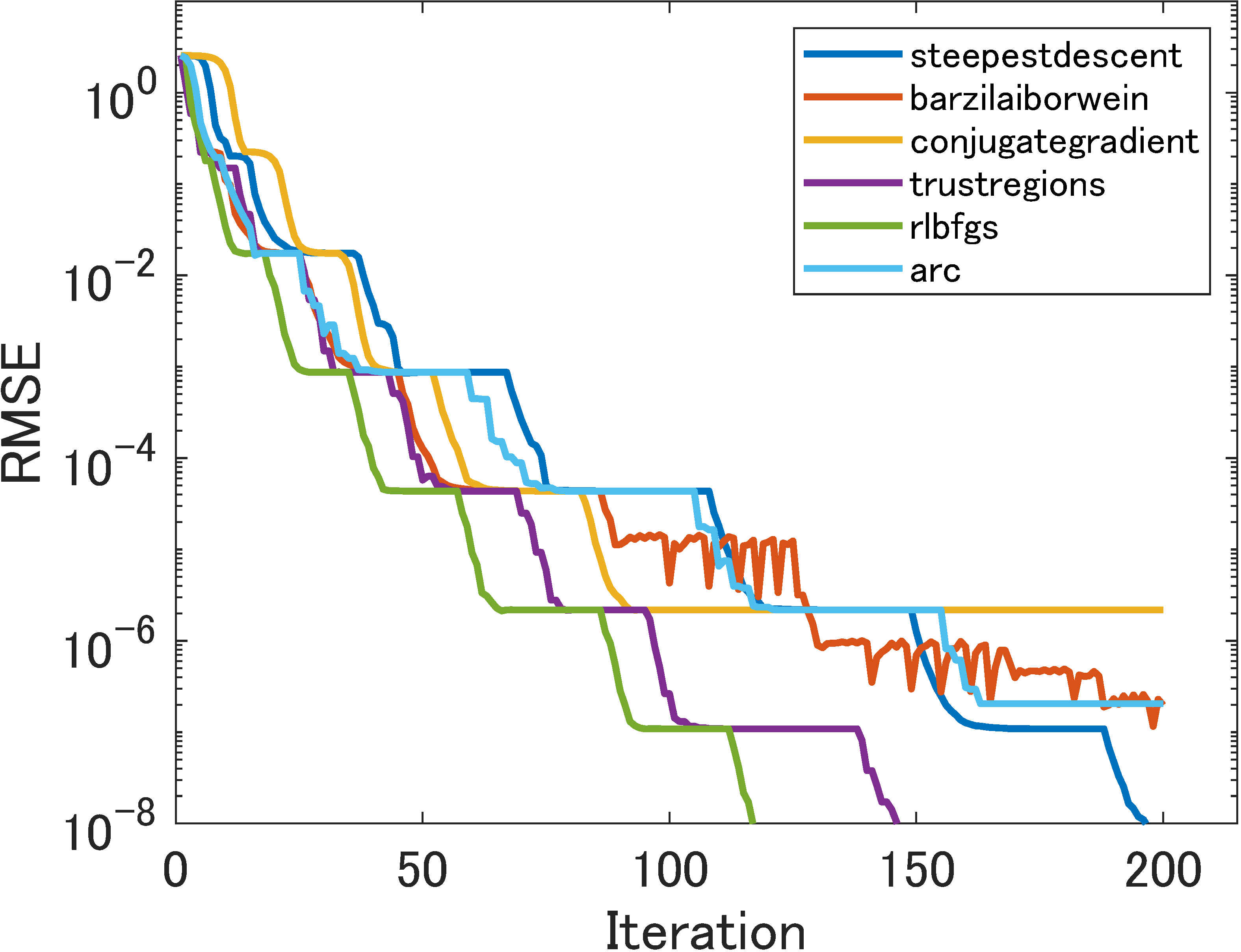}
		\caption{$ \tilde{f}_{5} $}
	\end{subfigure}
	\begin{subfigure}{0.24\textwidth}		
		\includegraphics[width=\textwidth]{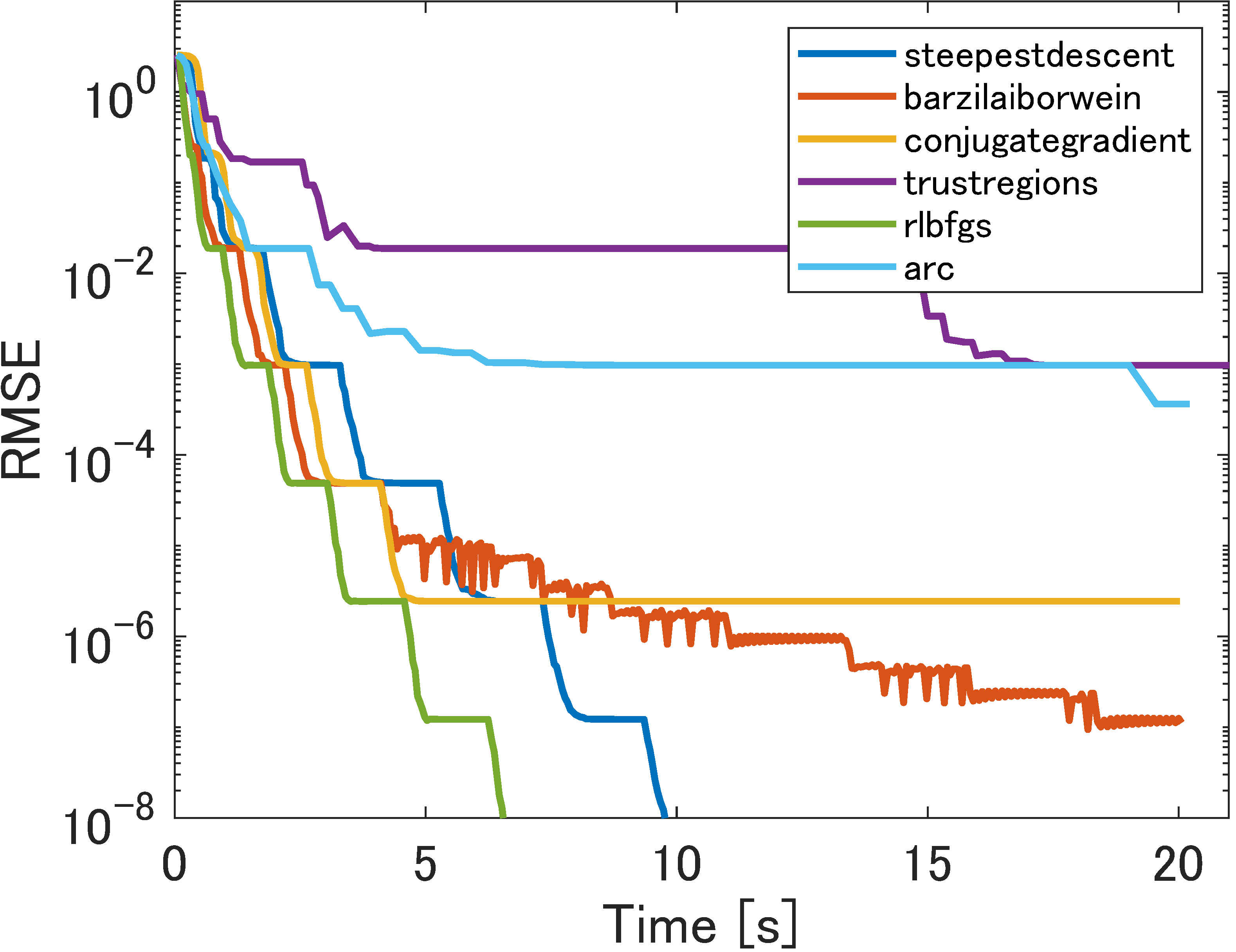}
		\caption{$ \tilde{f}_{1} $}
	\end{subfigure}
	\hfill
	\begin{subfigure}{0.24\textwidth}
		\includegraphics[width=\textwidth]{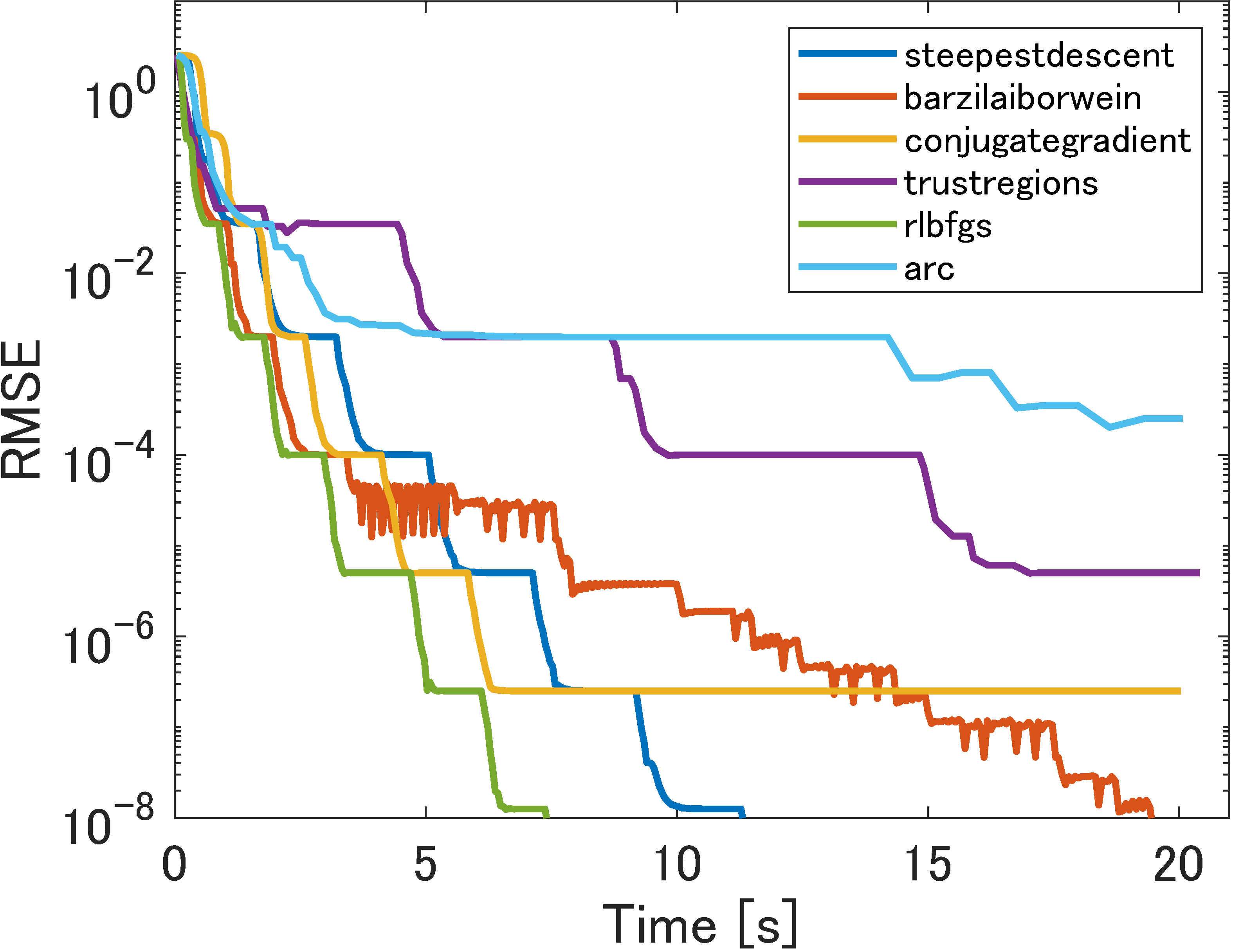}
		\caption{$ \tilde{f}_{2} $}
	\end{subfigure}
	\begin{subfigure}{0.24\textwidth}
		\includegraphics[width=\textwidth]{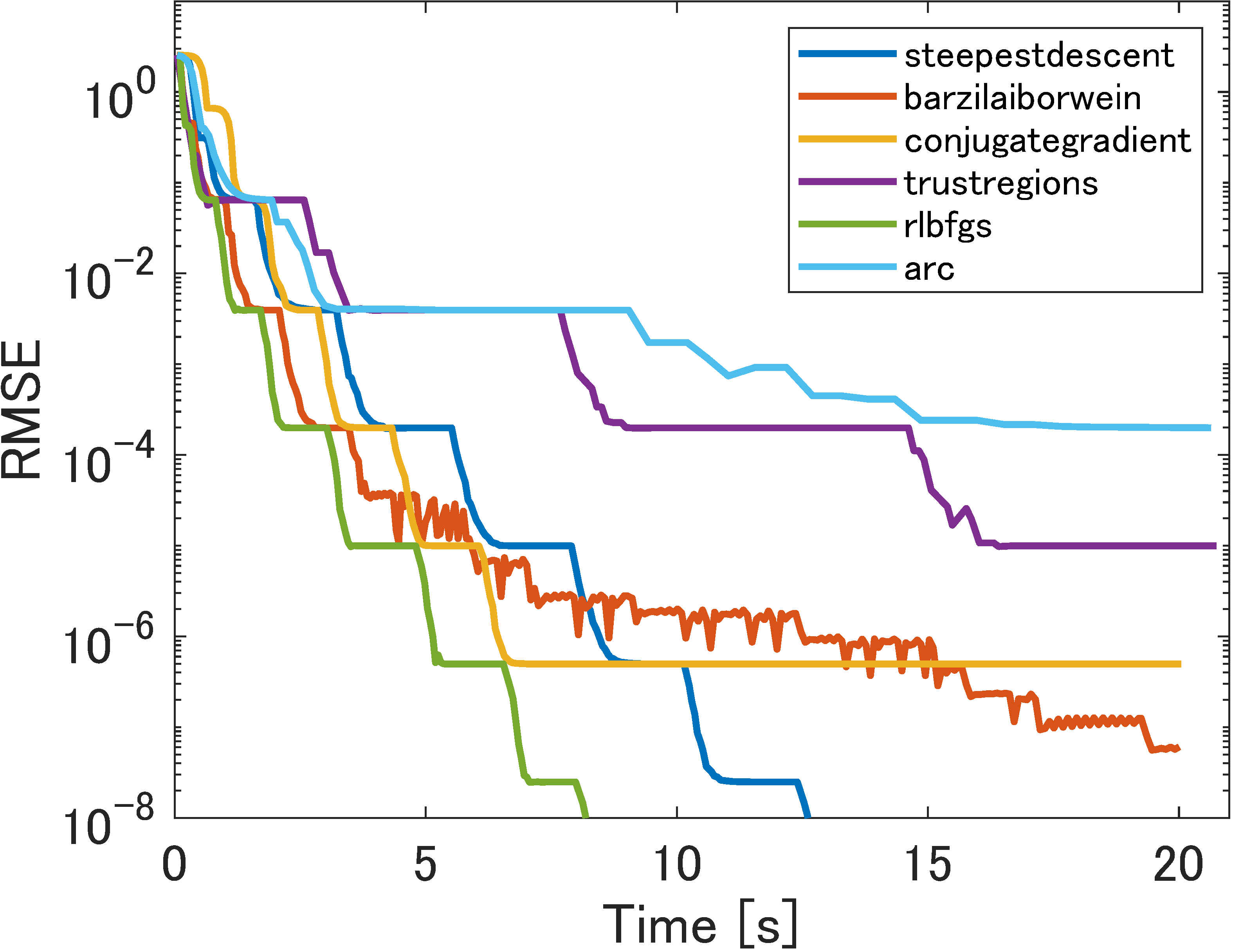}
		\caption{$ \tilde{f}_{3} $}
	\end{subfigure}
	\begin{subfigure}{0.24\textwidth}
		\includegraphics[width=\textwidth]{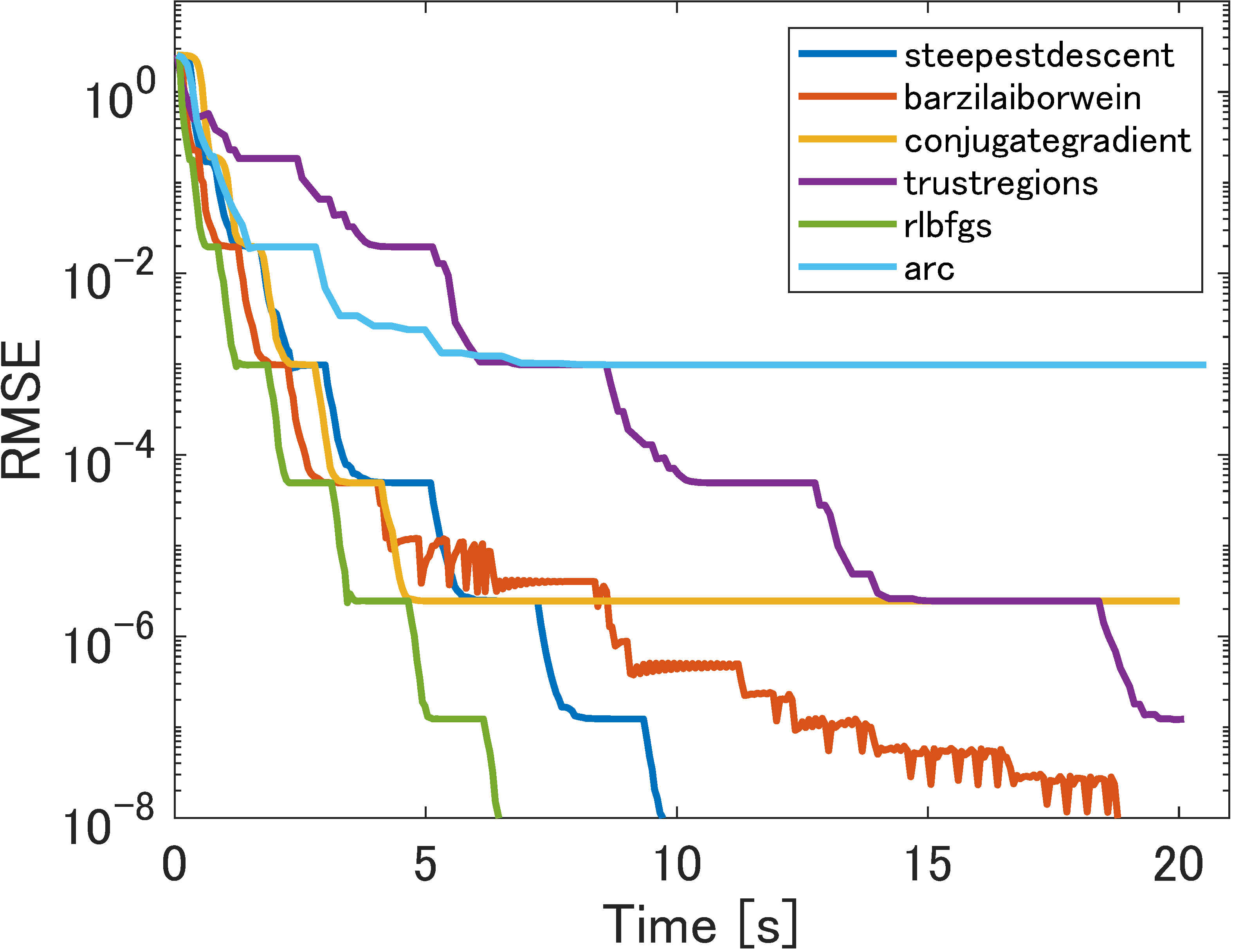}
		\caption{$ \tilde{f}_{4} $}
	\end{subfigure}
	\begin{subfigure}{0.24\textwidth}
		\includegraphics[width=\textwidth]{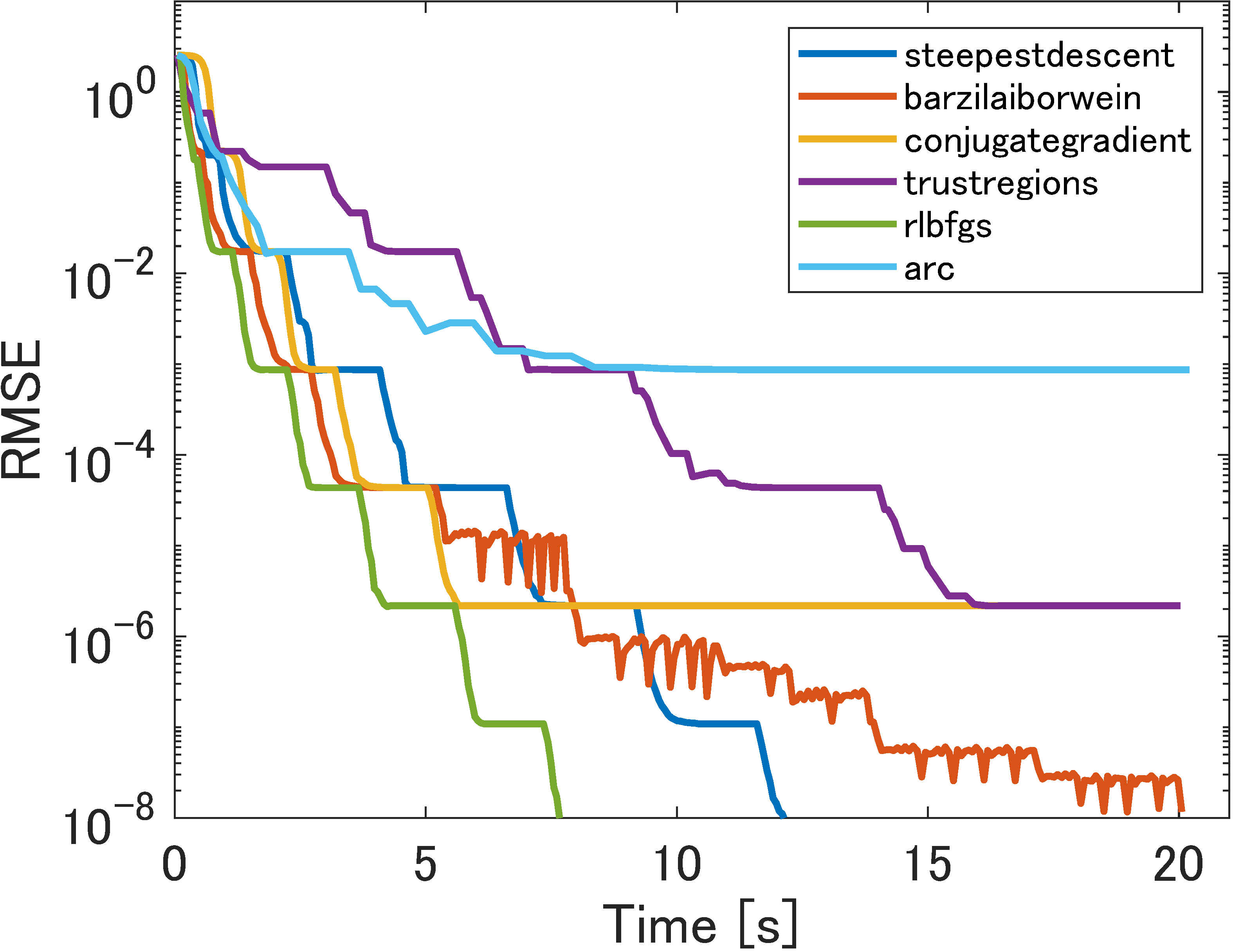}
		\caption{$ \tilde{f}_{5} $}
	\end{subfigure}
	\caption{Low-rank matrix completion with outliers for two rank-10 $500 \times 500$ matrices by using different smoothing functions in Table \ref{table:smlist}.
			(a)--(j) corresponds to one matrix with outliers created by using $\mu_{N}=\sigma_{N}=0.1$, while (k)--(t) corresponds to the other with outliers created by using $\mu_{N}=\sigma_{N}=1$.
			(a)--(e) and (k)--(o) comprise the running iteration comparison; (f)--(j) and (p)--(t) comprise the  time comparison.}
	\label{fig:outliers}
\end{figure}

\begin{figure}[!b]\centering
	\begin{subfigure}{0.24\textwidth} 	
		\includegraphics[width=\textwidth]{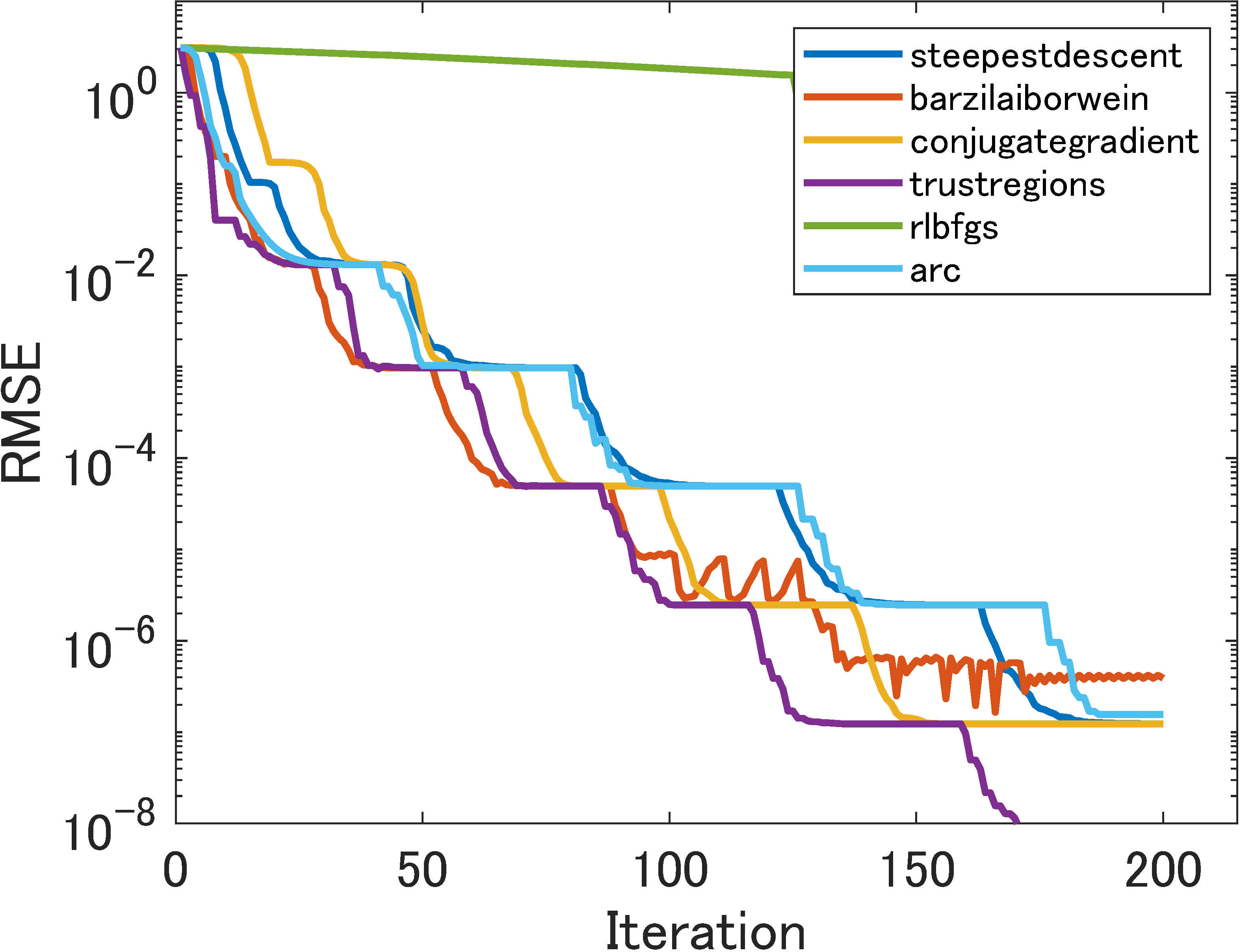}
		\caption{$ \tilde{f}_{1} $}
	\end{subfigure}
	\begin{subfigure}{0.24\textwidth}
		\includegraphics[width=\textwidth]{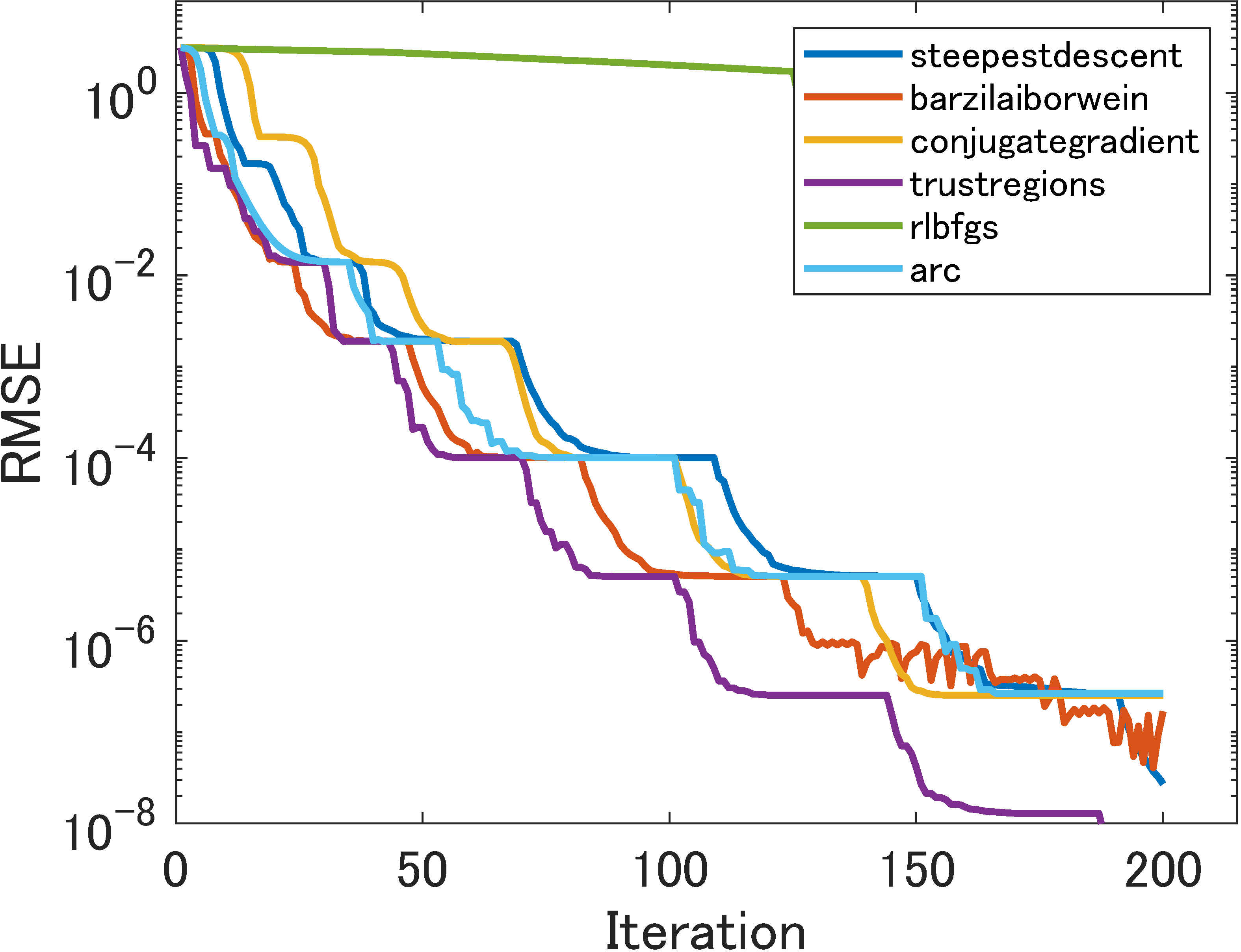}
		\caption{$ \tilde{f}_{2} $}
	\end{subfigure}
	\begin{subfigure}{0.24\textwidth}
		\includegraphics[width=\textwidth]{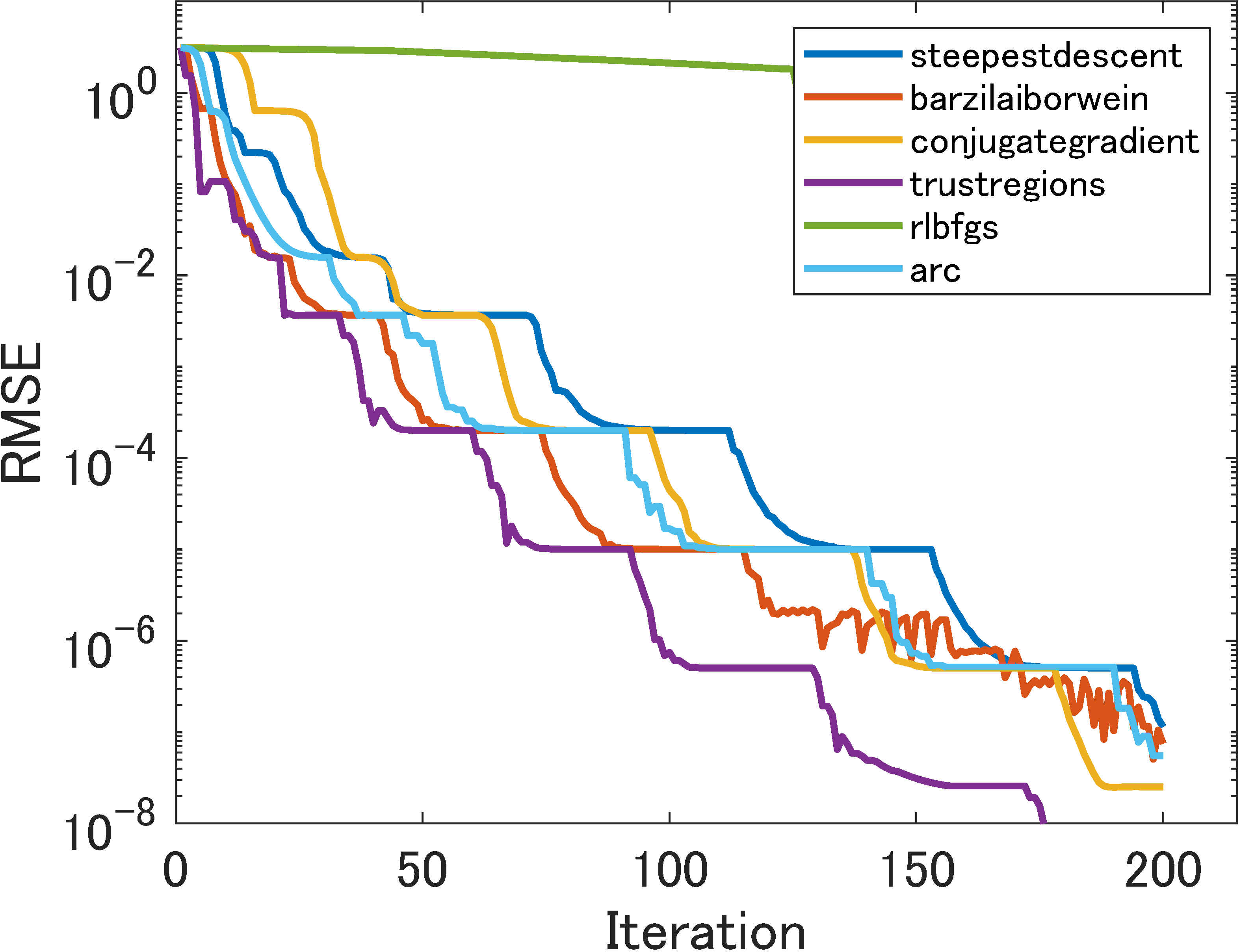}
		\caption{$ \tilde{f}_{3} $}
	\end{subfigure}
	\begin{subfigure}{0.24\textwidth}
		\includegraphics[width=\textwidth]{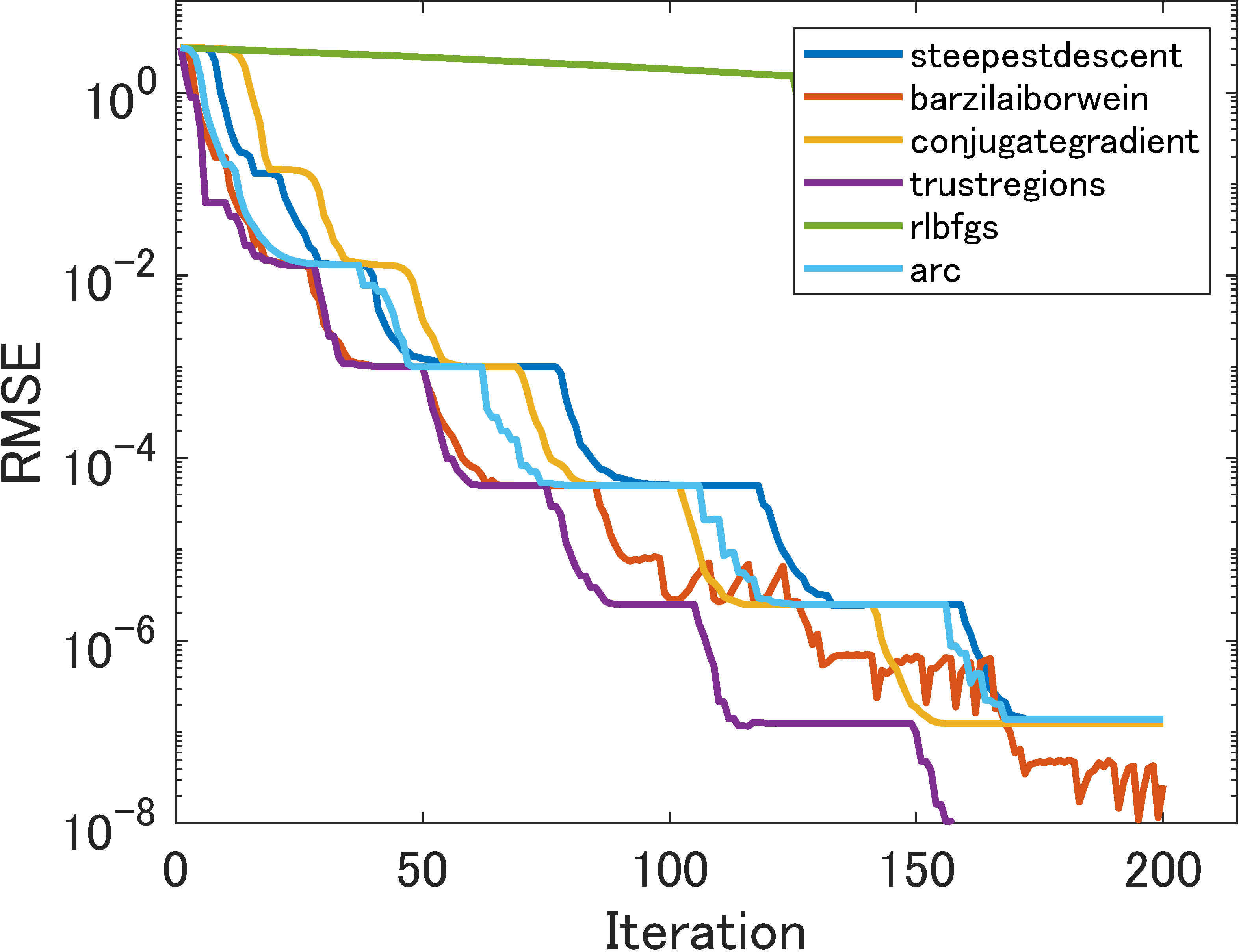}
		\caption{$ \tilde{f}_{4} $}
	\end{subfigure}
	\hfill 
	\begin{subfigure}{0.24\textwidth}
		\includegraphics[width=\textwidth]{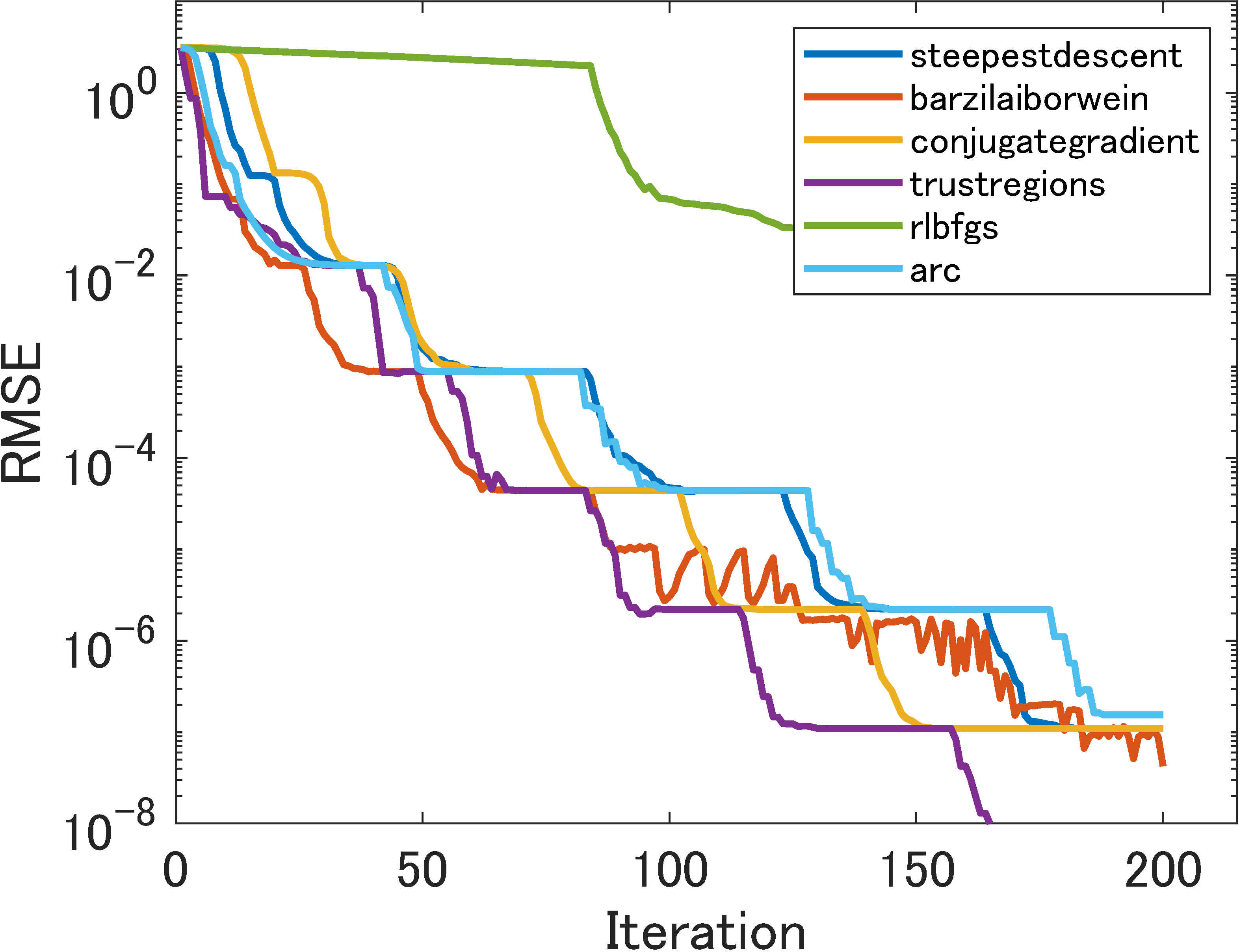}
		\caption{$ \tilde{f}_{5} $}
	\end{subfigure}
	\begin{subfigure}{0.24\textwidth}  
		\includegraphics[width=\textwidth]{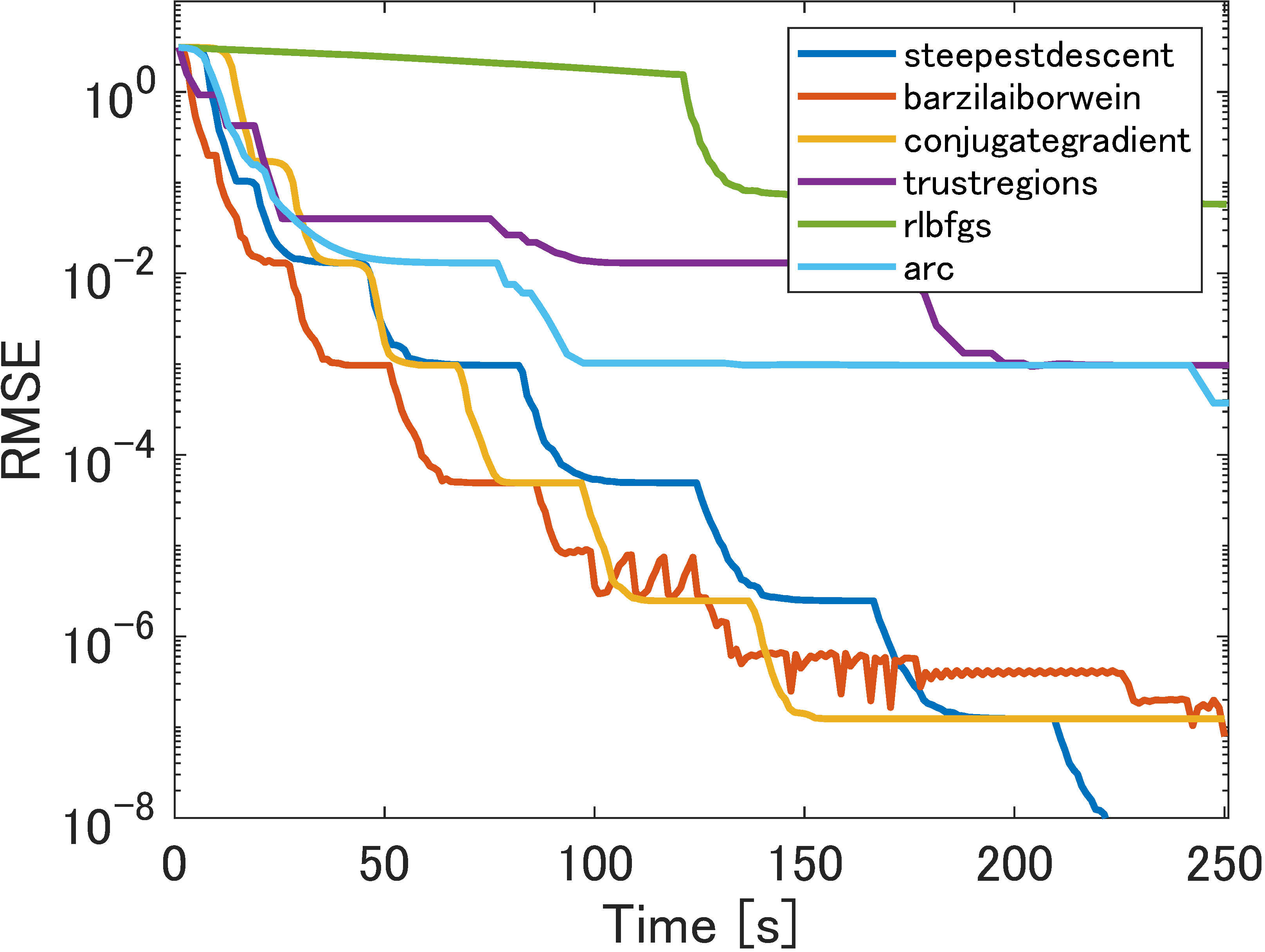}
		\caption{$ \tilde{f}_{1} $}
	\end{subfigure}
	\begin{subfigure}{0.24\textwidth}
		\includegraphics[width=\textwidth]{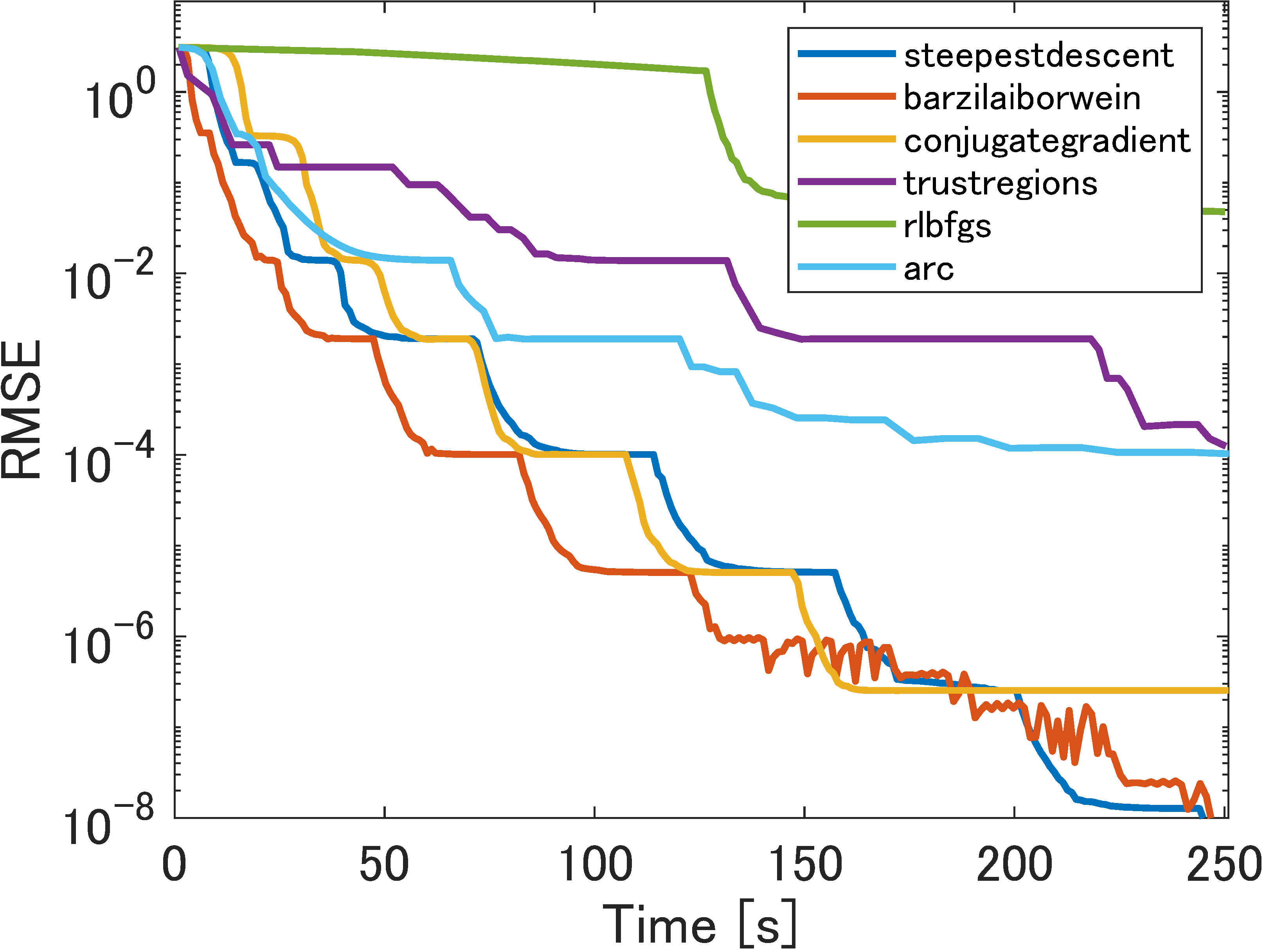}
		\caption{$ \tilde{f}_{2} $}
	\end{subfigure}
	\begin{subfigure}{0.24\textwidth}
		\includegraphics[width=\textwidth]{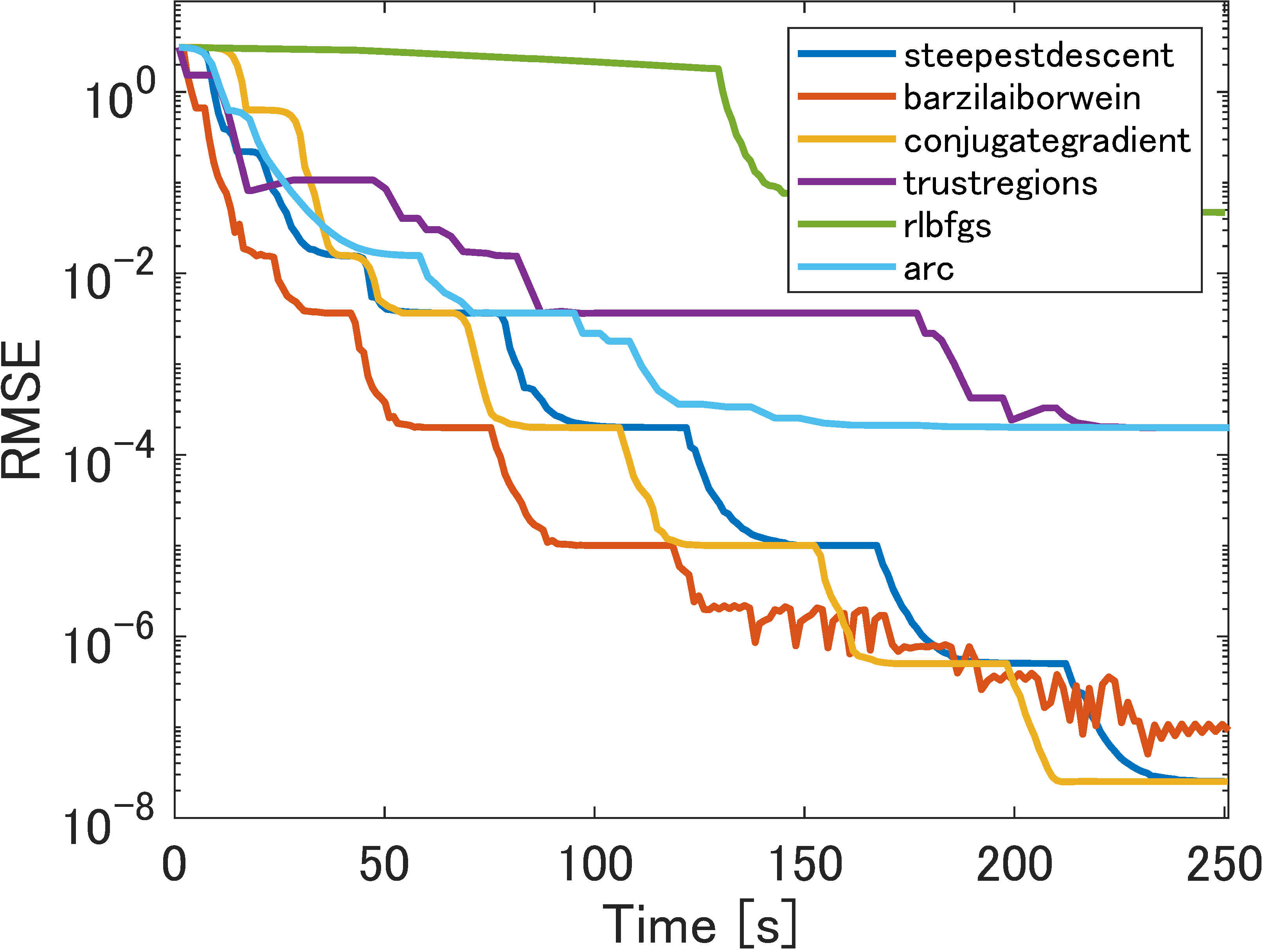}
		\caption{$ \tilde{f}_{3} $}
	\end{subfigure}
	\hfill
	\begin{subfigure}{0.24\textwidth}
		\includegraphics[width=\textwidth]{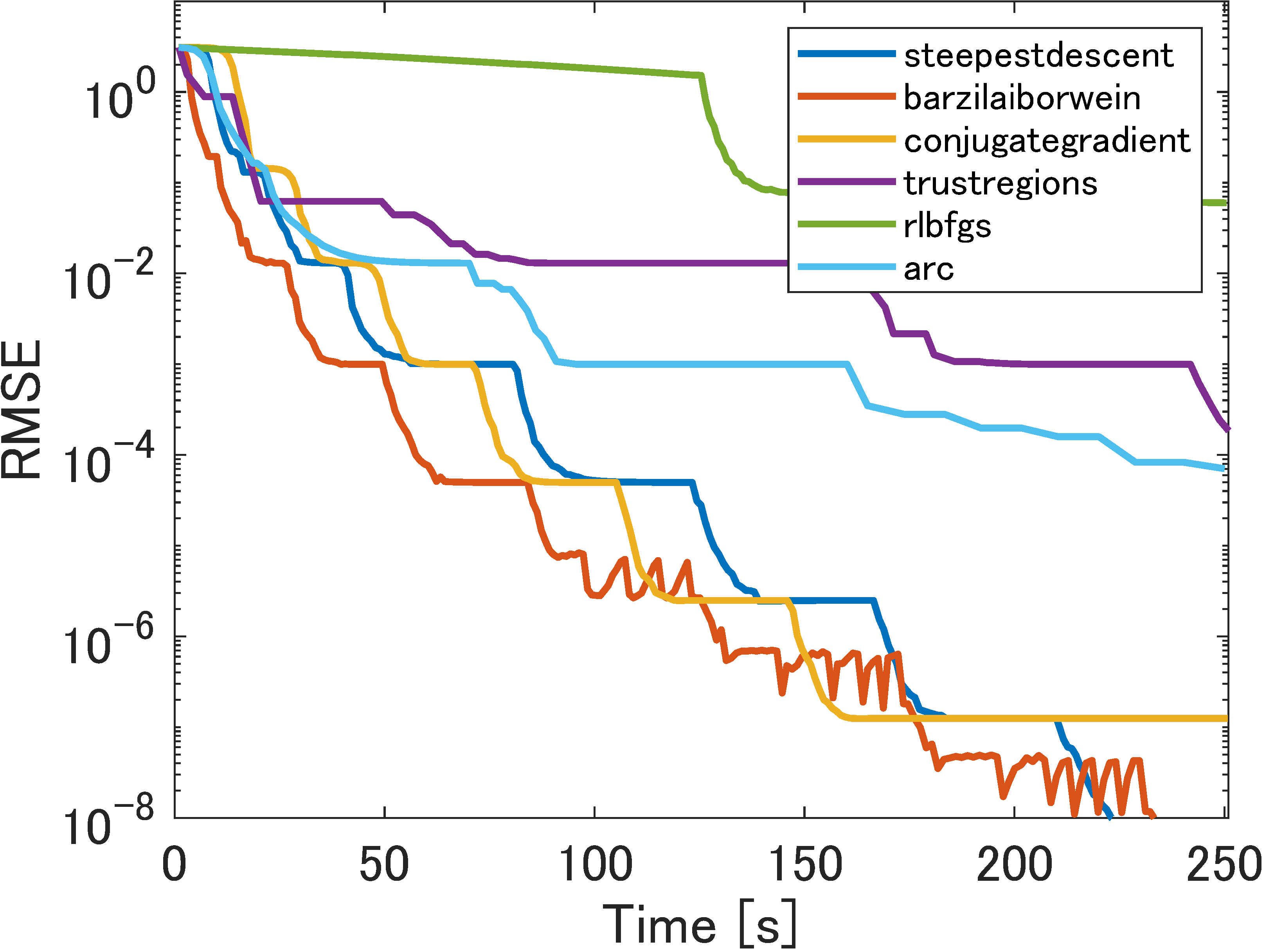}
		\caption{$ \tilde{f}_{4} $}
	\end{subfigure}
	\begin{subfigure}{0.24\textwidth}
		\includegraphics[width=\textwidth]{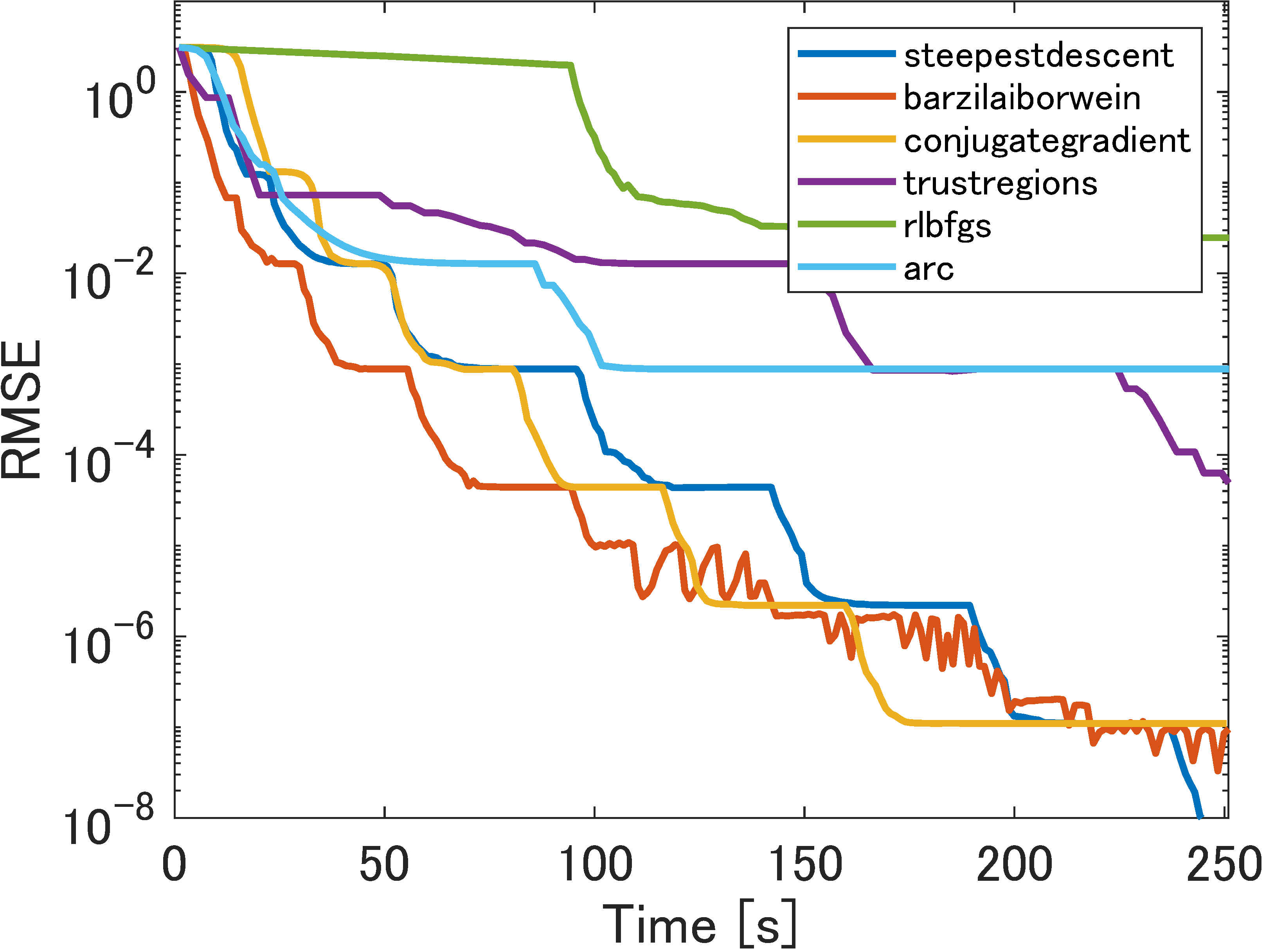}
		\caption{$ \tilde{f}_{5} $}
	\end{subfigure}
	\begin{subfigure}{0.24\textwidth}	
		\includegraphics[width=\textwidth]{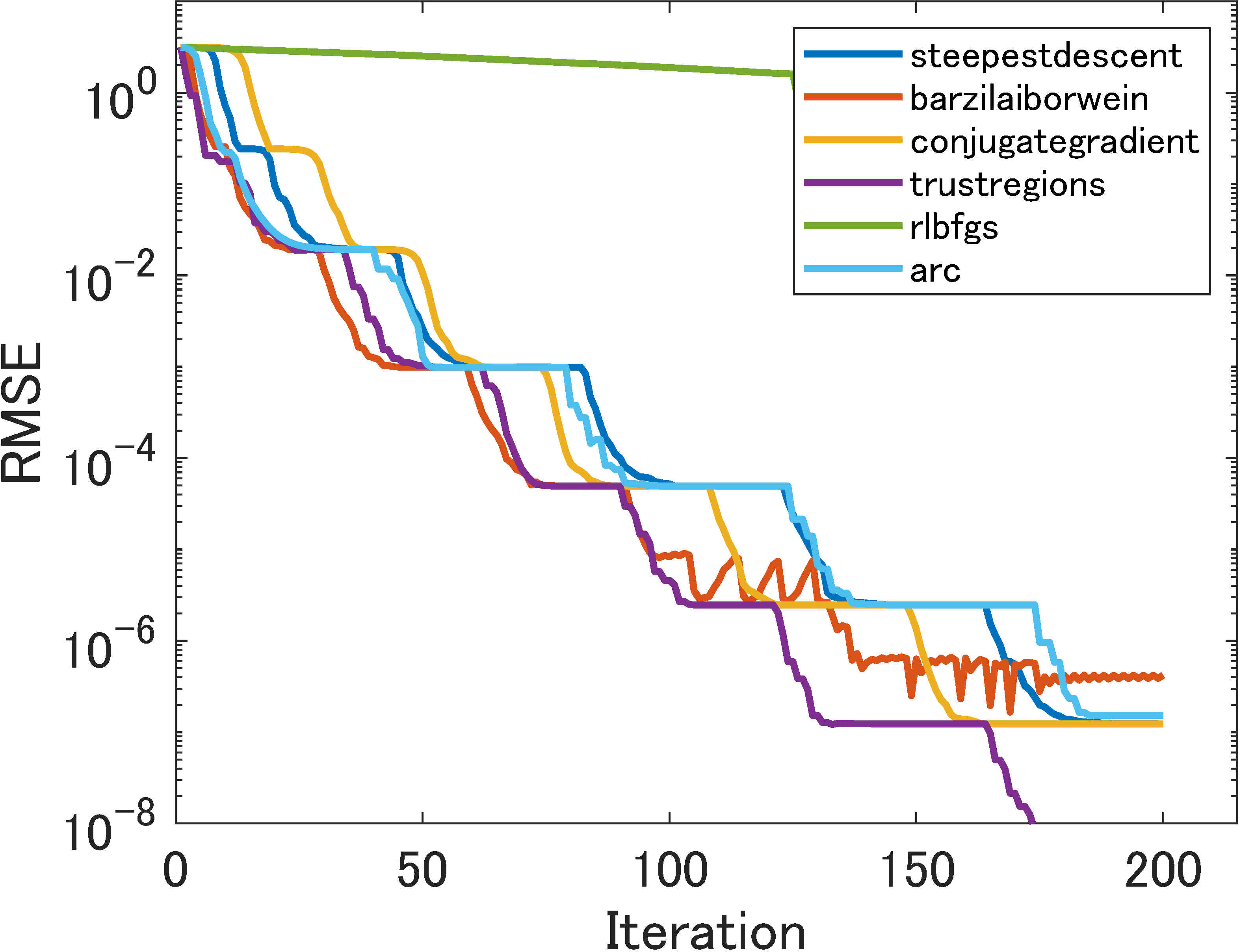}
		\caption{$ \tilde{f}_{1} $}
	\end{subfigure}
	\begin{subfigure}{0.24\textwidth}
		\includegraphics[width=\textwidth]{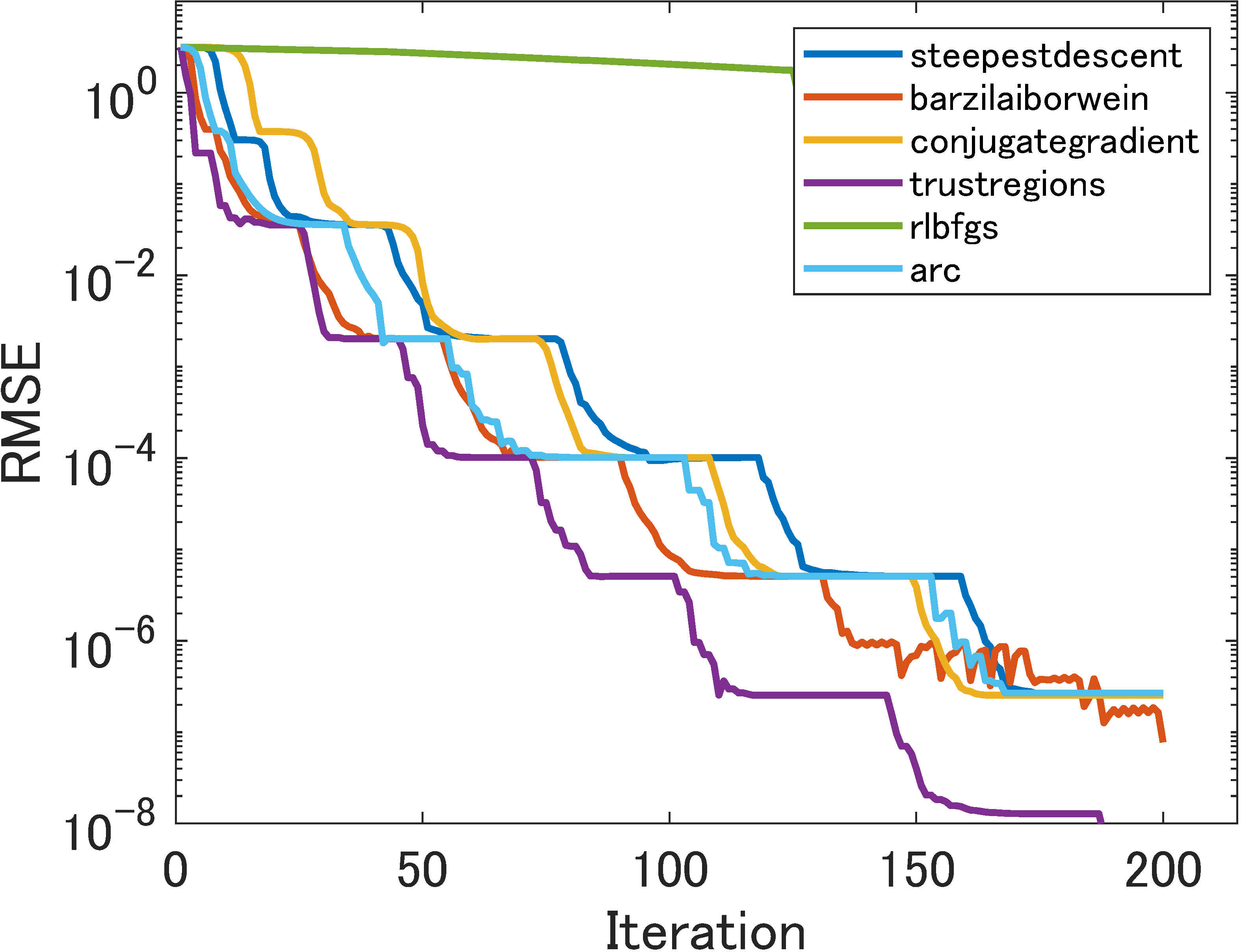}
		\caption{$ \tilde{f}_{2} $}
	\end{subfigure}
	\hfill
	\begin{subfigure}{0.24\textwidth}
		\includegraphics[width=\textwidth]{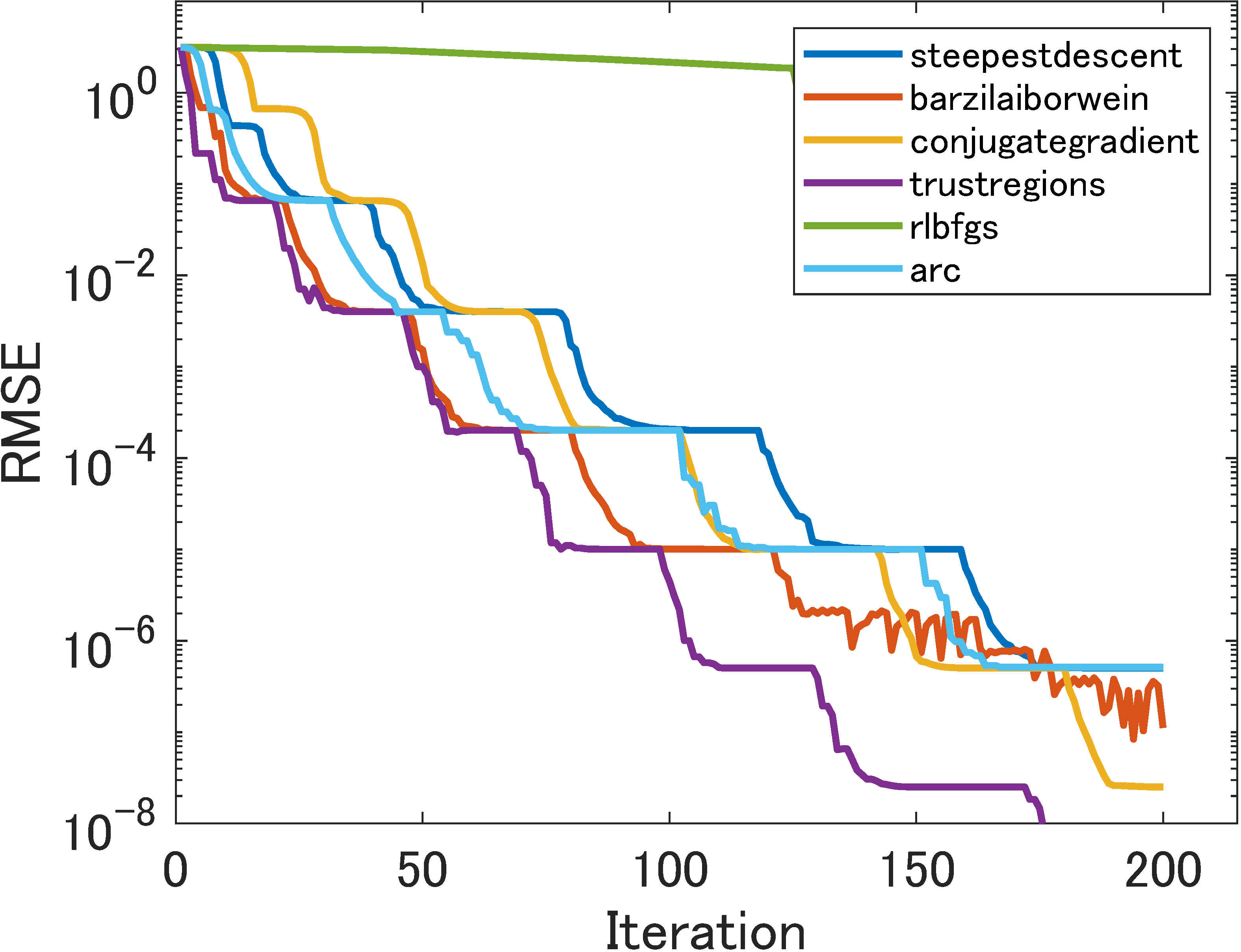}
		\caption{$ \tilde{f}_{3} $}
	\end{subfigure}
	\begin{subfigure}{0.24\textwidth}
		\includegraphics[width=\textwidth]{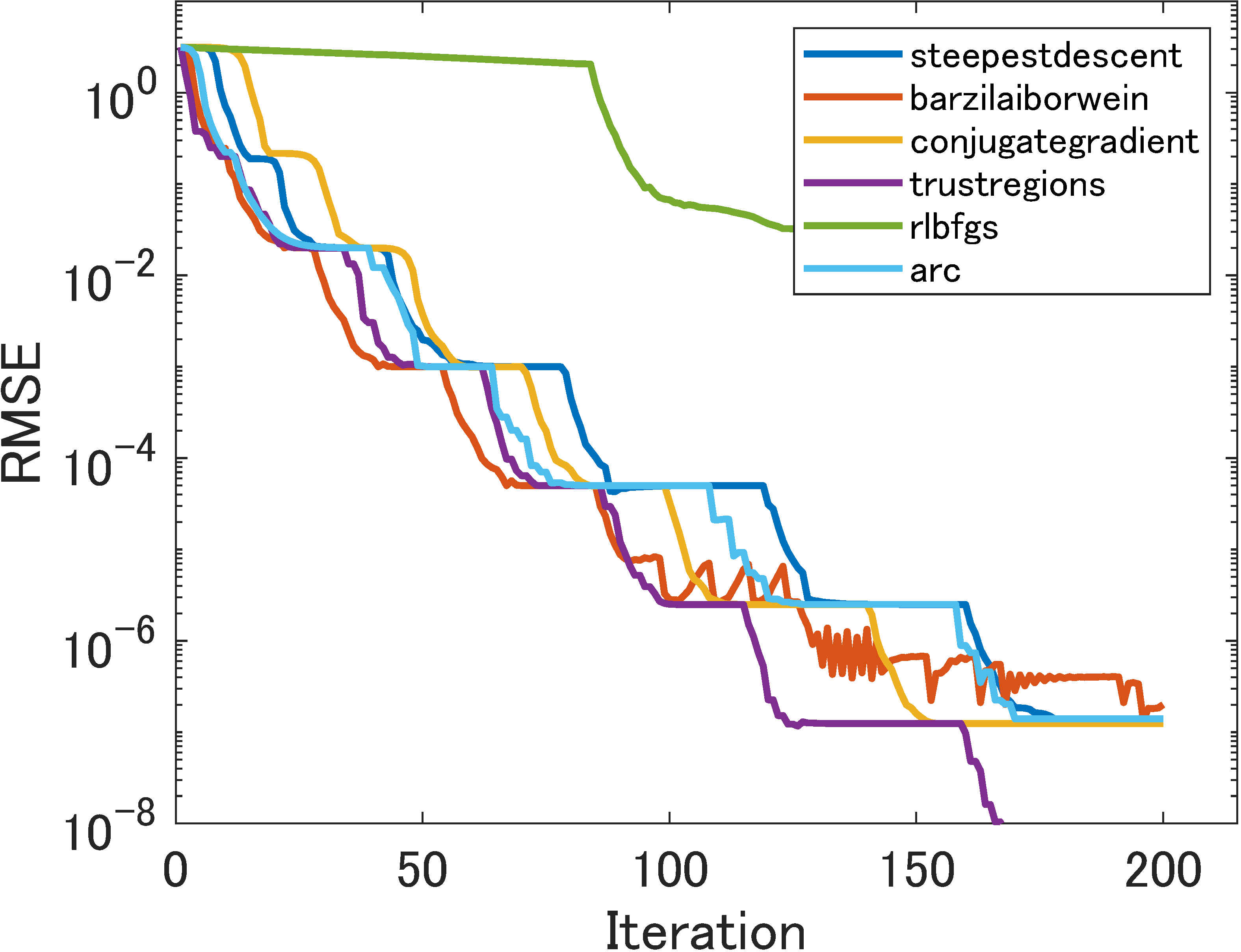}
		\caption{$ \tilde{f}_{4} $}
	\end{subfigure}
	\begin{subfigure}{0.24\textwidth}
		\includegraphics[width=\textwidth]{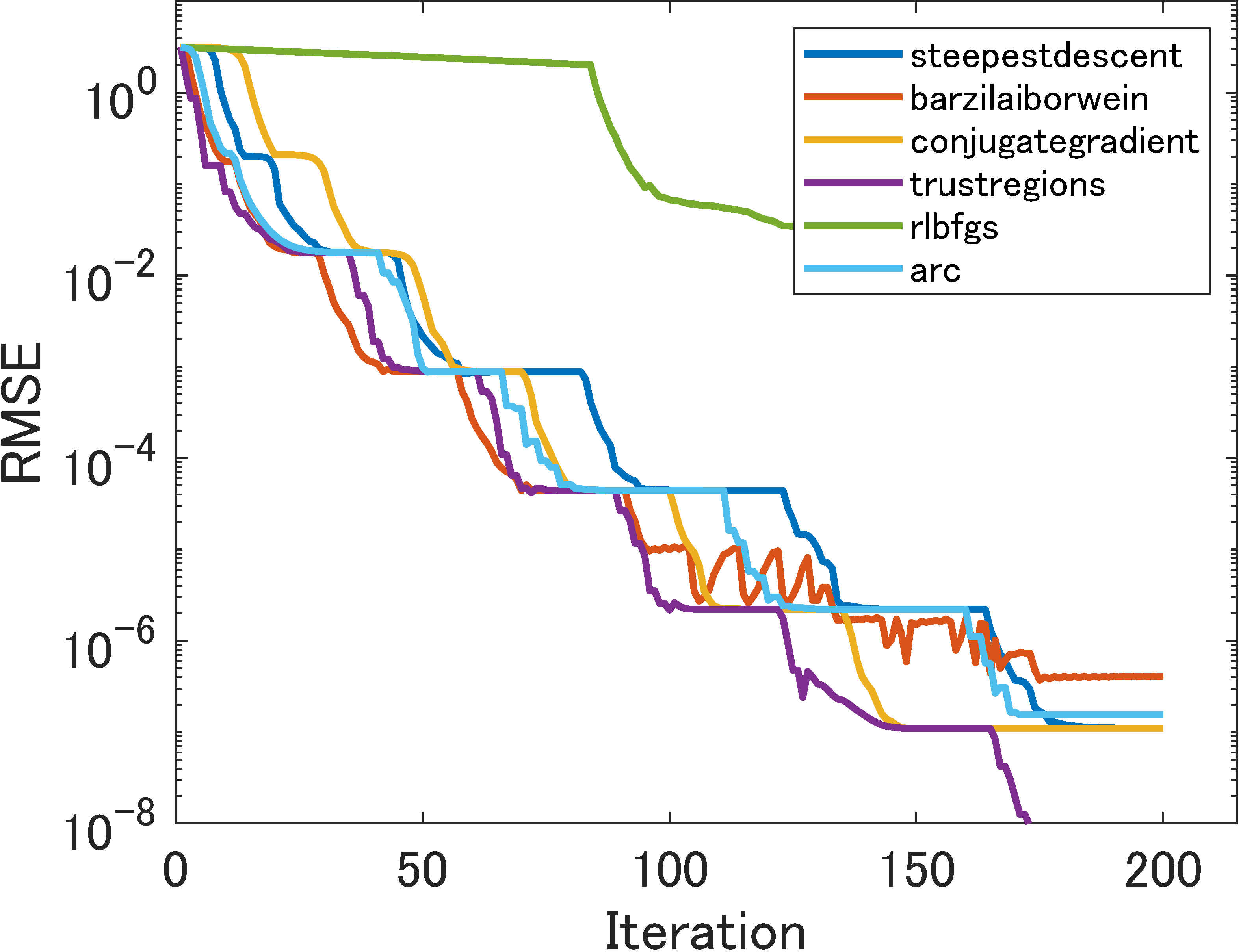}
		\caption{$ \tilde{f}_{5} $}
	\end{subfigure}
	\begin{subfigure}{0.24\textwidth}		
		\includegraphics[width=\textwidth]{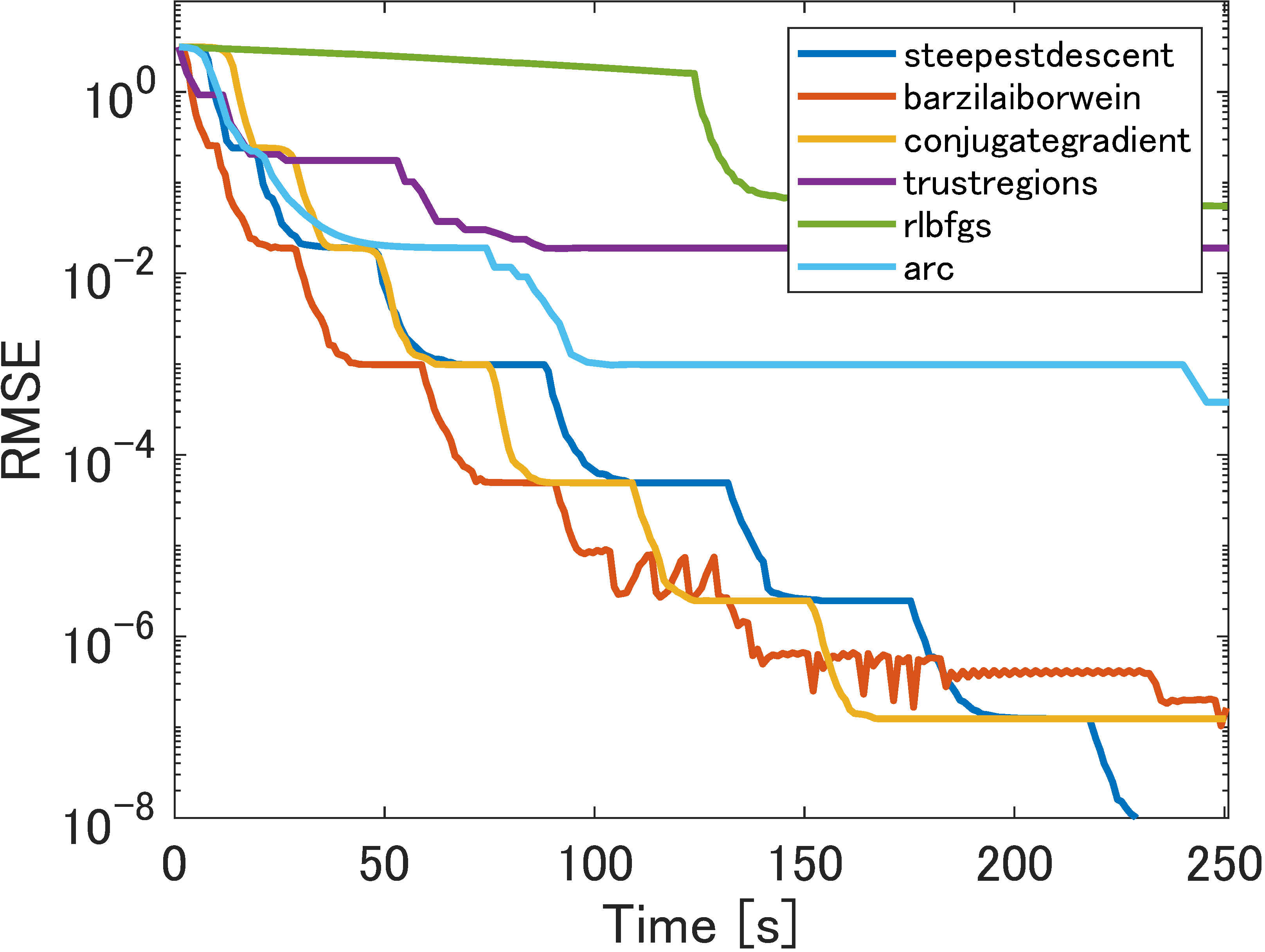}
		\caption{$ \tilde{f}_{1} $}
	\end{subfigure}
	\hfill
	\begin{subfigure}{0.24\textwidth}
		\includegraphics[width=\textwidth]{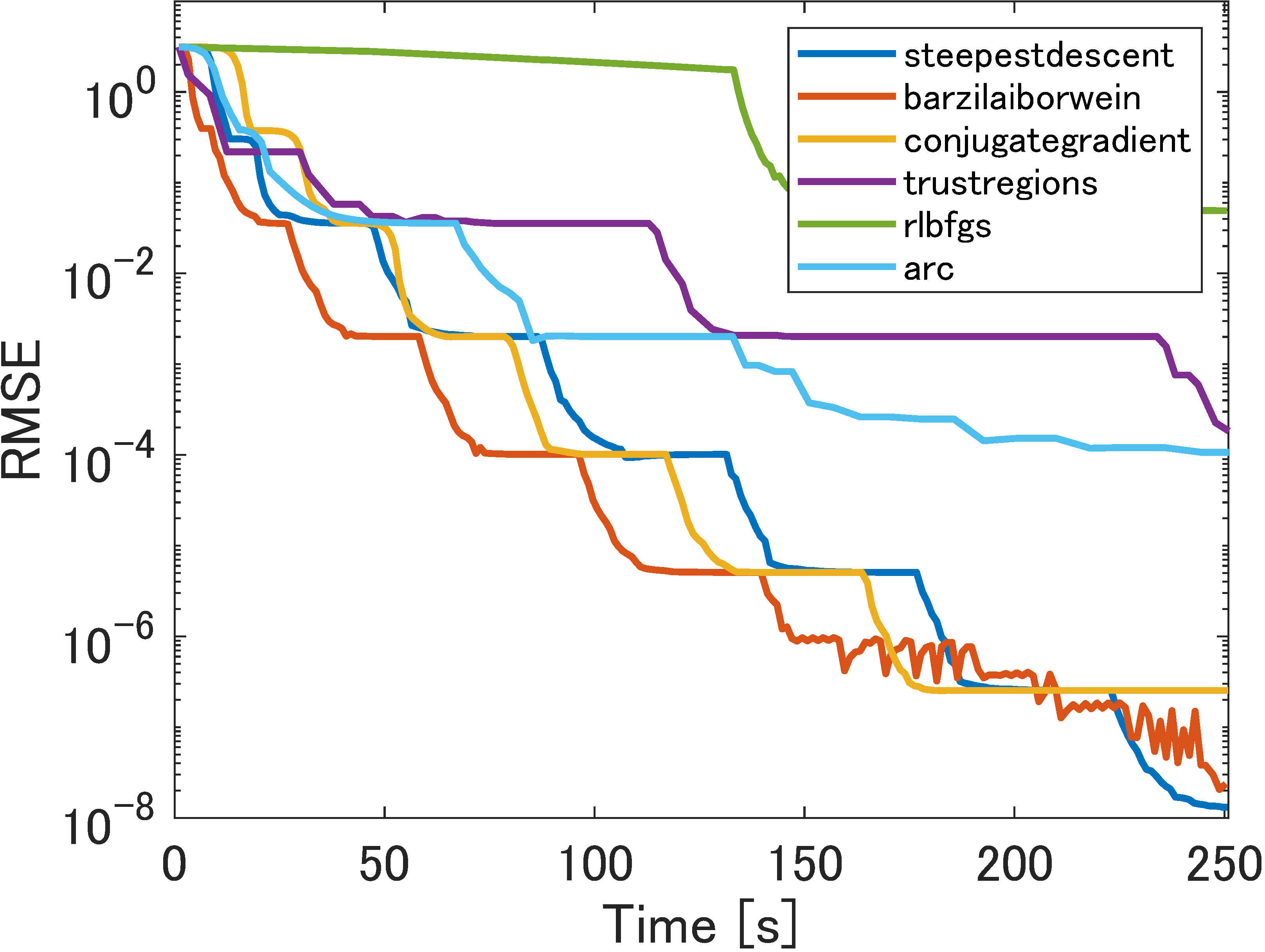}
		\caption{$ \tilde{f}_{2} $}
	\end{subfigure}
	\begin{subfigure}{0.24\textwidth}
		\includegraphics[width=\textwidth]{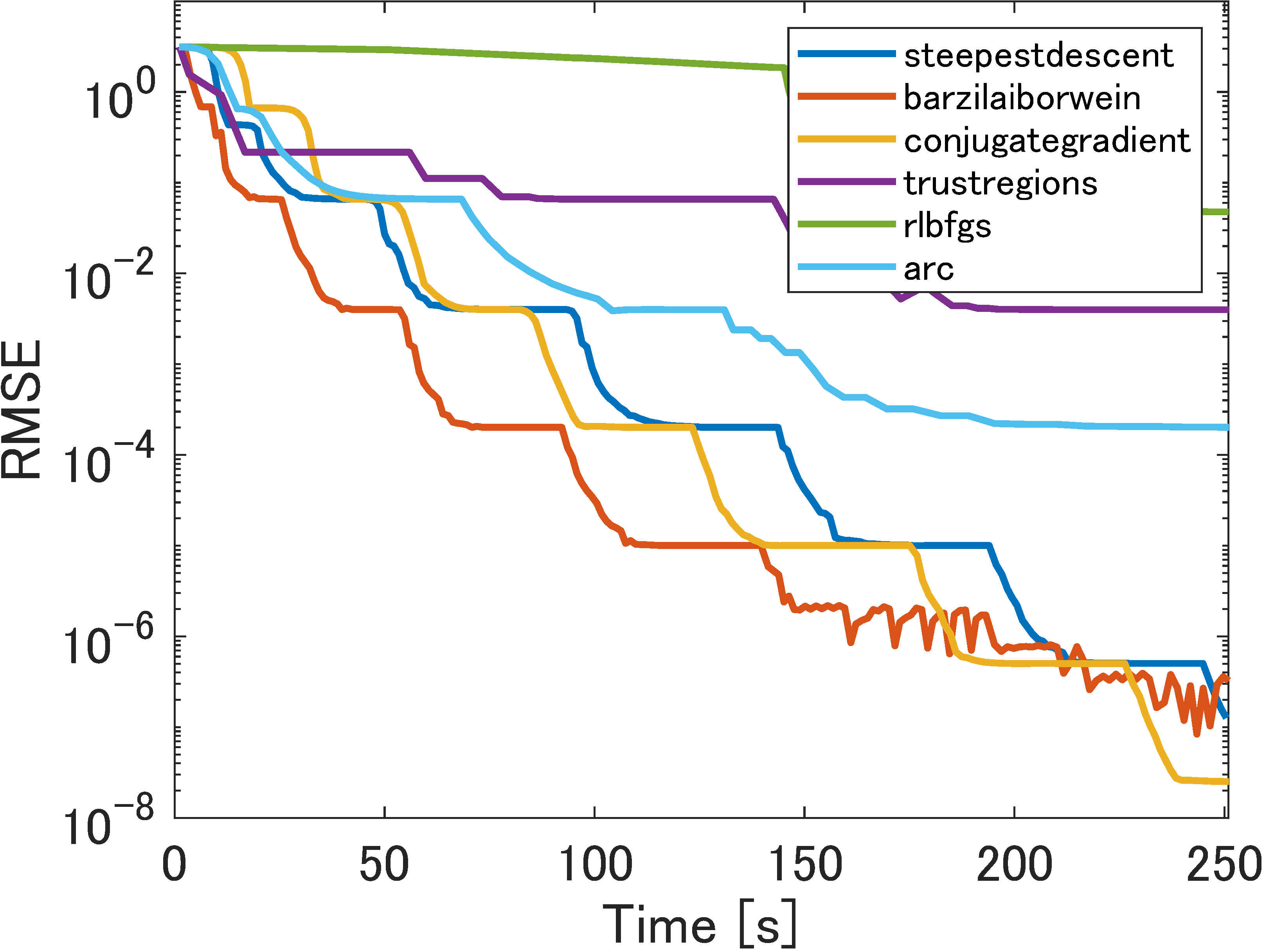}
		\caption{$ \tilde{f}_{3} $}
	\end{subfigure}
	\begin{subfigure}{0.24\textwidth}
		\includegraphics[width=\textwidth]{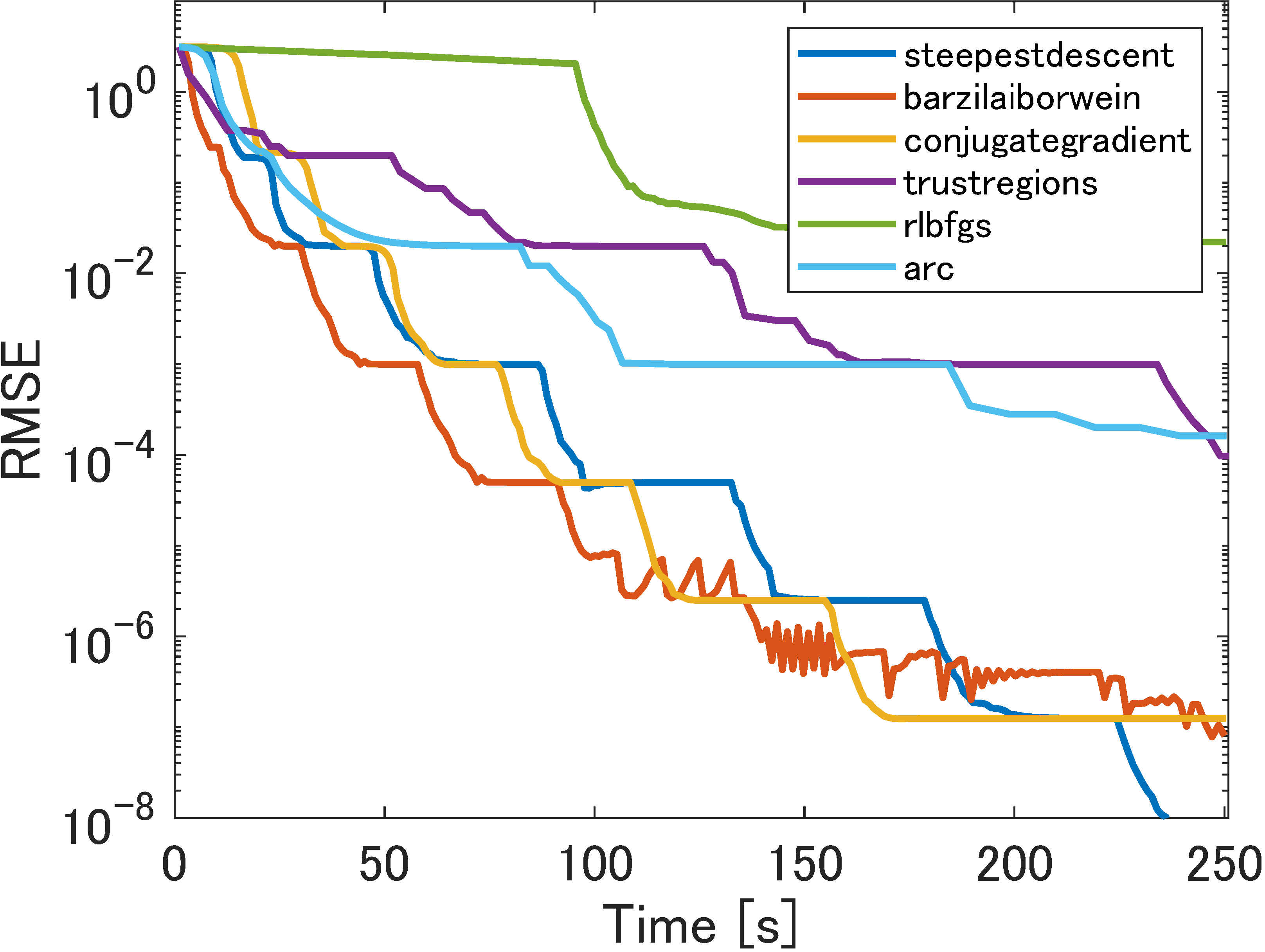}
		\caption{$ \tilde{f}_{4} $}
	\end{subfigure}
	\begin{subfigure}{0.24\textwidth}
		\includegraphics[width=\textwidth]{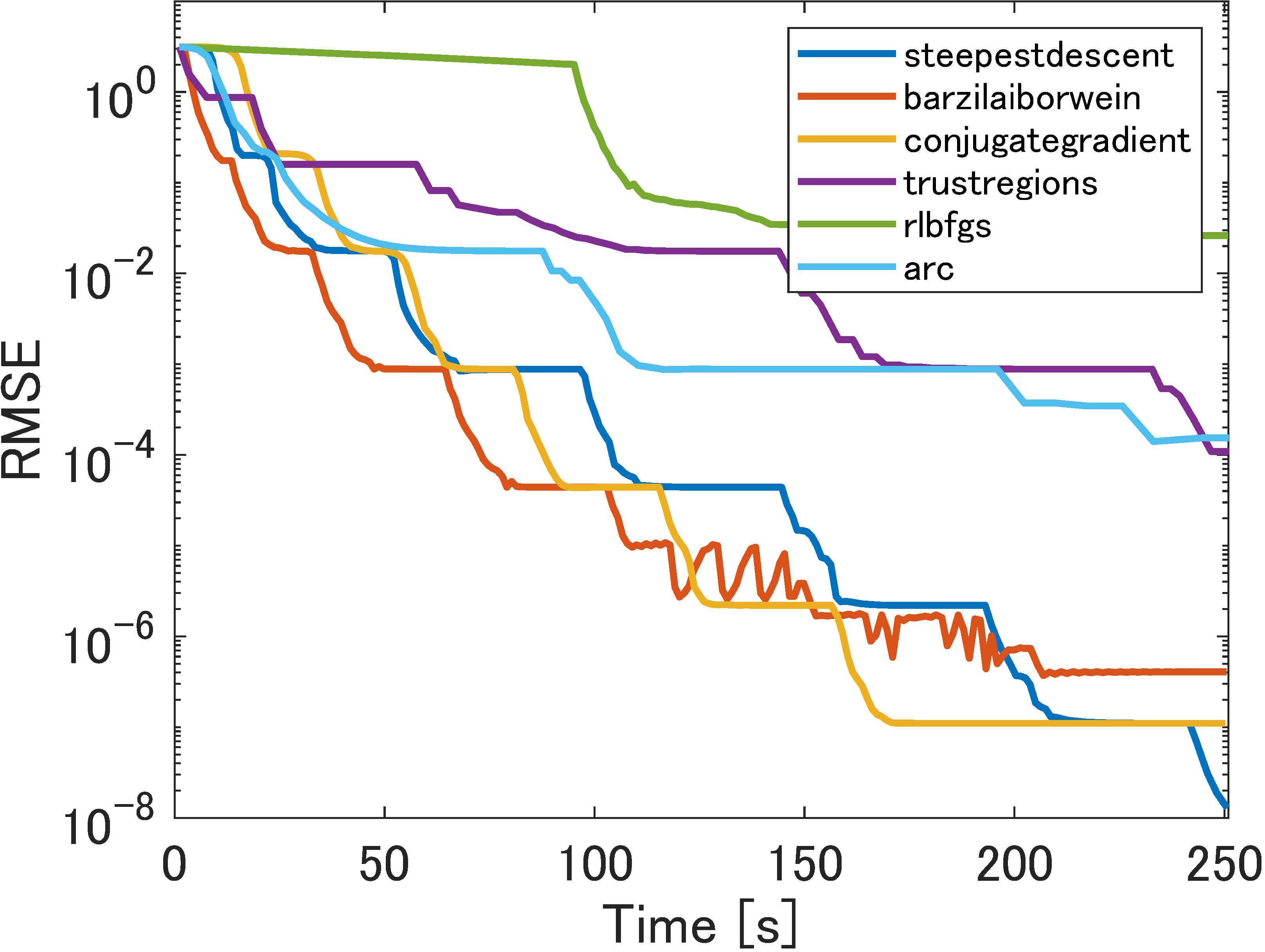}
		\caption{$ \tilde{f}_{5} $}
	\end{subfigure}
	\caption{Low-rank matrix completion with outliers for two rank-10 $5000 \times 5000$ matrices by using different smoothing functions in Table \ref{table:smlist}.
		(a)--(j) corresponds to one matrix with outliers created by using $\mu_{N}=\sigma_{N}=0.1$, while (k)--(t) corresponds to the other with outliers created by using $\mu_{N}=\sigma_{N}=1$.
		(a)--(e) and (k)--(o) comprise the running iteration comparison; (f)--(j) and (p)--(t) comprise the  time comparison.}
	\label{fig:outliers5000}
\end{figure}

\subsubsection{Perfect low-rank matrix completion}

As in \cite{cambier2016robust}, we first tested all the methods on a simple perfect matrix $M$ (without any outliers) of size $5000 \times 5000$ and rank 10. The results are shown in Figure \ref{fig:perfect_5000}. We can see that the choice of smoothing function does not have much effect on numerical performance. In terms of the number of iterations ((a)--(e)), our Algorithm \ref{alg:sm_NROP222} inherits the convergence of its sub-algorithm at least Q-superlinearly when trust regions or ARC are used. But the single iteration cost of trust regions and ARC is high; they are not efficient in terms of time. Specifically, the conjugate-gradient method employed in \cite{cambier2016robust} stagnates at lower precision. Overall, Barzilai-Borwein performed best in terms of time and accuracy.

\subsubsection{Low-rank matrix completion with outliers}

Given a $500 \times 500$ matrix for which we observed the entries uniformly at random with an oversampling $\rho$ of 5, we perturbed $5 \%$ of the observed entries by adding noise to them in order to create outliers. The added item was a random variable defined as $\mathcal{O}=\mathcal{S}_{\pm 1} \cdot \mathcal{N}(\mu_{N}, \sigma_{N}^{2})$ where $\mathcal{S}_{\pm 1}$ is a random variable with equal probability of being equal to $+1$ or $-1$, while $\mathcal{N}(\mu_{N}, \sigma_{N}^{2})$ is a Gaussian random variable of mean $\mu_{N}$ and variance $\sigma_{N}^{2}$.

Figure \ref{fig:outliers} reports the results of two $500 \times 500$ instances with outliers generated using $\mu_{N}=\sigma_{N}=0.1$ and $\mu_{N}=\sigma_{N}=1$. Again, we can see that the choice of smoothing function does not have much effect. In most cases, BFGS and trust regions were better than the other methods in terms of number of iterations, and BFGS was the fastest. Furthermore, the conjugate-gradient method employed in \cite{cambier2016robust} still stagnated at solutions with lower precision, approximately $ 10^{-6} $, while steepest descent, BFGS, and trust regions always obtained solutions with at least $ 10^{-8} $ precision.

Next, we ran the same experiment on larger $5000 \times 5000$ matrices, with $5 \%$ outliers. Figure \ref{fig:outliers5000} illustrates the results of these experiments, with $\mu_{N}=\sigma_{N}=0.1$ and $\mu_{N}=\sigma_{N}=1$. In most cases, trust regions still outperformed the other methods in terms of number of iterations, while BFGS performed poorly. Barzilai-Borwein and the conjugate-gradient method were almost as good in terms of time.

\section{Concluding remarks}

We examined the problem of finding a CP factorization of a given completely positive matrix and treated it as a nonsmooth Riemannian optimization problem. 
To this end, we studied a general framework of Riemannian smoothing for Riemannian optimization. 
The numerical experiments clarified that our method can compete with other efficient CP factorization methods, in particular on large-scale matrices.

Let us we summarize the relation of our approach to the existing CP factorization methods.
Groetzner and D{\"u}r \cite{groetzner2020factorization} and Chen et al. \cite{chen2020difference} proposed different methods to solve (\ref{FeasCP}).
Bo{\c{t}} and Nguyen \cite{boct2021factorization} tried to solve another model (\ref{bot}).
However, the methods they used do not belong to the Riemannian optimization techniques, but are rather Euclidean ones, since they treated the set $ \mathcal{O}(r) := \{ X \in \mathbb{R}^{r \times r} : X^{\top}X=I\}$ as a usual constraint in Euclidean space.
By comparison, we recognize the existence of manifolds, namely, the Stiefel manifold $ \mathcal{M}= \mathcal{O}(r) $, and use optimization techniques specific to them.
This change in perspective suggests the possibility of using the rich variety of Riemannian optimization techniques.
As the experiments in Section \ref{sec:ne} show, our Riemannian approach is faster and more reliable than the Euclidean methods.

In the future, we plan to extend Algorithm \ref{alg:sm_NROP222} to the case of general manifolds and, particularly, to quotient manifolds. 
This application is believed to be possible, although effort should be put into establishing analogous convergence results. 
In fact, convergence has been verified in a built-in example in Manopt 7.0 \cite{manopt}: \texttt{robust\_pca.m} computes a robust version of PCA on data and optimizes a nonsmooth function over a Grassmann manifold. 
The nonsmooth term consists of the $ l_2 $ norm, which is not squared, for robustness. 
In \texttt{robust\_pca.m}, Riemannian smoothing with a pseudo-Huber loss function is used in place of the $ l_2 $ norm.

As in the other numerical methods, there is no guarantee that Algorithm \ref{alg:sm_NROP222} will find a CP factorization for every $A \in \mathcal{C} \mathcal{P}^{n}$. 
It follows from Proposition \ref{prop:cp} that $A \in \mathcal{C} \mathcal{P}_{n}$ if and only if the global minimum of (\ref{OptCP}), say $t$, is such that $t \leqslant 0$. 
Since our methods only converge to a stationary point, Algorithm \ref{alg:sm_NROP222} provides us with a local minimizer at best. We are looking forward to finding a global minimizer of (\ref{OptCP}) in our future work.

\subsection*{Acknowledgments}

This work was supported by the Research Institute for Mathematical Sciences, an International Joint Usage/Research Center, at Kyoto University, JSPS KAKENHI Grant, number (B)19H02373, and JST SPRING Grant, number JPMJSP2124.  
The authors would like to sincerely thank the anonymous reviewers and the coordinating editor for their thoughtful and valuable comments which have significantly improved the paper.

\subsection*{Data availability}
The data that support the findings of this study are available from the corresponding author upon request.

\subsection*{Conflict of interest}
All authors declare that they have no conflicts of interest.

\end{document}